\documentclass[12pt, english, reqno]{amsart}
\usepackage{amsmath,amssymb,enumerate}
\usepackage{caption}

\usepackage[T1]{fontenc}
\usepackage[utf8]{inputenc}
\usepackage[all]{xy}

\usepackage{babel}
\usepackage{amstext}
\usepackage{amsmath}
\usepackage{amsfonts}
\usepackage{latexsym}
\usepackage{ifthen}
\usepackage{multirow}
\usepackage{multicol}
\usepackage{graphicx}
\usepackage{tikz-cd}
\usepackage{float}

\usepackage[colorlinks,linkcolor=red,anchorcolor=green,citecolor=blue]{hyperref} %color of hyperlinks
\hypersetup{linktocpage = true} %color of hyperlinks of TOC

\xyoption{all}
\usepackage{array}
\usepackage{mathtools} %an extension package to amsmath + loading amsmath
\usepackage{ytableau} %young tableau
\usepackage{longtable} %long table
\newcolumntype{M}[1]{>{\centering\arraybackslash}m{#1}} %define dimension for long stable

\usepackage{setspace} %set­ting the spac­ing be­tween l

\usepackage{lscape}
\usepackage{rotating}

\usepackage{dynkin-diagrams}

\numberwithin{equation}{section}

\newcommand\LG{\rm LG}

\newcommand{\Bs}{{\rm Bs}}
\newcommand{\sJ}{{\mathcal J}}
\newcommand{\sK}{{\mathcal K}}

\newtheorem{lemma1}{}[section]

\newenvironment{lemma}{\begin{lemma1}{\bf Lemma.}}{\end{lemma1}}
\newenvironment{example}{\begin{lemma1}{\bf Example.}\rm}{\end{lemma1}}

\newenvironment{theorem}{\begin{lemma1}{\bf Theorem.}}{\end{lemma1}}
\newenvironment{proposition}{\begin{lemma1}{\bf Proposition.}}{\end{lemma1}}
\newenvironment{corollary}{\begin{lemma1}{\bf Corollary.}}{\end{lemma1}}
\newenvironment{remark}{\begin{lemma1}{\bf Remark.}\rm}{\end{lemma1}}

\newenvironment{definition}{\begin{lemma1}{\bf Definition.}}{\end{lemma1}}

\newenvironment{conjecture}{\begin{lemma1}{\bf Conjecture.}}{\end{lemma1}}

\newenvironment{problem}{\begin{lemma1}{\bf Problem.}}{\end{lemma1}}

\newenvironment{remark*}{{\bf Remark.}}{}
\newenvironment{example*}{{\bf Example.}}{}
\newenvironment{assumption*}{{\bf Assumption.}}{}
\newenvironment{claim*}{{\bf Claim.}}{}

\newcommand{\R}{\ensuremath{\mathbb{R}}}
\newcommand{\Q}{\ensuremath{\mathbb{Q}}}
\newcommand{\Z}{\ensuremath{\mathbb{Z}}}
\newcommand{\N}{\ensuremath{\mathbb{N}}}
\newcommand{\PP}{\ensuremath{\mathbb{P}}}

\usepackage{eurosym}

\makeatletter
\ifnum\@ptsize=0  \fi \ifnum\@ptsize=2  \fi 

\newcommand\sH{{\mathcal H}}

\newcommand\sN{{\mathcal N}}
\newcommand\sO{{\mathcal O}}

\newcommand\sV{{\mathcal V}}
\newcommand\sX{{\mathcal X}}
\newcommand\sY{{\mathcal Y}}
\newcommand\sC{{\mathcal C}}

\newcommand\bQ{{\mathbb Q}}

\newcommand\bB{{\mathbb B}}

\newcommand\Proj{{\rm Proj}}

\newcommand\Gr{{\rm Gr}}
\newcommand\defect{{\rm def}}
\newcommand\Eff{{\rm Eff}}
\newcommand\Bigcone{{\rm Big}}
\newcommand\Mov{{\rm Mov}}
\newcommand\Nef{{\rm Nef}}
\newcommand\Amp{{\rm Amp}}

\newcommand\Aut{{\rm Aut}}
\newcommand\PGL{\rm PGL}

\DeclareMathOperator*{\Sym}{Sym}
\DeclareMathOperator*{\pic}{Pic}

\DeclareMathOperator*{\codeg}{codeg}

\newcommand{\Chow}[1]{\ensuremath{\mbox{\rm Chow}(#1)}}

\DeclareMathOperator*{\supp}{Supp}

\DeclareMathOperator*{\Rat}{RatCurves^n}

\usepackage{enumitem}
\setlist[enumerate]{leftmargin=*}
\title{normalized tangent bundle, varieties with small codegree and pseudoeffective threshold} 
\date{\today}

\author{Baohua Fu}

\author{Jie Liu}

\address{Baohua Fu, HLM and MCM, Academy of Mathematics and System Science, Chinese Academy of Sciences, Beijing 100190, China and School of Mathematical Sciences, University of Chinese Academy of Sciences, Beijing 100049, China}
\email{bhfu@math.ac.cn}

\address{Jie Liu, Institute of Mathematics, Academy of Mathematics and Systems Science, Chinese Academy of Sciences, Beijing, 100190, China}
\email{jliu@amss.ac.cn}

\subjclass[2010]{14J45,14M17,14N05,51N35}
\keywords{Fano vareities, rational homogeneous spaces, dual variety, minimal rational curves, stratified Mukai flops}

\begin{document}

\begin{abstract}  
We propose a conjectural list of Fano manifolds of Picard number $1$ with pseudoeffective normalized tangent bundles, which we prove in various situations by  relating it to the complete divisibility conjecture of Russo and Zak on varieties with small codegree. Furthermore, the pseudoeffective thresholds and hence the pseudoeffective cones of the projectivized tangent bundles of rational homogeneous spaces of Picard number $1$ are explicitly determined by studying the total dual VMRT and the geometry of stratified Mukai flops. As a by-product, we obtain sharp vanishing theorems on the global twisted symmetric holomorphic vector fields on rational homogeneous spaces of Picard number $1$.
\end{abstract}

\newpage

\maketitle

\setcounter{tocdepth}{2}
\tableofcontents

\vspace{-0.2cm}

\section{Introduction}

\subsection{Positivity of normalized tangent bundles}

Let $X$ be an $n$-dimensional projective manifold, and let $C\subset X$ be an irreducible projective curve. The (semi-)stability of the restriction $T_X|_C$ is very closely related to the global geometry of $X$.  For example, a famous result of Mehta and Ramanathan \cite{MehtaRamanathan1982} says that the restriction $T_X|_C$ of $T_X$ to a general complete intersection curve  $C$ of \emph{sufficiently ample} divisors is again (semi-)stable provided that $T_X$ itself is (semi-)stable with the respective polarization. However, apart from very special situations, the variety $X$ usually contains many dominating families of irreducible curves to which the restrictions of $T_X$ are not (semi-)stable. Using the language of positivity of $\Q$-twisted vector bundles, the semi-stability of $T_X|_C$ is equivalent to the nefness of the restriction of the \emph{normalized tangent bundle}  $T_X\hspace{-0.8ex}<\hspace{-0.8ex}-\frac{1}{n}c_1(X)\hspace{-0.8ex}>$ of $X$ to $C$ (see \cite[Proposition 6.4.11]{Lazarsfeld2004a}). 
Thus, our expectation above can be rephrased by saying that $X$ should be very special if its normalized tangent bundle is positive in some algebraic sense. 

Let $\pi:\mathbb{P}(T_X)\rightarrow X$ be the  projectivized tangent bundle (in the Grothendieck sense) with tautological divisor $\Lambda$.  The normalized tangent bundle $T_X\hspace{-0.8ex}<\hspace{-0.8ex}-\frac{1}{n}c_1(X)\hspace{-0.8ex}>$  is said {\em pseudoeffective} (resp. {\em ample, big, nef}) if so is the class $\Lambda  - \frac{1}{n} \pi^* (c_1(X))$.  The normalized tangent bundle is said {\em almost nef} if all irreducible curves $C \subset X$, to which  the restriction of $T_X\hspace{-0.8ex}<\hspace{-0.8ex}-\frac{1}{n}c_1(X)\hspace{-0.8ex}>$ is not nef,  are contained in a countable union of proper subvarieties of $X$.  Since the normalized tangent bundle of a curve is numerically trivial, we will only consider varieties of dimension at least 2.

The positivity of normalized tangent bundles has already been studied  in various contexts. In particular, we have the following theorem, which can be easily derived from the works of Jahnke-Radloff \cite[Theorem 0.1]{JahnkeRadloff2013}, H\"oring-Peternell \cite[Theorem 1.9]{HoeringPeternell2019} and Liu-Ou-Yang \cite[Theorem 1.6]{LiuOuYang2020}. It can be viewed as a strong evidence to our expected picture above.

\begin{theorem}
	\label{Thm:Almost-Nef-normalization}
	Let $X$ be a projective manifold of dimension at least 2. Then the normalized tangent bundle of $X$ is almost nef if and only if $X$ is isomorphic to a finite \'etale quotient of an Abelian variety.
\end{theorem}

The motivation of this paper is to study a weaker positivity:  the pseudoeffectivity of normalized tangent bundles. This problem has already been studied by H\"oring-Peternell in \cite{HoeringPeternell2019} for klt projective variety with numerically trivial canonical class. Moreover, Nakayama has studied this problem in \cite{Nakayama2004} for semi-stable vector bundles of rank $2$ over projective manifolds of arbitrary dimension and he obtained a complete classification for such vector bundles (see \cite[IV, Theorem 4.8]{Nakayama2004} for a precise statement). In particular, Nakayama's result provides a satisfactory answer to our problem above for projective surfaces. For instance, it turns out that a del Pezzo surface $S$ has pseudoeffective normalized tangent bundle if and only if $S$ is isomorphic to a quadric surface (see Theorem \ref{Thm:normalized-Surface}).
Note that the product of two projective manifolds with pseudoeffective normalized tangent bundle has again pseudoeffective normalized tangent bundle.  To exclude the product cases, 
we will focus on the case where $X$ is a  Fano manifold of Picard number $1$ with dimension at least $3$ in this paper. Note that in this situation the pseudoeffectivity of the normalized tangent bundle of $X$ implies that the tangent bundle of $X$ is big and it is expected that the bigness of the tangent bundle is already a rather restrictive property (see \cite{HoeringLiuShao2020}).   We expect the following classification: 

\begin{conjecture}
	\label{Conj:normalized-Tangent}
	Let $X$ be a Fano manifold of Picard number $1$  with  dimension at least $3$. Then the normalized tangent bundle of $X$ is pseudoeffective if and only if $X$ is one of the following varieties:
	\begin{enumerate}
		\item a smooth quadric hypersurface;
		
		\item the Grassmann variety $\Gr(n,2n)$;
		
		\item the Spinor variety $\mathbb{S}_{2n}$;
		
		\item the Lagrangian Grassmann variety $\LG(n,2n)$;
		
		\item the $27$-dimensional $E_7$-variety $E_7/P_7$.
	\end{enumerate}
\end{conjecture}

Note that the normalized tangent bundles of the varieties in the list are already shown to be pseudoeffective but not big by \cite[Corollary 1.4]{Shao2020} (See Proposition \ref{Prop:Fano-contact}  for another proof). We will prove a more general result in Theorem \ref{t.RationalHomSpace}. On the other hand, if we use the pseudoeffective threshold (with respect to an ample line bundle $A$) introduced in  \cite{Shao2020} which is defined as 
\[
\alpha(X,A):= {\rm sup} \{\alpha \in \mathbb{R}\, |\, \Lambda - \alpha \pi^*A\ \text{is effective} \},
\] 
then we can reformulate Conjecture \ref{Conj:normalized-Tangent} as follows:
\begin{conjecture} \label{Conj2}
	Let $X$ be a Fano manifold of Picard number $1$ with dimension at least $3$. Then 
	\[
	\alpha(X,-K_X)\leq \frac{1}{\dim(X)}
	\]
	with equality if and only if $X$ is one of the varieties in Conjecture \ref{Conj:normalized-Tangent}.
\end{conjecture}

\begin{remark}
By Theorem \ref{Thm:Almost-Nef-normalization}, the normalized tangent bundle of a projective manifold can not be nef and big (see also \cite[IV, Corollary 4.7]{Nakayama2004}).  On the other hand, Conjecture \ref{Conj:normalized-Tangent} implies that there does not exist examples of Fano manifolds of Picard number $1$ with big normalized tangent bundle, and we suspect the existence of such examples even for Fano manifolds of higher Picard number. Here we recall that if the tangent bundle $T_X$ of a Fano manifold $X$ is semi-stable with respect to some ample line bundle $A$, then the normalized tangent bundle of $X$ can not be big (see Lemma \ref{l.semistable}).
\end{remark}

A powerful tool to study Fano manifolds is the VMRT (abbreviation for  variety of minimal rational tangents) theory developed by Hwang and Mok (cf. \cite{Hwang2001}). Fix a dominating family  of minimal rational curves $\mathcal{K}$ on a Fano manifold $X$ and a general point $x \in X$.   
The tangent directions  at $x$ of members in $\mathcal{K}$ passing through $x$ form a  projective subvariety $\mathcal{C}_x$ in $\mathbb{P}(\Omega_{X,x})$.   The projective geometry of $\mathcal{C}_x$ encodes many global properties of $X$.  For example, we can recover irreducible Hermitian symmetric spaces  (IHSS for short) from its VMRT by the following result of Mok:
\begin{theorem} \cite[Main Theorem]{Mok2008a} \label{t.IHSSMok}
Let $G/P$ be an irreducible Hermitian symmetric space and let $X$ be a Fano manifold of Picard number $1$.  Assume that the VMRT of $X$ at a general point is projectively equivalent to that of $G/P$.  Then $X$ is isomorphic to $G/P$. 
\end{theorem}

As all the varieties listed in Conjecture \ref{Conj:normalized-Tangent} are IHSS,  we may try to determine first the VMRT of $X$ in Conjecture \ref{Conj:normalized-Tangent} and then apply Theorem \ref{t.IHSSMok}.  This is the approach that we will use in this paper.

It is interesting to remark that among IHSS, only the following varieties do not appear in Conjecture \ref{Conj:normalized-Tangent}: $\Gr(a, a+b) $ with $a \neq b$, $\mathbb{S}_{2n+1}$ and $E_6/P_1$.  These varieties are exactly those among IHSS which appear in stratified Mukai flops (cf. Proposition \ref{p.stratiMukai}).   As it will become clearer, there exists a delicate relationship between the pseudoeffective threshold and the birational geometry.

\renewcommand*{\arraystretch}{1.6}
\begin{longtable}{|M{2.2cm}|M{2.4cm}|M{2cm}|M{1.5cm}|M{3cm}|M{1.1cm}|M{1.1cm}|}
	\caption{IHSS and their VMRTs}
	\label{Table:IHSSVMRT}
	\\
	\hline

	   IHSS $G/P$               &   $\bQ^n$
	&  $\Gr(a,a+b)$             &   $\mathbb{S}_n$
	&  $\LG(n,2n)$              &   $E_6/P_1$
	&  $E_7/P_7$
	\\
	\hline
	
	   VMRT $\sC_o$                   &   $\bQ^{n-2}$
    &  $\PP^{a-1}\times \PP^{b-1}$    &   $\Gr(2,n)$
    &  $\PP^{n-1}$                    &   $\mathbb{S}_5$
	&  $E_6/P_1$
	\\
	\hline
	
	   embedding                      &   Hyperquadric
	&  Segre                          &   Pl\"ucker
	&  second Veronese                &   Spinor
	&  Severi
	\\
	\hline
\end{longtable}

\subsection{Varieties with small codegree}

Recall that the \emph{codegree} $\codeg(Z)$ of a projective variety $Z\subset \PP^N$ is defined as the degree of its dual variety $\check{Z}\subset \check{\PP}^N$ (see Definition \ref{Def:Projectively-Dual}).  Varieties with small degree have been thoroughly studied while very little is known for varieties with small codegree.  Segre proved in \cite{Segre1951} that for an irreducible and linearly non-degenerate projective variety $Z\subsetneq \PP^N$, if its dual variety $\check{Z}\subset \check{\PP}^N$is a hypersurface with non-vanishing hessian, then we have the following Segre inequality
\begin{equation}
	\label{Eq:Segre-Zak-Ineq}
	\codeg(Z):=\deg(\check{Z})\geq \frac{2(N+1)}{\dim(Z)+2}.
\end{equation}

The above inequality is sharp and in fact the following  \emph{complete divisibility conjecture} due to Russo and Zak predicts the boundary varieties:

\begin{conjecture}
	\label{Conj:Zak-Russo-Conjecture}
	\cite[Question 5.3.11]{Russo2003}
	\cite[Conjecture 4.15]{Zak2004}
	Let $Z\subsetneq \PP^N$ be an irreducible and linearly non-degenerate projective variety. If the dual variety $\check{Z}\subset \check{\PP}^N$ is a hypersurface with non-vanishing hessian such that
	\begin{equation}
		\label{Eq:Segre-Zak}
		%\tag{$\clubsuit$}
		\codeg(Z):=\deg(\check{Z})=\frac{2(N+1)}{\dim(Z)+2}.
	\end{equation}
	Then $Z$ is isomorphic to one of the following varieties:
	\begin{enumerate}
		\item a smooth quadric hypersurface  ( $\codeg(Z)=2$);
		
		\item the Segre variety $\PP^{n-1}\times \PP^{n-1}\subset \PP^{(n-1)(n+1)}$  ($\codeg(Z)=n$);
		
		\item the Grassmann variety $\Gr(2,2n)\subset \PP^{n(2n-1)-1}$  ( $\codeg(Z)=n$);
		
		\item the Veronese variety $\nu_2(\PP^{n-1})\subset \PP^{\frac{(n-1)(n+2)}{2}}$  ($\codeg(Z)=n$);
		
		\item the $16$-dimensional Cayley plane $E_6/P_1\subset \PP^{26}$  ($\codeg(Z)=3$).
	\end{enumerate}
\end{conjecture}

Conjecture \ref{Conj:Zak-Russo-Conjecture} is still widely open for $\codeg(Z) \geq 4$.   The case $\codeg(Z)=2$ is easy as $Z$ must be a hyperquadric.  When  $\codeg(Z)=3$, then we have $\dim Z = \frac{2N-4}{3}$, which is the bound for Severi varieties.  Thanks to Zak's  classification of smooth varieties of codegree $3$ (see \cite[Theorem 5.2]{Zak1993}), it turns out in this case, $Z$ is one of the following Severi varieties: 
$$\nu_2(\PP^2) \subset \PP^5, \qquad \PP^2 \times \PP^2 \subset \PP^8,  \qquad \Gr(2,6) \subset \PP^{14}, \qquad E_6/P_1\subset \PP^{26}.$$
As a corollary, Conjecture \ref{Conj:Zak-Russo-Conjecture} is confirmed in the following two cases:
\begin{enumerate}
	\item $\dim(Z)>\frac{2N-4}{3}$;
	
	\item $Z$ is smooth and $\dim(Z)>\frac{N-3}{2}$.
\end{enumerate}
On the other hand, initiated from 1950s, there have been many efforts trying to classify nonsingular curves and surfaces with small codegree, which proves Conjecture \ref{Conj:Zak-Russo-Conjecture} up to dimension $2$. More precisely, we have 
\begin{proposition}
	\label{Prop:SmallCodegreeDim2}
	\cite[Proposition 3.1 and 3.2]{Zak2004} \cite[Theorem 2.1]{TurriniVerderio1993}
	 If $Z\subsetneq \PP^N$ is a smooth projective variety of dimension at most $2$ satisfying \eqref{Eq:Segre-Zak}, then  $Z$ is either a conic curve,  a quadric surface or 
	  the Veronese surface $\nu_2(\PP^2)\subset \PP^5$.
\end{proposition}

	There are very few papers devoted to threefold cases, see for example \cite{LanteriTurrini1987}.  We will  confirm Conjecture \ref{Conj:Zak-Russo-Conjecture} for smooth projective threefolds.  More precisely we shall show:
\begin{proposition}
	\label{Thm:Threefold-Small-Codegree}
	Let $Z\subsetneq \PP^N$ be a linearly non-degenerate smooth projective threefold of degree $d$ and codegree $d^*$. Then one of the following statements holds.
	\begin{enumerate}
		\item $p_g(S)\not=0$ and $d^*\geq 2N$, where $S$ is a general hyperplane section of $Z$.
		
		\item  $2d^*\geq d$ with equality if and only if $Z$ is projectively equivalent to either the Veronese variety $\nu_2(\PP^3)\subset \PP^9$ or its isomorphic projection in $\PP^8$.
	\end{enumerate}
	In particular, Conjecture \ref{Conj:Zak-Russo-Conjecture} holds for smooth projective threefolds. More precisely, if $Z\subsetneq \PP^N$ is a linearly non-degenerate smooth projective threefold satisfying the equality \eqref{Eq:Segre-Zak}, then  $Z$ is either  a quadric threefold in $\PP^4$ or  the Veronese embedding $\nu_2(\PP^3)\subset \PP^9$.
\end{proposition}

The relation between Conjecture \ref{Conj:normalized-Tangent} and Conjecture \ref{Conj:Zak-Russo-Conjecture} can be easily seen from Table \ref{Table:IHSSVMRT} above: the varieties listed in Conjecture \ref{Conj:Zak-Russo-Conjecture} are nothing else but the VMRTs of the varieties listed in Conjecture \ref{Conj:normalized-Tangent}. Indeed, if  we assume that the VMRT of $X$ at a general point is not dual defective,  then the pseudoeffectivity of the normalized tangent bundle can be interpreted as information on the cohomological class of the total dual VMRT (cf. \cite{HwangRamanan2004,HoeringLiuShao2020}). This allows us to relate Conjecture \ref{Conj:normalized-Tangent} to Conjecture \ref{Conj:Zak-Russo-Conjecture}. In particular, combining this with the known results for Conjecture \ref{Conj:Zak-Russo-Conjecture} yields the following first main result of this paper. 

\begin{theorem}
	\label{Thm:normalized-I}
	Let $X$ be an $n$-dimensional Fano manifold of Picard number $1$ with $n\geq 3$.  Assume that the VMRT $\sC_x\subset \PP(\Omega_{X,x})$ at a general point $x\in X$ is not dual defective.
	
	\begin{enumerate}
		\item If we assume in addition that the VMRT is irreducible and linearly non-degenerate such that its dual variety has non-vanishing hessian, then 
		      \begin{enumerate}  
		      	\item $\alpha(X,-K_X)\leq \frac{1}{\dim(X)}$;
		      		
		      	\item Conjecture \ref{Conj:Zak-Russo-Conjecture} implies Conjecture \ref{Conj:normalized-Tangent} and hence Conjecture \ref{Conj2};
		      \end{enumerate}
	    \item Conjecture \ref{Conj:normalized-Tangent} and Conjecture \ref{Conj2} hold if one of the following holds.
	    
	    \begin{enumerate}
	    	\item $\dim(\sC_x)>\frac{2n-6}{3}$, or
	    	
	    	\item $\sC_x$ is smooth and $\dim(\sC_x)>\max\left\{\frac{n-4}{2},0\right\}$, or
	    	
	    	\item $\sC_x$ is irreducible, smooth, linearly non-degenerate and $\dim(\sC_x)\leq 3$.
	    \end{enumerate}
	\end{enumerate}
\end{theorem} 

Our statement is actually a bit more stronger: if in Theorem \ref{Thm:normalized-I} the normalized tangent bundle of $X$ is assumed to be pseudoeffective, then the VMRT $\sC_x\subset \PP(\Omega_{X,x})$ satisfies the reverse Segre inequality \eqref{Eq:Segre-Zak-Ineq} (see Proposition \ref{Prop:Criterion-Pseff-Normalized}). Typically a projective variety is dual defective only in very special cases and the VMRTs of a large class of Fano manifolds are smooth and irreducible. Thus the assumption on smoothness, irreducibility and non-defectiveness is not very restrictive. However, the assumption on the non-degeneracy seems to be a strong restriction as many known examples of Fano manifolds have degenerate VMRTs.

\begin{corollary}
	\label{c.less-than-5}
	Let $X$ be an $n$-dimensional Fano manifold of Picard number $1$ such that $3\leq n\leq 5$. Assume that the VMRT $\sC_x\subset \PP(\Omega_{X,x})$ at a general point is smooth and non-linear. Then the normalized tangent bundle of $X$ is pseudoeffective if and only if $X$ is a smooth quadric hypersurface in $\PP^{n+1}$ ($3\leq n\leq 5$).
\end{corollary}

\begin{remark}
	Let $X$ be a Fano manifold of Picard number $1$. To our best knowledge, the known examples of $X$ with singular VMRTs have dimension at least $6$ (cf. \cite[Theorem 1.3]{HwangKim2015}) and the known examples of $X$, not isomorphic to projective spaces, with linear VMRTs also have dimension at least $6$ (cf. \cite[Proposition A.8]{MunozOcchettaSolaConde2014})
\end{remark}

\begin{corollary}
	\label{c.less-than-11}
	Let $X$ be an $n$-dimensional Fano manifold of Picard number $1$ such that $3\leq n\leq 11$. Assume that the VMRT $\sC_x$ at a general point $x\in X$ is irreducible, smooth, linearly non-degenerate and not dual defective. Then the normalized tangent bundle of $X$ is pseudoeffective if and only if $X$ is one of the following varieties:
	\begin{enumerate}
		\item a smooth quadric hypersurface in $\PP^{n+1}$ ($3\leq n\leq 11$);
		
		\item the Lagrangian Grassmann varieties $\LG(3,6)$ and $\LG(4,8)$;
		
		\item  the Grassmann variety $\Gr(3,6)$.
	\end{enumerate}
\end{corollary}

\subsection{Rational homogeneous spaces}

As mentioned in the previous subsection, the pseudoeffectivity of the normalized tangent bundle implies the bigness of the tangent bundle and up to our knowledge, there are very few known examples of Fano manifolds of Picard number $1$ with big tangent bundle. Apart from rational homogeneous spaces, only two examples are known, namely  the del Pezzo threefold $V_5$ of degree $5$ \cite[Theorem 1.5]{HoeringLiuShao2020} and the horospherical $G_2$-variety $\mathbb{X}$ \cite[Theorem 2.3]{PasquierPerrin2010} (see Remark \ref{r.Examples-Big-tangent} and \cite{Liu2022}). Thus a natural question is to verify Conjecture \ref{Conj:normalized-Tangent} for those examples. This is more or less equivalent to determine the pseudoeffective cone of the projectivized tangent bundle, or equivalently to determine the invariant $\alpha(X,-K_X)$, and it fits into the following general problem in the study of positivity of vector bundles.

\begin{problem}
	\label{Prom:Nakayama}
	\cite[IV.4, Problem]{Nakayama2004}
	Let $E$ be a vector bundle over a projective manifold $X$ and let $\Lambda$ be the tautological class of the projectivized bundle $\pi:\PP(E)\rightarrow X$. Describe the set
	\[
		V(X,E):=\{ D\in N^1(X)\,|\, \Lambda+\pi^*D\ \text{is pseudoeffective}\}.
	\]
\end{problem}

The second part of this paper is devoted to study Problem \ref{Prom:Nakayama} for rational homogeneous spaces $X=G/P$ of Picard number $1$ and for $E=T_{X}$. This is equivalent to  determine whether the following cohomological group
\[
H^0(X,({\rm Sym}^r T_X)\otimes \sO_X(-dH))
\]
vanishes or not, where $H$ is the ample generator of ${\rm Pic}(X)$. In general, it is quite difficult to compute these cohomological groups due to the lack of tools. However, recently it is observed in \cite{HoeringLiuShao2020} that the problem can be translated into the calculation of the cohomological class of the total dual VMRT if the VMRT is not dual defective. By combining this with the geometry of  stratified Mukai flops, we will completely settle Problem \ref{Prom:Nakayama} for rational homogeneous spaces of Picard number 1 with $E$ being the tangent bundle, which reads as follows:

\begin{theorem}
   \label{t.RationalHomSpace}
   Let $G/P$ be a rational homogeneous space of Picard number 1 with dimension at least $2$.  Let $\Lambda$ be the tautological divisor on $\PP(T_{G/P})$ and $\pi: \PP(T_{G/P}) \to G/P$ the natural projection. Denote by $H$ the ample generator of $\pic(G/P)$. Then  there exist two integers $a, b$ (explicitly determined in Appendix \ref{Appendix} ) associated to $G/P$  such that 
	\begin{enumerate}
		\item  The pseudoeffective threshold $\alpha(G/P,H)$ is equal to $b/a$, namely  $\Lambda - \lambda \pi^* H$ is pseudoeffective if and only if $\lambda \leq b/a$.
		
		\item Let $r$ and $d$ be two arbitrary positive integers. Then 
		\[
			H^0(G/P, ({\rm Sym}^r T_{G/P})\otimes \sO_{G/P}(-dH))\not=0 \Longleftrightarrow b \left \lfloor \frac{r}{a} \right \rfloor\geq d,
		\]
		
        \item Conjecture \ref{Conj:normalized-Tangent} and hence Conjecture \ref{Conj2}  hold for  $G/P$.  
		\end{enumerate}

\end{theorem}

Note that Shao proved in \cite{Shao2020}, with completely different techniques (via Borel-Weil-Bott Theorem), the statements of Theorem \ref{t.RationalHomSpace} for IHSS. 
It seems hard to extend his arguments to this general setting.   

The main idea of the proof is to use the generically finite Springer map $\widehat{s}: T^*_{G/P}\rightarrow \overline{\sO}$ from the cotangent bundle of $G/P$ to its Richardson orbit closure.  By taking the Stein factorization and then taking the projectivization, we get a birational map  $\varepsilon:  \mathbb{P} (T_{G/P}) \to  \mathcal{Y} $. The birational geometry of $\varepsilon$ is well-understood (\cite{Namikawa2006}, \cite{Fu07}, \cite{Namikawa2008}), which implies for example when $\varepsilon$ is small, there exists a (projectivized) stratified Mukai flop (over $\mathcal{Y}$)
$\mu:  \mathbb{P} (T_{G/P})   \dasharrow  \mathbb{P} (T_{G/Q})$ with $G/P \simeq G/Q$.
This allows us to determine the  effective cone and the movable cone of  $\mathbb{P} (T_{G/P})$ (cf. Theorem \ref{Thm:Rational-Homogeneous-Spaces}) in terms of the exceptional divisor $\Gamma$ of $\varepsilon$ (resp.  $\mu^* \pi_2^* H$)  when $\varepsilon$ is  divisorial (resp. when $\varepsilon$ is small), where $\pi_2:  \mathbb{P} (T_{G/Q}) \to G/Q$ and $H$ is an ample generator of $\pic(G/Q)$. The two numbers $a$ and $b$ in Theorem \ref{t.RationalHomSpace} are the unique positive integers such that
\begin{center}
	$\Gamma \equiv a \Lambda - b \pi^*H$ \quad (resp. $\mu^*\pi_2^*H \equiv a \Lambda - b \pi_1^*H$).
\end{center}

It turns out the integer $a$ is very geometrical, which is related to the codegree of the VMRT of $G/P$   or to the degree of the images of lines under the stratified Mukai flops, while $b$ is an integer taking value $1$ or $2$, and $b=2$ if and only if the VMRT of $G/P$ is not dual defective and $G/P$ is not isomorphic to $E_7/P_4$. Subsequently we will divide $G/P$ into different types (Definition \ref{d.types}).  In order to compute them, we carry out a detailed study of stratified Mukai flops.  

One interesting observation is that we have $(a,b)=(4,2)$ for Fano contact manifolds of Picard number $1$ different to projective spaces (cf. Proposition \ref{Prop:Fano-contact}) and their VMRTs are the homogeneous Legendre varieties which form the main series of the conjectural list of nonsingular varieties with codegree $4$ and are also the examples of varieties of next to minimal degree  (cf. \cite[Remark 3.6 and Remark 4.16]{Zak2004} and \cite[p.168]{Tevelev2005}). 

Other interesting examples of Fano manifolds with Picard number $1$ are provided by moduli spaces $\text{SU}_C(r,d)$ of stable vector bundles of rank $r$ and degree $d$ over a nonsingular projective curve of genus $g$. Based on the work of Hwang-Ramanan \cite{HwangRamanan2004}, we show in Corollary \ref{Cor:Nonbig-Moduli} that the tangent bundle of $\text{SU}_C(r,d)$ is not big if $g\geq 4$, $r\geq 3$ and $(r,d)=1$.  In particular, the normalized tangent bundle of  $\text{SU}_C(r,d)$ is not pseudoeffective in this case.

\begin{remark}
As predicted by Conjecture \ref{Conj:normalized-Tangent},  a Fano manifold of Picard number 1 with pseudoeffective normalized tangent bundle must have nef tangent bundle. 
 If we can prove this, then Conjecture \ref{Conj:normalized-Tangent} would follow from Theorem \ref{t.RationalHomSpace} and the famous Campana-Peternell conjecture (see \cite{CampanaPeternell1991}) which predicts that Fano manifolds with nef tangent bundles must be homogeneous.
\end{remark}

Here is the organization of this paper: after a brief recall of various positivities of vector bundles in Section 2, we describe in Section 3 the pseudoeffective cone of $\mathbb{P}(T_X)$ in terms of total dual VMRT when $X$ is a Fano manifold of  Picard number $1$ with big tangent bundle whose VMRT is not dual defective.  Section 4 is devoted to the proof of Theorem \ref{Thm:normalized-I}.  We determine the pseudoeffective cone of $\mathbb{P}(T_{G/P})$ in Section 5 for rational homogeneous spaces $G/P$ of Picard number $1$.  Two non-homogeneous examples are studied in Section 6.

\subsection*{Acknowledgements}
We are very grateful to Jun-Muk Hwang, Francesco Russo and Fyodor Zak for their comments on a first version and it is a pleasure to thank Qifeng Li for the discussion on the VMRT of $F_4/P_3$. The second-named author would like to express his gratitude to Andreas H\"oring for bringing his attention to this problem and also for helpful discussions. We are very grateful to the anonymous referee for
the detailed report. The first-named author is supported by the National Natural Science Foundation of China (No. 12288201) and the second-named author by the National Key Research and Development Program of China (No. 2021YFA1002300) and the National Natural Science Foundation of China (No. 12288201 and No. 12001521).

\section{Cone of divisors and positivity of vector bundles}

\subsection{Cone of divisors}

Given a projective variety $X$, we consider the real vector space $N^1(X):=N^1_{\R}(X)$ of Cartier divisors, with real coefficients, up to numerical equivalence. Its dimension is equal to the \emph{Picard number} $\rho(X)$ of $X$. This vector space contains several important convex cones.

\begin{enumerate}
	\item The \emph{effective cone} $\Eff(X)$ is the convex cone in $N^1(X)$ generated by classes of effective divisors. This cone is neither closed nor open in general. The closure $\overline{\Eff}(X)$ of $\Eff(X)$ is called the \emph{pseudoeffective cone} of $X$. The interior of the effective cone $\Eff(X)$ is the \emph{big cone} $\Bigcone(X)$ of $X$, which is the convex cone generated by big $\R$-Cartier divisors.
	
	\item Denote by $\Mov(X)$ the cone in $N^1(X)$ generated by classes of movable divisors; that is, Cartier divisors $D$ on $X$ such that its stable base locus $\bB(D)$ has codimension at least two. Again, this cone is neither closed nor open. The closure $\overline{\Mov}(X)$ of $\Mov(X)$ is called the \emph{movable cone}. Recall that the \emph{stable base locus} of a $\Q$-Cartier, $\Q$-Weil divisor $D$ on a projective variety $X$ is the Zariski closed subset defined as
	\[
	\bB(D) := \bigcap_{m\in \N,\ mD\ {\rm Cartier}} \Bs(mD).
	\]
	
	\item The \emph{nef cone} $\Nef(X)$ is the cone of classes in $N^1(X)$ having non-negative intersection with all curves in $X$. This cone is closed by definition and its interior is the \emph{ample cone} $\Amp(X)$, which is generated by classes of ample divisors. In general  the nef cone is neither polyhedral nor rational.
\end{enumerate}
    Clearly, there are inclusions: $\Nef(X)\subseteq \overline{\Mov}(X) \subseteq \overline{\Eff}(X)$.
  
\subsection{Divisorial Zariski decomposition}

Let $D$ be a pseudoeffective $\R$-divisor on a smooth projective variety $X$. Recall that for a prime divisor $\Gamma$ on $X$ we can define
\[
\sigma_{\Gamma}(D)=\lim_{\epsilon\to 0^+} \inf \,\{\text{Mult}_{\Gamma} D'\, |\, D'\geq 0\ {\rm and}\ D'\sim_{\R} D+\epsilon A\},
\]
where $A$ is any fixed ample divisor.  By \cite[III, Corollary 1.11]{Nakayama2004},  there are only finitely many prime divisors $\Gamma$ on $X$ such that $\sigma_{\Gamma}(D)>0$. This allows us to make the following definition, see \cite[III]{Nakayama2004} and \cite{Boucksom2004}.

\begin{definition}
	Let $D$ be a pseudoeffective $\R$-divisor on a projective manifold $X$. Define
	\begin{center}
		$N_{\sigma}(D)=\sum_{\Gamma} \sigma_{\Gamma}(D)  \Gamma $\quad and \quad $P_{\sigma}(D)=D-N_{\sigma}(D)$.
	\end{center}
	The decomposition $D=N_{\sigma}(D)+P_{\sigma}(D)$ is called the divisorial Zariski decomposition of $D$.
\end{definition}

Note that $N_{\sigma}(D)$ is an effective $\R$-Weil divisor and $P_{\sigma}(D)$ is a movable $\R$-divisor, i.e., $[P_{\sigma}(D)]\in \overline{\Mov}(X)$ (cf. \cite[III, Proposition 1.14]{Nakayama2004}). In particular, for any prime divisor $\Gamma\subset X$ the restriction $P_{\sigma}(D)|_\Gamma$ is pseudoeffective.

\subsubsection{Augmented and restricted base loci}

Let $D$ be an $\R$-Cartier $\R$-Weil divisor on a normal projective variety $X$. The \emph{augmented base locus} (aka \emph{non-ample locus}) of $D$ is defined to be 
\[
\bB_+(D):=\bigcap_{A}\bB(D-A),
\]
where the intersection is over all ample divisors $A$ such that $D-A$ is a $\Q$-Cartier $\Q$-Weil divisor. The \emph{restricted base locus} (aka \emph{non-nef locus}) of $D$ is defined as
\[
\bB_{-}(D):=\bigcup_{A}\bB(D+A),
\]
where the union is taken over all ample divisors $A$ such that $D+A$ is a $\Q$-Cartier $\Q$-Weil divisor.  Recall that the augmented and restricted base locus depend only on the numerical equivalence class of $D$ and we refer the reader to \cite{EinLazarsfeldMustactuaNakamayeEtAl2006} for a detailed discussion of these notions. Let us denote by $\bB_{+}^1(D)$ (resp. $\bB_{-}^1(D)$) the union of codimension $1$ components of $\bB_{+}(D)$ (resp. $\bB_{-}(D))$.

\begin{lemma}
	\label{Lemma:Inclusion-ABL-RBL}
	Let $D$ and $D'$ be two pseudoeffective $\R$-Cartier $\R$-Weil divisors on a normal projective variety $X$. Assume that there exists an ample divisor $A$ such that $[D]$ is contained in the interior of the $2$-dimensional cone $\langle [D'],[A] \rangle$. Then we have $\bB_+(D)\subset \bB_{-}(D')$.
\end{lemma}

\begin{proof}
	By assumption, there exist positive real numbers $\lambda_{D'}$ and $\lambda_A$ such that $D\equiv_{\R} \lambda_{D'} D' + \lambda_A A$. By \cite[Lemma 1.14 and Lemma 1.8]{EinLazarsfeldMustactuaNakamayeEtAl2006}, we obtain
	\[
	\bB_{+}(D)=\bB_+(\lambda_{D'} D' + \lambda_A A) \subset \bB_{-}(\lambda_{D'} D') = \bB_{-}(D'),
	\]
	which concludes the proof.
\end{proof}

\begin{lemma}
	\label{Lemma:Base-Locus-DZD}
	Let $M$ be a movable $\R$-Cartier $\R$-Weil divisor on a normal projective variety $X$. Then $[M]$ is contained in the interior of $\overline{\Mov}(X)$ if and only if $\bB_{+}^1(M)=\emptyset$.
\end{lemma}

\begin{proof}
	Let $A$ be an arbitrary ample divisor on $X$. By \cite[Proposition 1.5]{EinLazarsfeldMustactuaNakamayeEtAl2006}, there exists $0<\epsilon\ll 1$ such that $\bB_{+}(M)=\bB(M-\epsilon' A)$ for any $0<\epsilon'\leq \epsilon$. In particular, it follows that $\bB_{+}^1(M)=\emptyset$ if and only if $\bB(M-\epsilon' A)$ does not contain divisorial parts, i.e., $M-\epsilon' A$ is movable, which holds if and only if $[M]$ is contained in the interior of $\overline{\Mov}(X)$.
\end{proof}

\subsubsection{Comparing base loci and $N_{\sigma}(D)$}

Given a pseudoeffective $\R$-Weil divisor $D$ on a projective manifold $X$, the augmented and restricted base loci are closely related to the divisorial Zariski decomposition of $D$.

\begin{lemma}
	\label{Lemma:Restricted-Base-Locus-DZD}
	Let $D$ be a pseudoeffective $\R$-Weil divisor on a projective manifold $X$. Then 
	\begin{enumerate}
		\item $\supp(N_{\sigma}(D))$ is precisely the divisor $\bB_{-}^1(D)$.
		
		\item If $D$ is not movable and $[D]$ generates an extremal ray of $\overline{\Eff}(X)$, then there exists a unique prime divisor $\Gamma \subset X$ such that 		
		$
		[\Gamma]\in \R_{>0} [D].
		$
		Moreover, we have 
		\[
		\Gamma=\supp(N_{\sigma}(D))=\bB_{-}^1(D).
		\]
	\end{enumerate}
\end{lemma}

\begin{proof}
	The statement (1) follows from \cite[V, Theorem 1.3]{Nakayama2004} and the statement (2) is proved in \cite[Lemma 2.5]{HoeringLiuShao2020}
\end{proof}

As an immediate application, a pseudoeffective $\R$-Weil divisor $D$ on a projective manifold $X$ is movable if and only if $\bB_{-}^1(D)$ is empty, see also \cite[III, Proposition 1.14]{Nakayama2004}.

\begin{corollary}
	\label{Cor:Equality-ABL-RBL}
	Given a projective manifold $X$, let $D$ be a pseudoeffective $\R$-Weil divisor and let $M$ be a movable $\R$-Weil divisor on $X$. Assume that
	\begin{enumerate}
		\item the divisor $D$ is not movable and $[D]$ generates an extremal ray of $\overline{\Eff}(X)$, and
		
		\item the divisor class $[M]$ is not contained in the interior of $\overline{\Mov}(X)$, and
		
		\item there exists an ample divisor $A$ such that $[M]$ is contained in the interior of the $2$-dimensional cone $\langle [D],[A]\rangle$.
	\end{enumerate}
    Then we have $\bB_+^1(M)=\bB_{-}^1(D)=\supp(N_{\sigma}(D))$, which is the unique prime divisor contained in the ray $\R_{>0}[D]$.
\end{corollary}

\begin{proof}
	By our assumption (3) and Lemma \ref{Lemma:Inclusion-ABL-RBL}, we have $\bB_+(M)\subset \bB_{-}(D)$. As $[M]$ is not contained in the interior of $\overline{\Mov}(X)$, it follows from Lemma \ref{Lemma:Base-Locus-DZD} that $\bB_+^1(M)$ is not empty. On the other hand, according to assumption (1) and Lemma \ref{Lemma:Restricted-Base-Locus-DZD}, one obtains that $\bB_-^1(D)=\supp(N_{\sigma}(D))$ is the unique prime divisor which is contained in $\R_{>0}[D]$. This forces that $\bB_{+}^1(M)=\bB_{-}^1(D)$.
\end{proof}

\subsection{Positivity of vector bundles}

Given a projective variety $X$, let $E$ be a vector bundle of rank $r$ over $X$. Denote by $\pi: \PP(E)\rightarrow X$ the projectivized bundle in the sense of Grothendieck; that is,
\[
\PP(E) := \Proj_X\left(\bigoplus_{r\geq 0} S^r E\right).
\]
Denote by $\Lambda$ the tautological divisor of $\PP(E)$, i.e., $\sO_{\PP(E)}(\Lambda)\cong \sO_{\PP(E)}(1)$. We will use the following terminologies throughout this paper, see \cite{Lazarsfeld2004a} for more details.

\begin{definition}
	Let $X$ be a projective variety. 
	\begin{enumerate}
		\item A $\Q$-twisted vector bundle $E\hspace{-0.8ex}<\hspace{-0.8ex}\delta\hspace{-0.8ex}>$ on $X$ is an ordered pair consisting of a vector bundle $E$ on $X$, defined up to isomorphisms, and a numerical equivalence $\Q$-Cartier divisor class $\delta\in N^1(X)$.
		
		\item The normalization of a vector bundle $E$ of rank $r$ on $X$ is the $\Q$-twisted vector bundle 
		\[
		E\hspace{-0.8ex}<\hspace{-0.8ex}-\frac{1}{r} c_1(E)\hspace{-0.8ex}>.
		\]
		
		\item A $\Q$-twisted vector bundle $E\hspace{-0.8ex}<\hspace{-0.8ex}\delta\hspace{-0.8ex}>$ is said to be pseudoeffective (resp. ample, big, nef) if the class $\Lambda+\pi^*\delta$ is pseudoeffective (resp. ample, big, nef) on $\mathbb{P}(E)$.
		
		\item A $\Q$-twisted vector bundle $E\hspace{-0.8ex}<\hspace{-0.8ex}\delta\hspace{-0.8ex}>$ is said almost nef if for a very general curve $C$, the restriction $E\hspace{-0.8ex}<\hspace{-0.8ex}\delta\hspace{-0.8ex}>|_C$ is nef.  Here very general curves mean that they intersect the complementary part of a countable union of proper subvarieties.

		%if there exist at most countably many proper subvarieties $Z_i\subset X$ such that the following holds: let $C\subset X$ be an irreducible closed curve such that the restriction $E\hspace{-0.8ex}<\hspace{-0.8ex}\delta\hspace{-0.8ex}>|_C$ is not nef, then $C$ is contained in $\cup Z_i$.
	\end{enumerate}
\end{definition}

The following properties are well known for experts and we include a complete proof for the reader's convenience, see also \cite[Lemma 2.2 and Lemma 2.3]{HoeringLiuShao2020}.

\begin{proposition}
	\label{Prop:criteria-pseff-bigness}
	Let $X$ be a projective variety. Let $E$ and $F$ be vector bundles over $X$ and let $\delta \in N^1(X)$ be a $\Q$-Cartier divisor class.
	\begin{enumerate}
		\item The $\Q$-twisted vector bundle $E\hspace{-0.8ex}<\hspace{-0.8ex}\delta\hspace{-0.8ex}>$ is pseudoeffective if and only if for  an arbitrary big $\Q$-Cartier $\Q$-Weil divisor $D$ on $X$ and an arbitrary $\Q$-Cartier $\Q$-Weil divisor $\Delta$ on $X$ such that $[\Delta]=\delta$, there exists an effective $\Q$-Weil divisor $N$ satisfying
		\[
		N\sim_{\Q} \Lambda + \pi^*(\Delta + D)
		\]
		
		\item The $\Q$-twisted vector bundle $E\hspace{-0.8ex}<\hspace{-0.8ex}\delta\hspace{-0.8ex}>$ is big if and only if the $\Q$-twisted vector bundle $E\hspace{-0.8ex}<\hspace{-0.8ex}\delta-\gamma\hspace{-0.8ex}>$ is pseudoeffective for some big $\Q$-Cartier class $\gamma \in \Bigcone(X)$.
	\end{enumerate}
\end{proposition}

\begin{proof}
	One direction of the statement (1) is clear since the pseudoeffective cone $\overline{\Eff}(\PP(E))$ is closed. For the converse, we assume that $E\hspace{-0.8ex}<\hspace{-0.8ex}\delta\hspace{-0.8ex}>$ is pseudoeffective. Since $D$ is big, by \cite[Chapter 2, Corollary 2.2.7]{Lazarsfeld2004}, there exists an ample $\Q$-Cartier $\Q$-Weil divisor $A$ and an effective $\Q$-Weil divisor $N'$ such that $D \sim_{\Q} A + N'$. On the other hand, as $\Lambda+\pi^*\delta$ is $\pi$-ample, there exists a rational number $0<\epsilon \ll 1$ such that the $\Q$-Cartier $\Q$-Weil divisor $\epsilon (\Lambda+\pi^*\Delta)+\pi^*A$ is ample. This implies that
	\[
	\Lambda+\pi^*(\Delta+D)\sim_{\Q} \epsilon(\Lambda+\pi^*\Delta)+\pi^*A + (1-\epsilon)(\Lambda+\pi^*\Delta)+N'
	\]
	is big since $\Bigcone(\PP(E))$ is the interior of $\overline{\Eff}(\PP(E))$. Then it follows again from \cite[Chapter 2, Corollary 2.2.7]{Lazarsfeld2004} that there exists an effective $\Q$-Cartier $\Q$-Weil divisor $N$ such that
	\[
	N\sim_{\Q} \Lambda + \pi^* (\Delta + D).
	\]
	
	One can easily obtain one implication of the statement (2), since $\Bigcone(\PP(E))$ is open. Conversely, we assume that $E\hspace{-0.8ex}<\hspace{-0.8ex}\delta-\gamma\hspace{-0.8ex}>$ is pseudoeffective for some big $\Q$-Cartier class $\gamma$. Similar to the proof of the statement (1), there exists a rational number $0<\epsilon \ll 1$ an ample $\Q$-Cartier $\Q$-Weil divisor $A$ and an effective $\Q$-Weil divisor $N$ such that the $\Q$-Cartier divisor class $\epsilon(\Lambda+\pi^*(\delta-\gamma))+\pi^*A$ is ample and
	\[
	\Lambda+\pi^*\delta \equiv_{\Q} \epsilon(\Lambda+\pi^*(\delta-\gamma))+\pi^*A+(1-\epsilon)(\Lambda+\pi^*(\delta-\gamma)) + N.
	\]
	Note that the $\Q$-Cartier divisor class $(1-\epsilon)(\Lambda+\pi^*(\delta-\gamma))$ is pseudoeffective by our assumption. Then it is clear that the $\Q$-Cartier divisor class $\Lambda+\pi^*\delta$ is big.
\end{proof}

We recall the following folklore  result: 

\begin{lemma} \label{l.semistable}
	Let $X$ be a projective manifold of dimension $n$ and $H$ an ample divisor. 
	Let $E$ be an $H$-semi-stable  vector bundle of rank $r$ on $X$. Then the normalized vector bundle $E\hspace{-0.8ex}<\hspace{-0.8ex}-\frac{1}{r} c_1(E)\hspace{-0.8ex}>$ is not big. 
\end{lemma}
\begin{proof}
	Assume $E\langle-ac_1(E) \rangle$ is effective for some rational number $a>0$; that is, we have
	\[
	H^0(X, {\rm Sym}^mE \otimes \det(E^*)^{\otimes (am)})\neq 0
	\]
	for some positive integer $m$ such that $am$ is an integer.  This gives an injection
	\[
	\det(E)^{\otimes (am)} \to {\rm Sym}^mE,
	\]  
	which  yields 
	\[
	\mu_H^{{\rm max}} ({\rm Sym^m}E) \geq \mu_H (\det(E)^{\otimes (am)}) = am c_1(E) \cdot H^{n-1}.
	\]
	On the other hand,  as $E$ is $H$-semi-stable, so is $\Sym^m E$. Hence, we obtain
	\[
	\mu_H^{{\rm max}} ({\rm Sym^m}E) = \mu_H ({\rm Sym^m} E) = \frac{m c_1(E)\cdot H^{n-1}}{r},
	\]
	which gives that $a \leq 1/r$. In particular, it follows from Proposition \ref{Prop:criteria-pseff-bigness} that $E\hspace{-0.8ex}<\hspace{-0.8ex}-\frac{1}{r} c_1(E)\hspace{-0.8ex}>$ is not big.
\end{proof}

\section{Fano manifolds with semi-ample tangent bundles}

\label{Section:Semisimple}

\subsection{Dual variety of VMRT}

Let $X$ be a smooth projective variety of dimension $n$. Denote by $\Rat(X)$ the normalization of the open subset of $\Chow{X}$ parametrizing integral rational curves. By a \emph{family of rational curves} in $X$, we mean an irreducible component $\sK$ of $\Rat(X)$. We denote by ${\rm Locus}(\sK)$ the locus of $X$ swept out by curves from $\sK$. We say that 
$\sK$ is \emph{minimal} if, for a general point $x\in {\rm Locus}(\sK)$ the closed subset $\sK_x$ of $\sK$ parametrizing curves through $x$ is proper. We say that $\sK$ is \emph{dominating} if ${\rm Locus}(\sK)$ is dense in $X$.  For an ample divisor $H$ on $X$, we write $H \cdot \sK$ the intersection number of $H$ with a general curve parametrized by $\sK$.

\subsubsection{Variety of minimal rational tangents}

Every uniruled projective manifold $X$ carries a dominating family of minimal rational curves. Fix one such family $\sK$. A general member $[C]\in \sK$ is a \emph{standard rational curve}, i.e. if we  denote by $f:\PP^1\rightarrow C$ its normalization, then there exists a non-negative integer $p$ such that
\[
f^*T_X \cong \sO_{\PP^1}(2)\oplus \sO_{\PP^1}(1)^{\oplus p} \oplus \sO_{\PP^1}^{\oplus (n-p-1)}.
\]

Given a general point $x\in X$, let $\sK_x^n$ be the normalization of $\sK_x$. Then $\sK_x^n$ is a finite union of smooth projective varieties of dimension $p$. Define
the tangent map $\tau_x:\sK^n_x\dashrightarrow \PP(\Omega_{X,x})$ by sending a curve that is smooth at $x$ to its tangent direction at $x$. Define $\sC_x$ to be the image of $\tau_x$ in $\PP(\Omega_{X,x})$. This is called the \emph{variety of minimal
rational tangents} (VMRT for short) at $x$ associated to the minimal family $\sK$. The map $\tau_x:\sK^n_x\dashrightarrow \sC_x \subset \PP(\Omega_{X,x})$ is in fact the normalization morphism by \cite{Kebekus2002,HwangMok2004}.

\subsubsection{Dual variety}

Let us recall the definition of dual varieties of projective varieties and we refer the reader to \cite{Tevelev2005} for more details. Let $V$ be a complex vector space of dimension $N+1$, and let $Z\subset \PP^N=\PP(V)$ be a projective variety. We denote by $T_{Z,z}$ the tangent space at any smooth point $z\in Z^{\text{sm}}$, where $Z^{\text{sm}}$ is the non-singular locus of $Z$. We denote by  $\mathbf{T}_{Z,z}\subset \PP^N$  the \emph{embedded projective tangent space}  of $Z$ at $z$. 
 A hyperplane $H\subset \PP^N$ is a tangent hyperplane of $Z$ if $\mathbf{T}_{Z,z}\subset H$ for some point $z\in Z^{\text{sm}}$. 

\begin{definition}
	\label{Def:Projectively-Dual}
	Let $Z\subset \PP^N=\PP(V)$ be a projective variety.
	
	\begin{enumerate}
		\item The closure of the set of all tangent hyperplanes of $Z$ is called the {\em dual variety} $\check{Z}\subset \check{\PP}^N=\PP(V^*)$, where $V^*$ is the dual space of $V$.
		
		\item The {\em dual defect} $\defect(Z)$ of $Z$ is defined as $N-1-\dim(\check{Z})$, and $Z$ is called dual defective if $\defect(Z)>0$.
		
		\item The {\em codegree} $\codeg(Z)$ of $Z$ is defined to be the degree of its dual variety $\check{Z}\subset \check{\PP}^N$.
	\end{enumerate}
	
\end{definition}

\subsubsection{Total dual variety of minimal rational tangents}

Let $C$ be a standard rational curve parametrized by $\sK$ with normalization $f:\PP^1\rightarrow C$. A minimal section of $\PP(T_X)$ over  the curve $C$ is a section (denoted by $\bar{C}$)  which corresponds to a quotient $f^*T_X\rightarrow \sO_{\PP^1}$. Recall that $p=n-1$ if and only $X$ is isomorphic to $\PP^n$ (cf. \cite{ChoMiyaokaShepherd-Barron2002,Kebekus2002a}). In particular, if $X$ is not isomorphic to projective spaces, such minimal sections always exist. Furthermore, we have $\Lambda \cdot \bar{C}=0$ for the tautological divisor $\Lambda$ on $\PP(T_X)$.

\begin{definition}
	\label{Def:Total-Dual-VMRT}
	Let $X$ be a uniruled projective manifold equipped with a dominating family $\sK$ of minimal rational curves. The total dual variety of minimal rational tangents (total dual VMRT for short) of $\sK$ is defined as
	\[
	\check{\sC}:=\overline{\bigcup_{[C]\in \sK:\ \text{standard} } \bar{C}}^{\text{Zar}} \subset \PP(T_X)
	\]
	where the union is taken over all minimal sections over all standard rational curves in $\sK$.
\end{definition}

We remark that $\check{\sC}$ is an irreducible projective variety. Moreover, for a general point $x\in X$, let us denote by $\check{\sC}_x$ the fibre of $\check{\sC}\rightarrow X$ over $x$. The next result justifies the terminology in Definition \ref{Def:Total-Dual-VMRT}:

\begin{proposition}
	\label{Prop:Dual-Defect-VMRT}
	\cite[Proposition 5.14 and 5.17]{MunozOcchettaSolaCondeWatanabeEtAl2015}
	Let $X$ be an $n$-dimensional uniruled projective manifold equipped with a dominating family $\sK$ of minimal rational curves and $x \in X$ a general point. Then $\check{\mathcal{C}}_x$ is the dual variety of $\mathcal{C}_x$. 
	
	Moreover, let $c$ be the dual defect of $\sC_x\subset \PP(\Omega_{X,x})$. Then for a minimal section $\bar{C}$ over a general standard rational curve $[C]\in \sK_x$ with normalization $\bar{f}: \PP^1 \rightarrow \bar{C}$, we have
	\[
	\bar{f}^*T_{\PP(T_X)}\cong \sO_{\PP^1}(-2)\oplus\sO_{\PP^1}(2)\oplus\sO_{\PP^1}(-1)^{\oplus c}\oplus\sO_{\PP^1}(1)^{\oplus c}\oplus\sO_{\PP^1}^{\oplus (2n-2c-3)}.
	\]
\end{proposition}

 As an immediate corollary of Proposition \ref{Prop:Dual-Defect-VMRT},  the dual variety of the VMRT $\sC_x$ at a general point $x$ is always pure dimensional. Moreover, the total dual VMRT $\check{\sC}$ is a prime divisor in $\PP(T_X)$ if and only if $\check{\sC}_x\subset \PP(T_{X,x})$ at a general point $x\in X$ is a (possibly reducible) hypersurface, i.e., $\sC_x\subset \PP(\Omega_{X,x})$ is not dual defective.
 
 The importance of the total dual VMRT in the study of positivity of tangent bundles is illustrated in the following theorem, see also \cite{MunozOcchettaSolaCondeWatanabeEtAl2015,OcchettaCondeWatanabe2016,HoeringLiuShao2020}.

\begin{theorem}
	\label{Thm:Class-dual-VMRT}
	Let $X$ be a Fano manifold of Picard number $1$ equipped with a dominating family $\sK$ of minimal rational curves. Let $H$ be the ample generator of $\pic(X)$ and let $\Lambda$ be the tautological divisor of $\pi:\PP(T_X)\rightarrow X$. Assume that the VMRT $\sC_x\subset \PP(\Omega_{X,x})$ at a general point $x\in X$ is not dual defective. Denote by $a$ and $b$ the unique integers such that 
	\[
	[\check{\sC}] \equiv a\Lambda - b\pi^*H.
	\]
	Then $a$ is equal to the codegree of $\sC_x$ and the following statements hold.
	\begin{enumerate}		
		\item $T_X$ is big if and only if $b>0$. 
		
		\item If $T_X$ is big, then $bH\cdot \sK\leq 2$ with equality if and only if there exists a minimal section $\bar{C}$ over a general standard rational curve $[C]\in \sK$ such that $\check{\sC}$ is smooth along $\bar{C}$. 
		
		\item If $T_X$ is big, then $[\check{\sC}]$ generates an extremal ray of $\overline{\Eff}(\PP(T_X))$; that is ,we have
		\[
		\overline{\Eff}(\PP(T_X))=\langle [\check{\sC}], [\pi^*H]\rangle.
		\]
	\end{enumerate}
\end{theorem}

\begin{proof}
	By our assumption, the projective variety $\check{\sC}_x\subset \PP(T_{X,x})$ is a (possibly reducible) hypersurface of degree $\codeg(\sC_x)$. On the other hand, we have
	\[
	[\check{\sC}]|_{\PP(T_{X,x})} \equiv (a\Lambda-b\pi^*H)|_{\PP(T_{X,x})}\equiv c_1\left(\sO_{\PP(T_{X,x})}(a)\right).
	\]
	This implies that $a$ is equal to the codegree of $\sC_x$. 
	
	\textit{Proof of (1).} Note that if $b>0$, then it follows from Proposition \ref{Prop:criteria-pseff-bigness} that $T_X$ is big. Now we assume that $T_X$ is big. Denote by $\alpha_X:=\alpha(X,H)$ the pseudoeffective threshold of $X$, namely the	 maximal positive real number such that $\Lambda-\alpha_X \pi^*H$ is pseudoeffective. Note that $\check{\sC}$ is dominated by minimal sections $\bar{C}$ over standard rational curves in $\sK$ and we have
	\[
	(\Lambda-\alpha_X \pi^*H) \cdot \bar{C}=-\alpha_X H\cdot C<0.
	\]
	Therefore, the restriction $(\Lambda-\alpha_X \pi^*H)|_{\check{\sC}}$ is not pseudoeffective. In particular, the $\R$-divisor $\Lambda-\alpha_X \pi^*H$ is not movable and the total dual VMRT $\check{\sC}$ is contained in the effective Weil divisor 
	\[
	\Gamma:=\supp(N_{\sigma}(\Lambda-\alpha_X \pi^*H))=\bB_{-}^1(\Lambda-\alpha_X \pi^*H).
	\] 
	As $X$ has Picard number $1$, it follows that $\rho(\PP(T_X))=2$ and $R=\R_{\geq 0}[\Lambda-\alpha_X \pi^*H]$ is an extremal ray of $\overline{\Eff}(\PP(T_X))$. Then it follows from Lemma \ref{Lemma:Restricted-Base-Locus-DZD} that $\Gamma$ is a prime divisor generating the extremal ray $R$. This yields that $\Gamma=\check{\sC}$ and hence $b>0$.
	
	\textit{Proof of (2).} Let $\bar{C}$ be a minimal section over a general standard rational curve $C$ in $\sK$ with normalization $\bar{f}:\PP^1\rightarrow \bar{C}$. As $\sC_x$ is not dual defective, by Proposition \ref{Prop:Dual-Defect-VMRT}, we have
	\begin{equation}
		\label{Eq:Splitting-Type-Min-Section}
		\bar{f}^*T_{\PP(T_X)} \cong \sO_{\PP^1}(-2)\oplus \sO_{\PP^1}(2)\oplus \sO_{\PP^1}^{\oplus (2n-3)}.
	\end{equation}
	Moreover, by the generic choice of $C$, we may assume that $\bar{C}$ is not contained in the singular locus of $\check{\sC}$. Then we have the following exact sequence of sheaves
	\[
	\sN^*_{\check{\sC}/\PP(T_X)}\longrightarrow \Omega_{\PP(T_X)}|_{\check{\sC}}\longrightarrow \Omega_{\check{\sC}} \longrightarrow 0,
	\]
	where $\sN^*_{\check{\sC}/\PP(T_X)}$ is the conormal sheaf of $\check{\sC}$ in $\PP(T_X)$. In particular, since $\check{\sC}$ is a Cartier divisor in $\PP(T_X)$, we have 
	\[
	\sN^*_{\check{\sC}/\PP(T_X)}=\sO_{\PP(T_X)}(-\check{\sC})|_{\check{\sC}}=\sO_{\PP(T_X)}(-a\Lambda+b\pi^*H)|_{\check{\sC}}.
	\]
	Consequently, the conormal sheaf is invertible. Pulling back the exact sequence by $\bar{f}$ yields an exact sequence
	\[
	\begin{tikzcd}[column sep=large, row sep=large]
		\bar{f}^*\sN^*_{\check{\sC}/\PP(T_X)}\cong \sO_{\PP^1}(bH\cdot C) \arrow[r,"\iota"] 
		    & \bar{f}^*\Omega_{\PP(T_X)} \arrow[r]
		        & \bar{f}^*\Omega_{\check{\sC}} \arrow[r]
		            & 0
	\end{tikzcd}
	\]
	Note that the map $\iota$ is generically injective since $\bar{C}$ is not contained in the singular locus of $\check{\sC}$. As $b>0$, it follows from \eqref{Eq:Splitting-Type-Min-Section} that $bH\cdot C\leq 2$ with equality if and  only if $\iota$ is an injection of vector bundles, i.e., $\bar{f}^*\Omega_{\check{\sC}}$ is locally free. By Nakayama's lemma, the latter one is equivalent to the smoothness of $\check{\sC}$ along $\bar{C}$. Conversely, if $\check{\sC}$ is smooth along $\bar{C}$, $\iota$ is an injection of vector bundles. In particular, as $b>0$, we obtain $bH\cdot C=2$ by \eqref{Eq:Splitting-Type-Min-Section}.
	
	\textit{Proof of (3).} Since $T_X$ is big and $X$ has Picard number $1$,  we have %then there exists a unique positive real number $\alpha(X)>0$ such that 
	\[
	\overline{\Eff}(\PP(T_X))=\langle [\Lambda - \alpha_X\pi^*H],[\pi^*H]\rangle.
	\]
	On the other hand, note that $\check{\sC}$ is dominated by curves with $\Lambda$-degree $0$, it follows that the restriction $(\Lambda-\alpha_X\pi^*H)|_{\check{\sC}}$ is not pseudoeffective. In particular, the $\R$-divisor $\Lambda - \alpha_X \pi^*H$ is not movable and $\check{\sC}$ is contained in $\supp(N_{\sigma}(\Lambda-\alpha_X \pi^*H))$. Then it follows from Lemma \ref{Lemma:Restricted-Base-Locus-DZD} that $\check{\sC}=\supp(N_{\sigma}(\Lambda-\alpha_X \pi^*H))$ and $[\check{\sC}]$ is contained in the ray $\R_{>0}[\Lambda-\alpha_X\pi^*H]$. 
\end{proof}

\begin{corollary}
	\label{Cor:Nonbig-Moduli}
	Let $C$ be a nonsingular projective curve of genus $\geq 4$. Let $X:={\rm SU}_C(r,d)$ be the moduli space of stable vector bundles of rank $r$ with fixed determinant of degree $d$. Assume that $r$ and $d$ are coprime. If $r\geq 3$, then $T_X$ is not big.
\end{corollary}

\begin{proof}
	It is known that $X$ is a nonsingular Fano manifold of Picard number $1$ such that $-K_X=2H$, where $H$ is the ample generator of $\pic(X)$. On the other hand, there exists a dominating family $\sK$ of minimal rational curves on $X$ given by the so-called Hecke curves such that $-K_X\cdot \sK=2r$ \cite[\S\,3]{HwangRamanan2004}.  By  \cite[Theorem 4.4]{HwangRamanan2004}, the total dual VMRT $\check{\sC}$ is a divisor in $\PP(T_X)$. Then Theorem \ref{Thm:Class-dual-VMRT} implies that $T_X$ is not big as $H\cdot \sK=r\geq 3$.
\end{proof}

\begin{remark}
	If $C$ is a nonsingular projective curve of genus $g=2$. Then the moduli space $X:={\rm SU}_C(2,r)$ with $r$ odd is isomorphic to the intersection of two quadrics in $\mathbb{P}^5$ and it is shown in \cite[Theorem 1.5]{HoeringLiuShao2020} that $T_X$ is pseudoeffective but not big. 
\end{remark}

\subsection{Semi-ample tangent bundles}

We consider in this subsection Fano manifolds with big and nef tangent bundle. It is conjectured by Campana-Peternell in \cite{CampanaPeternell1991} that a Fano manifold with nef tangent bundle must be a rational homogeneous space. Conversely, it is also known that the tangent bundle of a rational homogeneous space is big and globally generated. Recall that a vector bundle $E$ over a projective variety is said to be \emph{semiample} if $\sO_{\PP(E)}(1)$ is semiample.

\begin{lemma}
	\label{Lemma:Big+Nef=MDS}
	Let $X$ be an $n$-dimensional projective manifold such that $T_X$ is big and nef.
	\begin{enumerate}
		\item The tangent bundle $T_X$ is semi-ample.
		
		\item The projectivized tangent bundle $\PP(T_X)$ is a Mori dream space.
	\end{enumerate}
\end{lemma}

\begin{proof}
	As $T_X$ is big and nef, $\PP(T_X)$ is a weak Fano manifold, i.e., $-K_{\PP(T_X)}$ is big and nef. Then the statement (1) follows from the base-point-free theorem and the statement (2) follows from \cite[Corollary 1.3.2]{BirkarCasciniHaconMcKernan2010} since a weak Fano manifold is always log Fano.
\end{proof}

We refer the reader to \cite{HuKeel2000} for the definition of Mori dream spaces and their basic properties. 

\begin{definition}
	Let $X$ be a $\Q$-factorial normal projective variety. A small $\Q$-factorial modification (SQM for short) of $X$ is a birational map $g: X \dashrightarrow X'$, where $X'$ is a $\Q$-factorial normal projective variety and $g$ is an isomorphism in codimension $1$.
\end{definition}

Throughout the rest of this subsection, we will always assume that $X$ is  a Fano manifold of Picard number $1$ such that $T_X$ is big and nef. Let us denote by $H$ the ample generator of $\pic(X)$ and by $\Lambda$ the tautological divisor of $\PP(T_X)$. Then the evaluation of global sections defines a birational morphism
\begin{equation}
	\label{Eq:Evaluation-Mor}
	\begin{tikzcd}[column sep=large, row sep=large]
		\sX:= \Proj_X\left(\bigoplus_{r\geq 0} S^r T_X\right) \arrow[r,"\varepsilon"] \arrow[d,"\pi"] 
		    & \sY:=\Proj\left(\bigoplus_{r\geq 0} H^0(X,S^r T_X)\right)   \\
		 X 
		    &
	\end{tikzcd}
\end{equation}

The morphism $\varepsilon$ is an isomorphism if and only if $T_X$ is ample, and Mori proved in \cite{Mori1979} that the tangent bundle of a projective manifold $X$ is ample if and only if $X$ is isomorphic to a projective space. For projective spaces, we have the following description of the cones of divisors.

\begin{example}
	\label{Ex:Projective-Spaces}
	Let $X$ be the $n$-dimensional projective space $\PP^n$ with $n\geq 2$, and let $\Lambda$ be the tautological divisor class of $\pi:\PP(T_{X})\rightarrow X$. Then we have
	\begin{center}
		$\overline{\Eff}(\PP(T_{\PP^n}))=\Eff(\PP(T_{\PP^n}))=\Nef(\PP(T_{\PP^n}))=\left\langle [\Lambda -\pi^*H], [\pi^*H]\right\rangle$,
	\end{center}
	where $H$ is a hyperplane section of $\PP^n$. Indeed, we consider the following Euler sequence
	\[
	0\longrightarrow \sO_{\PP^n} \longrightarrow \sO_{\PP^n}(1)^{\oplus (n+1)} \longrightarrow T_{\PP^n} \longrightarrow 0.
	\]
	It follows that $T_{\PP^n}(-1)$ is globally generated. In particular, the divisor class $[\Lambda-\pi^*H]$ is contained in the intersection $\Eff(\PP(T_X))\cap \Nef(\PP(T_X))$. On the other hand, it is known that $\Lambda-\pi^*H$ is not big and hence $[\Lambda-\pi^*H]$ is not contained in the interior of $\Eff(\PP(T_X))$.
\end{example}

Let us collect some basic properties about the morphism $\varepsilon$. 
\begin{proposition}
	\label{Prop:Properties-Contraction}
	Let $X$ be a Fano manifold of Picard number $1$ such that $T_X$ is big and nef. Denote by $\varepsilon:\sX\rightarrow \sY$ the birational morphism given in \eqref{Eq:Evaluation-Mor}. If $\varepsilon$ is a divisorial contraction, then  the following statements hold.
	\begin{enumerate}
		\item The projective variety $\sY$ has at worst $\Q$-factorial canonical singularities. 
		
		\item The exceptional locus of $\varepsilon$ is an irreducible divisor $\Gamma$ such that the general fibre of $\Gamma\rightarrow \varepsilon(\Gamma)$ consists of either a smooth $\PP^1$ or the union of two $\PP^1$'s meeting at a point.
		
		\item Let $F$ be an irreducible component of a general one dimensional fibre of $\varepsilon$. Then there exists a non-negative integer $a$ such that
		\[
		T_{\sX}|_{F}\cong \sO_{\PP^1}(2)\oplus \sO_{\PP^1}(-2)\oplus \sO_{\PP^1}(1)^{\oplus a} \oplus \sO_{\PP^1}(-1)^{\oplus a}\oplus \sO_{\PP^1}^{\oplus 2n-2a-3}.
		\]
	\end{enumerate}
\end{proposition}

\begin{proof}

Claims (1) and (2) follow from  \cite[Theorem 1.3]{Wierzba2003} and \cite[Proposition 5.10]{MunozOcchettaSolaCondeWatanabeEtAl2015}.
To prove (3), we follow the argument of \cite[Proposition 2.13]{Wierzba2003}. Let 
\[
0\rightarrow E\rightarrow T_{\sX}\rightarrow \sO_{\sX}(1)\rightarrow 0
\]
be the natural contact structure on $\sX$. By our assumption, we have $\sO_{\sX}(1)|_F\cong \sO_F$. Then the contact structure induces a natural isomorphism $E|_F\cong E^*|_F$. Let 
	\[
	E|_F\cong \bigoplus_{i=1}^{2n-2} \sO_{\PP^1}(a_i)
	\]
	be the decomposition with $a_1\geq \cdots\geq a_{2n-2}$. Since the problem is local in $\sY$, after removing a subvariety of codimension at least $4$ of $\sY$, we may assume that all fibres of $\varepsilon$ are at most $1$-dimensional. Then the argument of \cite[Proposiion 2.13]{Wierzba2003} applies verbatim to our situation to obtain $h^1(F,\Omega_X|_F)=1$ and the short exact sequence below
	\begin{equation}
		\label{Eq:NormalseqSymplecticcontraction}
		0\longrightarrow \sO_F \longrightarrow \Omega_X|_F \longrightarrow E^*|_F=E|_F\longrightarrow 0
	\end{equation}
	implies $h^1(F,E|_F)=1$. This implies that $a_{2n-2}=-2$ and $a_{2n-3}\geq -1$. The isomorphism $E|_F\cong E^*|_F$ shows that $E|_F$ must be of the form
	\[
	\sO_{\PP^1}(2)\oplus \sO_{\PP^1}(1)^{\oplus a} \oplus \sO_{\PP^1}^{\oplus 2n-2a-4} \oplus \sO_{\PP^1}(-1)^{\oplus a} \oplus \sO_{\PP^1}(-2).
	\]
	Then it follows from \eqref{Eq:NormalseqSymplecticcontraction} and the fact $h^1(E|_F)=h^1(E^*|_F)=1$ that $T_X|_F$ is either of the form $E\oplus \sO_{\PP^1}$ or of the form
	\[
	\sO_{\PP^1}(2)\oplus \sO_{\PP^1}(1)^{\oplus a} \oplus \sO_{\PP^1}^{\oplus 2n-2a-4} \oplus \sO_{\PP^1}(-1)^{\oplus a+2}.
	\]
	It is clear that the $\Chow{\sX}$ has dimension $\geq 2n-3$ at $[F]$ as $\Gamma$ has dimension $2n-2$ and the deformation of $F$ dominates $\Gamma$. Hence, we have $h^0(F,T_{\sX}|_F)\geq 2n$ and $T_X|_F$ is of the form $E\oplus \sO_{\PP^1}$.
\end{proof}

\begin{definition}
	\label{Def:Singularities-A1-A2}
	Let $X$ be a Fano manifold of Picard number $1$ such that $T_X$ is big and nef. Denote by $\varepsilon:\sX\rightarrow \sY$ the birational morphism given in \eqref{Eq:Evaluation-Mor}. The projective variety $\sY$ is of type $A_1$ (resp. $A_2$) if the morphism $\varepsilon$ is a divisorial contraction and the general fibre of $E\rightarrow \varepsilon(E)$ is a smooth $\PP^1$ (resp. union of two $\PP^1$'s meeting in a point).
\end{definition}

In the sequel of this subsection, we will focus on the description of the cones of divisors of $\sX$.  Similar to the pseudoeffective threshold $\alpha_X:=\alpha(X,H)$, we define the movable threshold $\beta_X:=\beta(X,H)$ to be the 
 maximal real number such that the $\R$-divisor $\Lambda-\beta_X \pi^*H$ is movable.  Clearly we have $\alpha_X \geq \beta_X$. Since $\Lambda$ is big, by Proposition \ref{Prop:criteria-pseff-bigness}, we obtain $\alpha_X>0$. Moreover, as $\Lambda$ is semiample, we also have $\beta_X\geq 0$. 

Given a Weil divisor $\Gamma\subset \sX$, let us denote by $a(\Gamma)$ and $b(\Gamma)$ the unique integers such that 
\[
\Gamma\equiv a(\Gamma)\Lambda - b(\Gamma)\pi^*H.
\]
Firstly we have the following general observation.

\begin{proposition}
	\label{Prop:Cones-Semiample}
	Let $X$ be an $n$-dimensional Fano manifold of Picard number $1$ such that $T_X$ is big and nef. 
	\begin{enumerate}
		\item Both $\alpha_X$ and $\beta_X$ are rational numbers and there exists a SQM $g:\sX'\dashrightarrow \sX$ such that the $\Q$-Cartier $\Q$-Weil divisor $g^*(\Lambda-\beta_X\pi^*H)$ is semi-ample.
		
		\item If $\alpha_X\not=\beta_X$, then there exists a unique prime divisor $\Gamma \subset \sX$ such that 
		\begin{center}
			$[\Gamma]\in \R_{>0}[\Lambda-\alpha_X\pi^*H]$\quad and\quad
			$g^*\Gamma \cdot \left(g^*(\Lambda-\beta_X\pi^*H)\right)^{2n-2}=0$,
		\end{center}
		where $g:\sX'\dashrightarrow \sX$ is the SQM provided in the statement (1).
	\end{enumerate}
\end{proposition}
 
\begin{proof}
	Recall that $\sX$ is a Mori dream space by Lemma \ref{Lemma:Big+Nef=MDS}. By \cite[Proposition 1.11]{HuKeel2000}, there exists a SQM $g:\sX' \dasharrow  \sX$ such that 
	\[
	[g^*(\Lambda-\beta_X\pi^*H)]\in \Nef(\sX').
	\]
	Moreover, as $\sX'$ is again a Mori dream space, it follows that $\Nef(\sX')$ is generated by semi-ample $\Q$-Cartier divisors. Hence, $\beta_X$ is a rational number. 
	
	Now assume that $\alpha_X\not=\beta_X$. Then $\Lambda-\alpha_X\pi^*H$ is not movable. Moreover, as $X$ has Picard number $1$, it is clear that $R=\R_{>0}[\Lambda-\alpha_X\pi^*H]$ is an extremal ray of $\overline{\Eff}(\sX)$. Then, by Lemma \ref{Lemma:Restricted-Base-Locus-DZD}, there exists a unique prime divisor $\Gamma \subset \sX$ such that $[\Gamma]\in \R_{>0}[\Lambda-\alpha_X \pi^*H]$. In particular, we have
	\[
	\alpha_X=\frac{b(\Gamma)}{a(\Gamma)}
	\]
	and hence $\alpha_X$ is again a rational number. Denote by $\Gamma'$ the divisor $g^*\Gamma$. Note that the pseudoeffective cones and movables are preserved by $g^*$. In particular, by Lemma \ref{Lemma:Inclusion-ABL-RBL}, we obtain 
	\[
	\bB_+^1(g^*(\Lambda - \beta_X\pi^*H))\subset \bB_{-}^1(\Gamma')\subset \Gamma'.
	\]
	Since $\Lambda-\beta_X \pi^*H$ is not contained in the interior of $\overline{\Mov}(\sX)$, so is the pull-back $g^*(\Lambda-\beta_X\pi^*H)$. In particular, by Lemma \ref{Lemma:Base-Locus-DZD}, we have
	\[
	\bB_+^1(g^*(\Lambda - \beta_X \pi^*H))=\Gamma'.
	\]
	On the other hand, as $g^*(\Lambda-\beta_X\pi^*H)$ is nef, by \cite[Theorem 1.4]{Birkar2017}, $\Gamma'$ is contained in the null locus of $g^*(\Lambda-\beta_X\pi^*H)$. In particular, we obtain
	\[
	g^*\Gamma \cdot \left(g^*(\Lambda-\beta_X\pi^*H)\right)^{2n-2} = \Gamma'\cdot \left(g^*(\Lambda-\beta_X\pi^*H)\right)^{2n-2} = 0.
	\]
	This completes the proof.
\end{proof}

According to Proposition \ref{Prop:Cones-Semiample}, the calculation of the cones of divisors of $\sX$ is very closely related to the study of possible SQMs of $\sX$, which in general seems to be a very difficult problem. However, if we assume that the morphism $\varepsilon:\sX\rightarrow \sY$ is a divisorial contraction, then the cones of divisors of $\sX$ can be explicitly determined.

\begin{proposition}
	\label{Prop:Pseff-Cone-Divisorial-Contraction}
	Let $X$ be an $n$-dimensional Fano manifold of Picard number $1$ such that $T_X$ is big and nef and the VMRT of $X$ at a general point is smooth. Assume that the evaluation morphism $\varepsilon:\sX\rightarrow \sY$ is a divisorial contraction with exceptional divisor $\Gamma$. Let $F$ be an irreducible component of a general fibre of $\Gamma \rightarrow \varepsilon(\Gamma)$. Then we have
	\begin{enumerate}
		\item $\beta_X=0$, $\bB_{+}(\Lambda)=\Gamma$ and $[\Gamma]$ generates the extremal ray $ \R_{>0}[\Lambda-\alpha_X\pi^*H]$. In particular, we have $\Gamma\cdot \Lambda^{2n-2}=0$.
		
		\item $b(\Gamma)\leq 2$ with equality if and only if $\sY$ is of type $A_1$ and there exists a dominating family $\sK$ of minimal rational curves on $X$ such that $\check{\sC}=\Gamma$ and $H\cdot \sK=1$. 
	\end{enumerate}
\end{proposition}

\begin{proof}
	Since $\Lambda$ is big and nef, it follows from \cite[Theorem 1.4]{Birkar2017} that $\bB_{+}(\Lambda)$ coincides with the exceptional locus $\Gamma$ of $\varepsilon$. In particular, $\Gamma\cdot \Lambda^{2n-2}=(\Lambda|_{\Gamma})^{2n-2} = 0$. Moreover, according to Lemma \ref{Lemma:Base-Locus-DZD}, $[\Lambda]$ is not contained in the interior of $\overline{\Mov}(\sX)$. This implies $\beta_X=0$. Then it follows from Corollary \ref{Cor:Equality-ABL-RBL} that $[\Gamma]\in \R_{>0}[\Lambda-\alpha_X\pi^*H]$. Combining Proposition \ref{Prop:Properties-Contraction} with the same argument as in the proof of Theorem \ref{Thm:Class-dual-VMRT}(3) shows that $b(\Gamma)\pi^*H\cdot F\leq 2$ with equality if and only if $\Gamma$ is smooth along $F$. Then we obtain that $b(\Gamma)\leq 2$ with equality if and only if $\Gamma$ is smooth along $F$ and $\pi^*H\cdot F=1$. Now the result follows from the following two claims.
	
	\bigskip
	
	\textbf{Claim 1.} \textit{$\Gamma$ is smooth along $F$ if and only if $\sY$ is of type $A_1$.}
	
	\bigskip
	
    \textit{Proof of Claim 1.} Firstly we assume that $\Gamma$ is smooth along $F$, then the non-singular locus $\Gamma^{\text{sm}}$ contains $F$. In particular, by generic smoothness and the generic choice of $F$, it follows that the fibre of $\Gamma^{\text{sm}}\rightarrow \varepsilon(\Gamma^{\text{sm}})$ over $\varepsilon(F)$ is smooth. Nevertheless, if the fibre of $\varepsilon$ over $\varepsilon(F)$ consists of another irreducible component $F'$ such that $F$ and $F'$ meeting at a point $x$, then we have $x\in \Gamma^{\text{sm}}$ and therefore $F'\cap \Gamma^{\text{sm}}$ is not empty. In particular, the fibre of $\Gamma^{\text{sm}}\rightarrow \varepsilon(\Gamma^{\text{sm}})$ over $\varepsilon(F)$ is not smooth, a contradiction. Hence, $\sY$ is of type $A_1$. 
    
    Conversely, if $\sY$ is of type $A_1$, then $\Gamma\rightarrow \varepsilon(\Gamma)$ is a smooth $\PP^1$-fibration over a Zariski open subset of $\varepsilon(\Gamma)$. In particular, the singular locus of $\Gamma$ does not dominate $\varepsilon(\Gamma)$ and hence $\Gamma$ is smooth along $F$ as $F$ is a general fibre.
    
    \bigskip
    
    \textbf{Claim 2.} \textit{$\pi^*H\cdot F=1$ if and only if there exists a dominating family $\sK$ of minimal rational curves over $X$ such that $H\cdot \sK=1$ and $\check{\sC}=\Gamma$.}
    
    \bigskip
	
	\textit{Proof of Claim 2.} Firstly we assume that $\pi^*H\cdot F=1$. Then the induced morphism $F\rightarrow \pi(F)$ is birational and $H\cdot F=1$. In particular, the images of the irreducible components of general fibres of $\Gamma \rightarrow \varepsilon(\Gamma)$ in $X$ form a dominating family $\sK$ of rational curves such that $H\cdot \sK=1$. Therefore, $\sK$ is actually a dominating family of minimal rational curves. Let $\check{\sC}$ be the total dual VMRT of $\sK$. As $\check{\sC}$ is dominated by curves with $\Lambda$-degree $0$, it follows that $\check{\sC}\subset \Gamma$. On the other hand, as the VMRT of $X$ at a general point is smooth, every rational curve parametrised by $\sK$ passing through a general point is standard (cf. \cite[Proposition 1.4]{Hwang2001}). In particular, by generic choice of $F$, we may assume that $\pi(F)$ is a standard rational curve parametrised by $\sK$. In particular, the curve $F$ is a minimal section over $\pi(F)$. It follows that $\Gamma \subset \check{\sC}$ and hence $\Gamma=\check{\sC}$.
	
	Conversely, assume that there exists a dominating family $\sK$ of minimal rational curves on $X$ such that $H\cdot \sK=1$ and $\check{\sC}=\Gamma$. As $\check{\sC}$ is dominated by minimal sections $\bar{C}$ over standard rational curves $C$ in $\sK$, it follows that $\bar{C}$ is contained in a general fibre of $\Gamma \rightarrow \varepsilon(\Gamma)$. In particular, we have $\pi^*H\cdot \bar{C}=H\cdot C=1$. By the generic choice of $F$, the curve $F$ is actually a minimal section over some standard rational curve in $\sK$ and hence $\pi^*H\cdot F=1$.
\end{proof}

\begin{remark}
	In the setting of Proposition \ref{Prop:Pseff-Cone-Divisorial-Contraction}, to explicitly determine the pseudoeffective cone $\overline{\Eff}(\sX)$, it is enough to calculate the cohomological class of $\Gamma$ in $\pic(\sX)$, i.e., determining $a(\Gamma)$ and $b(\Gamma)$. The statement (2) gives a totally geometric method to determine $b(\Gamma)$. Then one can use the equality $\Gamma\cdot \Lambda^{2n-2}=0$ in the statement (1) to obtain the rational number $b(\Gamma)/a(\Gamma)$ and finally we get the precise value of $a(\Gamma)$. On the other hand, if there exists a dominating family $\sK$ of minimal rational curves on $X$ such that $\check{\sC}$ is a divisor, then we must have $\check{\sC}=\Gamma$ and we can also apply Theorem \ref{Thm:Class-dual-VMRT} to calculate $a(\Gamma)$. In a later section we will apply these results to rational homogeneous spaces.
\end{remark}

\section{Varieties of small codegree and proof of Theorem \ref{Thm:normalized-I}}

\subsection{Segre inequality}

Let us recall the following Segre inequality, which gives a sharp lower bound for the codegree of an irreducible and linearly non-degenerate projective variety in terms of its  dimension and codimension.

\begin{theorem}
	\label{Thm:Segre-Zak}
	\cite{Segre1951}
	Let $Z\subsetneq \PP^N$ be an $n$-dimensional irreducible and linearly non-degenerate projective variety. Assume that the dual variety $\check{Z}\subset \check{\PP}^N$ is a hypersurface with non-vanishing hessian. Then we have
	\begin{equation}
		\label{Eq:Ineq-Segre-Zak}
		\codeg(Z):=\deg(\check{Z})\geq \frac{2(N+1)}{n+2}.
	\end{equation}
	Moreover,  the equality holds if and only if $\check{Z}\subset \check{\PP}^N$ is a hypersurface defined by $F=0$ such that  its  hessian $h_F$ satisfies $h_F =  F^{N-n-1}$.  
\end{theorem}

\begin{remark}
		Zak kindly informed us that  the  Segre inequality may fail if  the dual variety $\check{Z}$ is a hypersurface with vanishing hessian. There are very few known examples of smooth projective varieties whose dual variety is a hypersurface with vanishing-hessian. Gondim, Russo and Staglian\`o proved in \cite[Corollary 4.5]{GondimRussoStagliano2020} that the projection from an internal point of $\nu_2(\PP^n)\subset \PP^{\frac{n^2+3n}{2}}$ is a smooth variety $Z\subset \PP^{\frac{n^2+3n-2}{2}}$ such that the dual variety $\check{Z}$ is a degree $n+1$ hypersurface with vanishing hessian.  It would be very interesting to find more examples.
\end{remark}

It is somehow surprising that there exists a link between Conjecture \ref{Conj:Zak-Russo-Conjecture} and Conjecture \ref{Conj:normalized-Tangent} ,  which is bridged by the following simple observation:

\begin{proposition}
	\label{Prop:Criterion-Pseff-Normalized}
	Let $X$ be an $n$-dimensional Fano manifold of Picard number $1$ equipped with a dominating family $\sK$ of minimal rational curves. If the normalized tangent bundle of $X$ is pseudoeffective and the VMRT $\sC_x\subset \PP(\Omega_{X,x})$ at a general point is not dual defective, then we have
	\begin{equation}
		\label{Eq:Pseff-implies-reverse-Segre-Zak}
		\codeg(\sC_x)\leq \frac{2\dim(X)}{\dim(\sC_x)+2}.
	\end{equation}
\end{proposition}

\begin{proof}
	Let $H$ be the ample generator of $\pic(X)$ and denote by $\alpha_X:=\alpha(X,H)$ the pseudoeffective threshold of $X$ with respect to $H$. Let $i_X$ be the index of $X$, i.e. $-K_X=i_XH$. Then the normalized tangent bundle of $X$ is pseudoeffective if and only if the following inequality holds:
	\[
	\alpha_X\geq \frac{i_X}{\dim(X)}.
	\]
	On the other hand, since $\sC_x$ is not dual defective, the total dual VMRT $\check{\sC}\subset \PP(T_X)$ is a prime divisor. Write $[\check{\sC}]\equiv a \Lambda - b\pi^*H$. Then, by Theorem \ref{Thm:Class-dual-VMRT}, we obtain
	\[
	a=\codeg(\sC_x),\quad 0<bH\cdot \sK\leq 2 \quad \text{and} \quad \alpha_X=\frac{b}{a}.
	\]
	Therefore we get
	\[
	\frac{2}{\codeg(\sC_x)}\geq \alpha_X H\cdot \sK \geq \frac{i_X H\cdot \sK}{\dim(X)}=\frac{\dim(\sC_x)+2}{\dim(X)},
	\]
	and the result follows. Here we use the fact that $\dim(\sC_x)=-K_X\cdot \sK-2=i_XH\cdot \sK-2$.
\end{proof}

Given a Fano manifold $X$ of Picard number $1$, once the VMRT $\sC_x\subset \PP(\Omega_{X,x})$ of $X$ can be explicitly determined and the VMRT is not dual defective, then Proposition \ref{Prop:Criterion-Pseff-Normalized} is quite useful to check whether the normalized tangent bundle of $X$ is pseudoeffective  or not. For Conjecture \ref{Conj:Zak-Russo-Conjecture}, we recall the following results for curves and surfaces.

\begin{theorem}
	\label{Thm:Classification-Small-Codegree}
	\cite[Proposition 3.1 and 3.2]{Zak2004} \cite[Theorem 2.1]{TurriniVerderio1993}
	\begin{enumerate}
		\item Let $C\subset \PP^N$ be a linearly non-degenerate smooth projective curve of degree $d$ and codegree $d^*$. Then the following statements hold.
		\begin{enumerate}
			\item $d^*\geq 2d-2$ with equality if and only if $C$ is a rational curve.
			
			\item $d^*\geq 2N-2$ with equality if and only if $C$ is a normal rational curve.
		\end{enumerate}
		
		\item Let $S\subset \PP^N$ be a linearly non-degenerate smooth projective surface of degree $d$ and codegree $d^*$. Then the following statements hold.
		\begin{enumerate}
			\item $d^*\geq d-1$ with equality if and only if $S$ is isomorphic to the Veronese surface $\nu_2(\PP^2)\subset \PP^5$ or its isomorphic projection in $\PP^4$, and $d^*=d$ if and only if $S$ is a scroll over a curve, and the cases $1\leq d^*-d\leq 2$ does not happen.
			
			\item $d^*\geq N-2$ with equality if and only if $S$ is isomorphic to the Vernoese surface $\nu_2(\PP^2)\subset \PP^5$, and $d^*=N-1$ if and only if $S$ is either an isomorphic projection of $\nu_2(\PP^2)$ to $\PP^4$ or a rational normal scroll, and the cases $0\leq d^*-N\leq 1$ does not happen.
		\end{enumerate}
	\end{enumerate}
\end{theorem}

\subsection{Projective threefolds with small codegree}

This subsection is devoted to prove Proposition \ref{Thm:Threefold-Small-Codegree}, which confirms Conjecture \ref{Conj:Zak-Russo-Conjecture} for smooth threefolds. We start with a classification of projective threefolds such that its general hyperplane section is a smooth surface with equal sectional genus and irregularity. Let us recall that for an $n$-dimensional polarized projective manifold $(X,L)$, the \emph{sectional genus} of $X$ (with respect to $L$) is defined to be
\[
g(X,L):=\frac{(K_X+(n-1)L)\cdot L^{n-1}}{2}+1.
\]

\begin{lemma}
	\label{Lemma:Surface-g=q}
	Let $Z\subsetneq \PP^N$ be an irreducible, smooth and linearly non-degenerate projective threefold and denote by $L$ the restriction $\sO_{\PP^N}(1)|_Z$. Let $S$ be a general smooth hyperplane section of $Z$. Assume that the sectional genus $g$ of $S$ is equal to the irregularity $q$ of $S$. Then $(Z,L)$ is isomorphic to one of the following varieties:
	\begin{enumerate}
		\item the $3$-dimensional quadric hypersurface $(\Q^3,\sO_{\Q^3}(1))$ and $\codeg(Z)=2$;
		
		\item  a $3$-dimensional scroll, i.e. a projective bundle $\PP(E)\rightarrow B$ over a smooth curve $B$ such that all fibres are linearly embedded, and $L$ is the tautological line bundle $\sO_{\PP(E)}(1)$. In particular, the dual defect of $Z=\PP(E)$ is equal to $1$ and $\codeg(Z)=\deg(Z)=c_1(E)$.
	\end{enumerate}
\end{lemma}

\begin{proof}
	Denote by $\bar{L}$ the restriction $L|_S$. As $g=q$, by \cite{Zak1973}(see also \cite[Corollary 1.5.2]{Sommese1979}), we know that either $S$ is a geometrically ruled surface with smooth $C\in |\bar{L}|$ as sections or the pair $(S,\bar{L})$ is isomorphic to one of the following:
	\begin{center}
		$(\PP^2,\sO_{\PP^2}(1))$\quad or \quad $(\PP^2,\sO_{\PP^2}(2))$.
	\end{center}
	
	Firstly we note that the case $(S,\bar{L})=(\PP^2,\sO_{\PP^2}(2))$ does not happen. This was already proved by Scorza. Indeed, by Bott's formula, we have $H^1(\PP^2,T_{\PP^2}(-2))=0$. Then we can apply Zak's inextendibility theorem (see \cite{Zak1991}) to conclude that $Z$ is a cone over $S$. In particular, $Z$ is singular, which is a contradiction.
	
	Next we assume that the pair $(S,\bar{L})$ is isomorphic to $(\PP^2,\sO_{\PP^2}(1))$. Then we have $L^3=\bar{L}^2=1$. In particular, $(Z,L)$ itself is isomorphic to $(\PP^3,\sO_{\PP^3}(1))$, which contradicts our assumption.
	
	Finally we assume that $S$ is a geometrically ruled surface over a smooth curve. According to \cite[Theorem 1.3]{Liu2019}, the pair $(Z,L)$ is one of the following varieties:
	\begin{enumerate}
		\item  $(\Q^3,\sO_{\Q^3}(1))$;
		
		\item  $(\PP^3,\sO_{\PP^3}(2))$;
		
		\item there exists a vector bundle $E$ of rank $3$ over $B$ such that $Z=\PP(E)$ and $S$ is an element in the linear system $|\sO_{\PP(E)}(1)|$.
	\end{enumerate}
	In Case (1), it is clear that $Z\subset \PP^N$ is linearly normal and hence it is a quadric hypersurface of $\PP^4$. In Case (2), one can easily obtain that $g(S)=1$ while $q(S)=0$, which does not satisfy our assumption. In case (3), we note that $Z\subset \PP^N$ is actually a $3$-dimensional scroll such that all the fibres of $\PP(E)\rightarrow B$ are linearly embedded. As $B$ is a curve, it is well-known that the dual defect of $Z$ is equal to $1$ in this case (see for instance \cite[Theorem 7.21]{Tevelev2005}) and the fact $\codeg(Z)=\deg(Z)=c_1(E)$ follows from Lemma \ref{Lemma:3-Scroll} below.
\end{proof}

\begin{lemma}
	\label{Lemma:3-Scroll}
	Let $Z=\PP(E)\subset \PP^N$ be a $3$-dimensional scroll over a smooth projective curve $B$ such that $\sO_{\PP(E)}(1)\cong \sO_{\PP^N}(1)|_Z$. Then we have
	\[
	\codeg(Z)=\deg(Z)=c_1(E).
	\]
\end{lemma}

\begin{proof}
	Let $H$ be a hyperplane section of $Z$ and denote by $\pi:\PP(E)\rightarrow B$ the natural projection. By \cite{BeltramettiFaniaSommese1992} (see also \cite[Theorem 6.1]{Tevelev2005}), we have
	\[
	\codeg(Z)= c_2(\sJ(H))\cdot H=c_1(\Omega_Z\otimes H)\cdot H^2 + c_2(\Omega_Z\otimes H)\cdot H
	\]
	where $\sJ(H)$ is the first jet bundle. By a straightforward computation, we get
	\[
	c_1(\Omega_Z)=\pi^*c_1(E)+\pi^*K_C-3H
	\]
	and 
	\[
	c_2(\Omega_Z)=-3\pi^*K_C\cdot H-2\pi^*c_1(E)\cdot H + 3H^2.
	\]
	As a consequence, we obtain $\codeg(Z)=c_1(E)\cdot H^2=H^3=\deg(Z)$.
\end{proof}

\begin{remark}
	Zak informed us a geometric proof of Lemma \ref{Lemma:3-Scroll} which is valid for scrolls of any dimension. We keep the proof here to indicate how to use the formula given in \cite{BeltramettiFaniaSommese1992} to compute the codegree of an arbitrary variety and this method will also be used in Lemma \ref{Lemma:Codegree-VMRT-F4/P3} to compute the codegree of the VMRT of $F_4/P_3$.
\end{remark}

Now we are in the position to prove Proposition \ref{Thm:Threefold-Small-Codegree}. 

\begin{proof}[Proof of Proposition \ref{Thm:Threefold-Small-Codegree}]
	Denote by $L$ the restriction of $\sO_{\PP^N}(1)|_Z$ and let $S\subset Z$ be a general hyperplane section. According to Lemma \ref{Lemma:3-Scroll}, we shall assume that $Z$ is not dual defective (cf. \cite[Example 7.6]{Tevelev2005}). By the codegree formula (cf. \cite[Proposition 1.1]{LanteriTurrini1987}), we have 
	\begin{equation}
		\label{Eq:Landman-Formula}
		d^*=(b_3(Z)-b_1(Z))+2(b_2(S)-b_2(Z))+2(g(S)-q(S)).
	\end{equation}
   Set $A=b_3(Z)-b_1(Z)$, $B=b_2(S)-b_2(Z)$ and $C=2(g(S)-q(S))$. Then both $A$ and $C$ are even non-negative integers and $B$ is a positive integer since $Z$ is not dual defective by our assumption (see \cite[Proposition 1.2 and 1.4]{LanteriTurrini1987}). If $C=0$, then we can conclude by Lemma \ref{Lemma:Surface-g=q} that $Z$ satisfies $d^*=d$. Hence, we may assume also that $C>0$ in the sequel. On the other hand, if $p_g(S)\not=0$, then it follows from \cite[Proposition 2.5]{LanteriTurrini1987} that we have $g(S)-q(S)\geq N-1$. In particular, we obtain
   \[
   d^*\geq 2B+2(g(S)-q(S))\geq 2+2(N-1)=2N.
   \]
   
   From now on, we shall assume that $B>0$, $C>0$ and $p_g(S)=0$. In particular, it follows \cite[Theorem]{Sommese1979} that $K_S+\bar{L}$ is globally generated and \cite[Proposition 2.1]{Sommese1979} implies that we have
   \[
   d=\bar{L}^2\leq K_S^2+4g(S)-4,
   \]
   where $\bar{L}$ is the restriction $L|_S$, and the equality holds if and only if $\Phi_{|K_S+\bar{L}|}$ is not generically finite, i.e., $\dim(\Phi_{|K_S+\bar{L}|}(S))\leq 1$. In particular, $S$ is unirueld and so is $Z$. On the other hand, by Noether's formula, we have
   \[
   K_S^2=12\chi(\sO_S)-\chi_{\rm top}(S)=10-8q(S)-h^{1,1}(S).
   \]
   This implies
   \[
   d\leq 6+4g(S)-8q(S)-h^{1,1}(S).
   \]
   Applying the codegree formula \eqref{Eq:Landman-Formula}, we get 
   \begin{align*}
   	d  & \leq 6+2(d^*-2B-A)-4q(S)-h^{1,1}(S)  \\
   	   & \leq 2d^*+6-4(B+q(S))-2A-h^{1,1}(S)  
   \end{align*}
   Note that $q(S)$ and $A$ are non-negative integers. Thus, since $B$ and $h^{1,1}(S)=b_2(S)$ are positive integers, it follows that we have $d\leq 2d^*$ unless the following condition happens
   \[
   	B=h^{1,1}(S)=b_2(S)=1\quad {\rm and}\quad A=q(S)=0.
   \]
   This is impossible since we have $b_2(S)>b_2(Z)\geq 1$ by our assumption. Moreover, an easy similar argument shows that if the equality $d=2d^*$ holds, then we must have
   \[
   B=1,\quad h^{1,1}(S)=b_2(S)=2\quad {\rm and}\quad A=q(S)=0.
   \]
   This implies that $\rho(Z)=b_2(Z)=1$. In particular, as $Z$ is uniruled, it follows that $Z$ is a Fano threefold of Picard number $1$. Note that $(K_Z+2L)|_S=K_S+\bar{L}$ is globally generated, but not big. This implies that $-K_Z=2L$. In particular, the pair $(Z,L)$ is isomorphic to either $(\PP^3,\sO_{\PP^3}(2))$ or a del Pezzo threefold. If $Z$ is a del Pezzo threefold, then $S$ is a del Pezzo surface with $b_2(S)=2$ and $-K_S=\bar{L}$. However, according to the classification of del Pezzo threefolds of Picard number $1$, we must have $d=\bar{L}^2=K_S^2\leq 5$. This implies that $b_2(S)\geq 4$, which is a contradiction. Hence, $2d^*=d$ if and only if $(Z,L)$ is isomorphic to $(\PP^3,\sO_{\PP^3}(2))$; that is, the projective variety $Z\subset \PP^N$ is projectively equivalent to either the second Veronese variety $\nu_2(\PP^3)\subset \PP^9$ or its isomorphic projection in $\PP^8$.
   
   Finally we assume that $Z$ satisfies the equality \eqref{Eq:Segre-Zak}. Then we have $5d^*=2(N+1)$. In particular, by our results above, we must have 
   \[
   \frac{4(N+1)}{5}=2d^*\geq d\geq N-2.
   \]
   This implies $N\leq 14$ and $d^*\leq 6$. Then, by the classification of smooth projective threefolds of codegree at most $6$ given in \cite{LanteriTurrini1987}, one can easily check that the only possibilities are the quadric threefold $\Q^3\subset \PP^4$ (with $d^*=2$) and the Veronese variety $\nu_2(\PP^3)\subset \PP^9$ (with $d^*=4$).
\end{proof}

\subsection{Proof of Theorem \ref{Thm:normalized-I}}

We start with  the following classification of del Pezzo surfaces with pseudoeffective normalized tangent bundle, which is easily deduced from  \cite[IV, Theorem 4.8]{Nakayama2004}.
\begin{theorem}
	\label{Thm:normalized-Surface}
	Let $S$ be a smooth del Pezzo surface, i.e. $-K_S$ is ample. Then the normalized tangent bundle of $S$ is pseudoeffective if and only if $S$ is isomorphic to the quadric surface $\PP^1\times \PP^1$.
\end{theorem}

\begin{proof}
	Note that $T_S$ is always semi-stable with respect to $-K_S$ by \cite{Fahlaoui1989} and its normalization is not nef by Theorem \ref{Thm:Almost-Nef-normalization}. On the other hand, since $S$ is simply connected, there does not exist non-trivial unramified coverings of $S$. In particular, applying \cite[IV, Theorem 4.8]{Nakayama2004} to $S$ and $T_S$, we see that only Case (A) and Case (C) of \cite[IV, Theorem 4.8]{Nakayama2004} may happen in our situation. In other words, either $T_S$ splits as a direct sum $L_1\oplus L_2$ as in Case (A) or $-K_S\equiv 2L$ for some line bundle $L$ on $S$ as in Case (C). In the latter case, it is easy to see that $S$ is isomorphic to the quadric surface from the classification of del Pezzo surfaces. In the former case, the surface $S$ is isomorphic to a product of curves (see for instance \cite[Theorem 1.4]{Hoering2007}). This implies immediately that $S$ is isomorphic to the product $\PP^1\times \PP^1$ as $S$ is rationally connected. 
\end{proof}

From now on, we will assume that $n \geq 3$. 
To prove Theorem \ref{Thm:normalized-I} , we start with the following: 

\begin{theorem}
	\label{Thm:I-II}
	Let $X$ be an $n$-dimensional Fano manifold of Picard number $1$ equipped a dominating family $\sK$ of minimal rational curves. Assume that the VMRT $\sC_x\subset \PP(\Omega_{X,x})$ at a general point $x\in X$ is not dual defective. If $\dim(\sC_x)\geq 1$ and $n\geq 3$, then $\codeg(\sC_x)\geq 2$ and the following statements hold.
	\begin{enumerate}
		\item If $\codeg(\sC_x)=2$, then $X$ is a smooth quadric hypersurface in $\PP^{n+1}$.
		
		\item If the normalized tangent bundle of $X$ is pseudoeffective and the VMRT $\sC_x$ is smooth with $\codeg(\sC_x)=3$, then $X$ is one of the following varieties: the Lagrangian Grassmann variety $\LG(3,6)$, the Grassmann variety $\Gr(3,6)$, the $15$-dimensional Spinor variety $\mathbb{S}_6$ and the $27$-dimensional $E_7$-variety $E_7/P_7$.
	\end{enumerate}
\end{theorem}

\begin{proof}
	By the biduality theorem, the dual variety $\check{\sC}_x$ does not contain hyperplanes as irreducible components since $\sC_x$ is purely dimensional, and hence $\codeg(\sC_x)\geq 2$. Let us denote by $\PP^m=\PP(W)\subset \PP(\Omega_{X,x})$ the linear span of $\sC_x$.
	
	Firstly we assume that $\codeg(\sC_x)=2$, i.e., the dual $\check{\sC}_x\subset \PP(T_{X,x})$ is an irreducible quadric hypersurface of $\PP(T_{X,x})$. Then the VMRT $\sC_x$ itself is irreducible. On the other hand, if $\check{\sC}_x$ is not smooth, then it is an irreducible quadric cone. According to the biduality theorem, since $\sC_x$ is not dual defective, the VMRT $\sC_x$ is a smooth quadric hypersurface in $\PP^m\subset \PP(\Omega_{X,x})$. Then it follows from \cite[Proposition 2.4 and Proposition 2.6]{Hwang2001} that we must have $\PP^m=\PP(\Omega_{X,x})$.  Therefore, by Theorem \ref{t.IHSSMok}, the variety $X$ is isomorphic to a quadric hypersurface.
	
	Next we assume that $\sC_x$ is smooth with $\codeg(\sC_x)=3$ and the normalized tangent bundle of $X$ is pseudoeffective. Then $\check{\sC}_x$ is an irreducible hypersurface of degree $3$ and hence $\sC_x$ is irreducible and smooth. By Zak's classification of linearly non-degenerate smooth varieties with codegree $3$ \cite[Theorem 5.2]{Zak1993}, we obtain that $\dim(\sC_x)\geq 2$ and
	\[
	\dim(\sC_x)>\frac{m-1}{2}
	\]
	unless $\sC_x\subset \PP^m$ is $\nu_2(\PP^2)\subset \PP^5$. On the other hand, it can be directly checked that the tangential variety of $\nu_2(\PP^2)\subset \PP^5$ is linearly non-degenerate. Therefore, it follows from \cite[Proposition 2.6]{Hwang2001} that the tangential variety of $\sC_x$ is linearly non-degenerate. Then  we can apply \cite[Proposition 2.4]{Hwang2001} to obtain that $\PP^m=\PP(\Omega_{X,x})$. In particular, as $\sC_x$ is assumed to be not dual defective, it follows from Proposition \ref{Prop:Criterion-Pseff-Normalized} that $\sC_{x}\subset \PP(\Omega_{X,x})$ is projectively equivalent to one of the four Severi varieties. Then one can apply  Theorem \ref{t.IHSSMok} to conclude that $X$ is isomorphic to one of the four varieties in the theorem.
	\end{proof}

Comparing with Proposition \ref{Prop:Criterion-Pseff-Normalized}, we do not require that the VMRT of $X$ at a general point is irreducible or linearly non-degenerate in Theorem \ref{Thm:I-II} above.

\begin{proof}[Proof of Theorem \ref{Thm:normalized-I}]

%The proof of (1.1) is similar to that of  Proposition \ref{Prop:Criterion-Pseff-Normalized}.
	For the statement (1), we assume that the VMRT $\sC_x\subset \PP(\Omega_{X,x})$ is irreducible, linearly non-degenerate and not dual defective. Let $\check{\sC}\subset \PP(T_X)$ be the total dual VMRT. Write $[\check{\sC}]\equiv a\Lambda - b\pi^*H$, where $\Lambda$ is the tautological divisor of the projectivized tangent bundle $\pi:\PP(T_X)\rightarrow X$ and $H$ is the ample generator of $\pic(X)$. Let $i_X$ be the index of $X$. Without loss of generality, we may assume that $\alpha(X,-K_X)>0$, or equivalently $T_X$ is big. In particular, by Theorem \ref{Thm:Class-dual-VMRT}, we have $b>0$ and 
	\[
	\alpha(X,-K_X) = \frac{b}{a i_X} =\frac{b}{i_X \codeg(\sC_x)}\leq \frac{b}{i_X}\cdot \frac{\dim(\sC_x)+2}{2\dim(X)}.
	\]
	The last inequality follows from the Segre inequality \eqref{Eq:Segre-Zak-Ineq}. In particular, note that we have $\dim(\sC_x)+2=-K_X\cdot \sK = i_X H \cdot \sK$ and $bH\cdot \sK\leq 2$, thus we get
	\[
	\alpha(X,-K_X)\leq \frac{1}{\dim(X)}
	\]
	with equality only if the VMRT $\sC_x\subset \PP(\Omega_{X,x})$ satisfies the equality \eqref{Eq:Segre-Zak}. Hence, if the normalized tangent bundle of $X$ is pseudoeffective and Conjecture \ref{Conj:Zak-Russo-Conjecture} holds, then the VMRT $\sC_x\subset \PP(\Omega_{X,x})$ is projectively equivalent to one of the varieties listed in Conjecture \ref{Conj:Zak-Russo-Conjecture} and we then conclude by Theorem \ref{t.IHSSMok} and Table \ref{Table:IHSSVMRT}.
	
	For the statement (2), assume that the VMRT $\sC_x\subset \PP(\Omega_{X,x})$ is not dual defective and the normalized tangent bundle of $X$ is pseudoeffective. By Proposition \ref{Prop:Criterion-Pseff-Normalized}, if the condition (2.1) (resp. condition (2.2)) holds, then we get $\codeg(\sC_x)<3$ (resp. $\codeg(\sC_x)<4$) and the results follows from Theorem \ref{Thm:I-II} above. If the condition (2.3) holds, then it is clear that $\dim(\sC_x)\geq 1$ as the VMRT can not be a single point. Then the result follows from the statement  (1.2), Proposition \ref{Thm:Threefold-Small-Codegree} and Theorem \ref{Thm:Classification-Small-Codegree}.
\end{proof}

\begin{proof}[Proof of Corollary \ref{c.less-than-5}]
	By assumption, the VMRT $\sC_x\subset \PP(\Omega_{X,x})$ is either a non-linear smooth curve, or a non-linear smooth surface, or a non-linear smooth hypersurface $(n=5)$. In particular, the VMRT $\sC_x$ is not dual defective by \cite[Example 1.19 and Example 7.5 and Theorem 4.25]{Tevelev2005}. Then the result follows directly from Theorem \ref{Thm:normalized-I} (2.2).
\end{proof}

\begin{proof}[Proof of Corollary \ref{c.less-than-11}]
	The result follows from Theorem \ref{Thm:normalized-I} (2.2) and (2.3).
\end{proof}

\section{Rational homogeneous spaces}
\label{Section:RHS}

Throughout this section, for a vector bundle $E$ over a variety $X$, we denote by $\mathbf{P}(E)$ the projective bundle over $X$, whose fiber over $x \in X$ is the set of lines in $E_x$.   It is isomorphic to $\mathbb{P}(E^*)$ in our previous notation. Moreover, all the varieties in this section are assumed to have dimension at least $2$. The main aim of this section is to calculate the cones of divisor of $\mathbf{P}(T^*_{G/P}) = \PP (T_{G/P})$ for a rational homogeneous space $G/P$ with Picard number $1$. 

\subsection{Springer maps}

Let $G$ be a complex simple Lie algebra and let $\mathfrak{g}$ be its Lie algebra. Then $G$ has the adjoint action on $\mathfrak{g}$. The  orbit $\sO_x$ of a nilpotent element $x\in \mathfrak{g}$ is called a \emph{nilpotent orbit}, which is invariant under the dilation action of $\mathbb{C}^*$ on $\mathfrak{g}$. For any parabolic subgroup $P$ of $G$,  the group $G$ has a Hamiltonian action on the cotangent bundle $T^*_{G/P}$ and the image of the moment map 
$ T^*_{G/P}\longrightarrow \mathfrak{g} \simeq \mathfrak{g}^*$
is a nilpotent orbit closure $\overline{\sO}$, which will be called the \emph{Richardson orbit} associated to  $P$. The induced morphism
 $$\widehat{s}: T^*_{G/P}\rightarrow \overline{\sO}$$will be called the \emph{Springer map} associated to $P$, which is a generically finite $G \times \mathbb{C}^*$-equivariant  projective morphism. We denote by
\[
T^*_{G/P} \xrightarrow{\widehat{\varepsilon}} \widetilde{\sO} \xrightarrow{\widehat{\tau}} \overline{\sO}
\]
the Stein factorisation of $\widehat{s}$.   It follows that $\widehat{\varepsilon}$ is birational and $\widehat{\tau}$ is a finite morphism.  Note that $\widehat{s}^{-1}(0) = G/P$ is irreducible,  the pre-image $\widehat{\tau}^{-1}(0)$ is a single point in $\widetilde{\sO}$.  This implies that the projectivised Springer map 
$$
s:  \mathbf{P} (T^*_{G/P})\rightarrow  \mathbf{P}(\overline{\sO})
$$
has the Stein factorization given by 
\[
\mathbf{P} (T^*_{G/P}) \xrightarrow{\varepsilon}  \mathbf{P}(\widetilde{\sO}) \xrightarrow{\tau} \mathbf{P}(\overline{\sO}).
\]

From now on, we shall assume that $G/P$ is a rational homogeneous space with Picard number $1$; that is, $P$ corresponds to a \emph{single-marked Dynkin diagram}.

\begin{example}
	Given an $(n+1)$-dimensional complex vector space $V$, the rational homogeneous spaces for the group $SL_{n+1}=SL_{n+1}(V)$, are determined by the different markings of the Dynkin diagram $A_n$. For instance, the Grassmann variety $\Gr(k,n+1)$ corresponds to the marking of the $k$-th node:
	\[
	\dynkin[labels={1,2,3, ,k, , , ,n}]
		 A{o3.o*o.o3}
	\]
\end{example}

\begin{proposition} \label{p.stratiMukai}
	\cite[Propostion 5.1]{Namikawa2006}
	Assume $G/P$  is  of Picard number $1$. Then the projectivized Springer map  $s: \mathbf{P} (T^*_{G/P})\rightarrow  \mathbf{P}(\overline{\sO})$ is birational and  small if and only $G/P$ is one of the following.

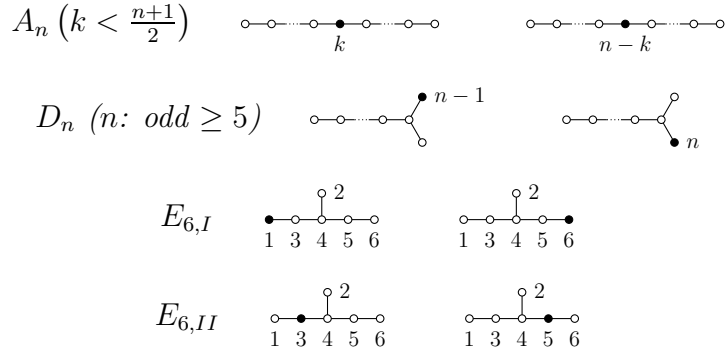
\begin{center}
 $A_{n}  \left(k< \frac{n+1}{2}\right)$  \quad 	\dynkin[labels={, , , k, , , }]A{o2.o*o.o2} \qquad
	\dynkin[labels={, , , n-k, , , }]A{o2.o*o.o2}
\end{center}

\begin{center}
	 $D_{n}$ ($n$: odd $\geq 5$) \quad \dynkin[labels={, , , ,n-1, }]D{o2.o2*o}\qquad
	\dynkin[labels={, , , , , n}]D{o2.o2o*}
\end{center}

\begin{center}
 $E_{6,I}$ \quad 	\dynkin[labels={1,2,3,4,5,6}]E{*ooooo}\qquad
	\dynkin[labels={1,2,3,4,5,6}]E{ooooo*}
\end{center}

\begin{center}
 $E_{6,II}$ \quad 	\dynkin[labels={1,2,3,4,5,6}]E{oo*ooo}\qquad
	\dynkin[labels={1,2,3,4,5,6}]E{oooo*o}
\end{center}
Furthermore the pair  $(P, Q)$ in each group has the same Richardson orbit and the corresponding Springer maps give a birational map $\widehat{\mu}: T^*_{G/P} \dasharrow T^*_{G/Q}$, which is called the stratified Mukai flop of type $A_{n,k}$ (resp. $D_n, E_{6, I}, E_{6, II}$) according to the types of corresponding marked Dynkin diagram. In this case, there exists a (non-canonical) isomorphism $G/P \simeq G/Q$.
\end{proposition}

\begin{proposition} \label{p.stratiMukai2}
	\cite[Propostion 3.1]{Namikawa2008}
	Assume $G/P$  is  of Picard number $1$. Then the   birational contraction $\varepsilon: \mathbf{P} (T^*_{G/P}) \to \mathbf{P}(\widetilde{\sO}) $ is small if and only if  either  $G/P$ is as in Proposition \ref{p.stratiMukai}  or  $\widehat{s}$ has degree 2 and $G/P$ is one of the following. 
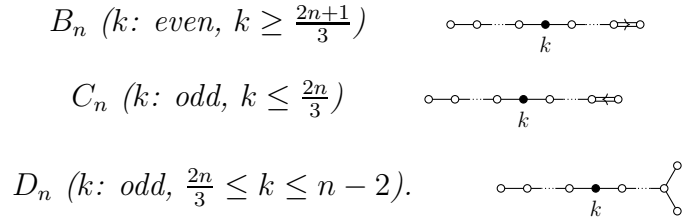
\begin{center}
$B_n$ ($k$: even, $k\geq \frac{2n+1}{3}$) \qquad 	\dynkin[labels={,,,k,,,}]B{o2.o*o.oo}
\end{center}

\begin{center}$C_n$ ($k$: odd, $k\leq \frac{2n}{3}$) \qquad
	\dynkin[labels={,,,k,,,}]C{o2.o*o.o2}
\end{center}

\begin{center}$D_n$ ($k$: odd, $\frac{2n}{3}\leq k\leq n-2$). \qquad
	\dynkin[labels={,,,k,,,,}]D{o2.o*o.o3}
\end{center}
In the latter case, by interchanging the two points in  general fibers of $\widehat{s}$, this gives a stratified Mukai flop of type $B_{n,k}$ (resp. $C_{n,k}, D_{n,k}$)  $\hat{\mu}: T^*_{G/P} \dasharrow T^*_{G/P}$ according to the types of corresponding marked Dynkin diagram.

\end{proposition}

We will describe in details these flops in Section \ref{s.GeomMukaiFlop}.

\begin{proposition} \label{p.birSpringer}
The Springer map  $\widehat{s}: T^*_{G/P} \to \overline{\mathcal{O}}$ is not birational if and only if $G/P$ is as in Proposition \ref{p.stratiMukai2} or  $G/P$ is 
$G_2/P_1$ or  $F_4/P_3$ with ${\rm deg}(\widehat{s})$ being $2$ and $4$ respectively.
\end{proposition}
\begin{proof}
For classical cases, this follows from the proof of \cite[Propostion 3.1]{Namikawa2008}.  For exceptional cases, 
assume $G$ is of exceptional type. In most cases $\mathcal{O}$ is an even orbit or an orbit with trivial fundamental group, which implies that $\widehat{s}$ is birational. 
For the remaining cases,  the degree is computed in \cite[Appendix]{Fu07}.
\end{proof}

\subsection{Cones of divisors}
%Firstly let us fix some notations. Let $X$ be a projective manifold such that $T_X$ is big and semiample. Denote by $\varepsilon:\sX:=\PP(T_X)\rightarrow \sY$ the birational morphism defined by the globally generated line bundle $\sO_{\PP(T_X)}(m)$ with $m$ large enough. Denote by $\pi:\PP(T_X)\rightarrow X$ be the natural projection, and let $\Lambda$ be the tautological divisor of the projectivised tangent bundle $\PP(T_X)$.

%Let $G/P$ be a rational homogeneous space of Picard number $1$. Denote by $H$ the ample generator of $\pic(G/P)$ and by $\alpha_{G/P}$ the pseudoeffective threshold $\alpha(G/P,H)$.
We start with the following result which describes the cones of divisors on $\mathbf{P} (T^*_{G/P}) = \mathbb{P} (T_{G/P})$.
\begin{theorem}
	\label{Thm:Rational-Homogeneous-Spaces}
	Let $G/P$ be a rational homogeneous space of Picard number $1$, but not a  projective space. Denote by $H$ the ample generator of $\pic(G/P)$. Let $\Gamma$ be the exceptional locus of $\varepsilon:  \mathbf{P} (T^*_{G/P}) \to  \mathbf{P}(\widetilde{\sO}) $. Then the following statements hold.
	\begin{enumerate}
		\item If $\Gamma$ has codimension at least $2$, i.e., $\varepsilon$ is small, then there exists a commutative diagram
		\[
		\begin{tikzcd}[column sep=large, row sep=large]
			\mathbf{P} (T^*_{G/P}) \arrow[rr,dashed, "\mu"] \arrow[rd,"\varepsilon"] \arrow[d,"\pi_1"'] 
			&
			&  \mathbf{P} (T^*_{G/Q}) \arrow[ld,"\varepsilon"'] \arrow[d,"\pi_2"] \\
			G/P        
			&   \mathbf{P}(\widetilde{\sO})  
			&  G/Q
		\end{tikzcd}
		\]
		where $\mu: \mathbf{P} (T^*_{G/P}) \dashrightarrow \mathbf{P} (T^*_{G/Q})$  is a non-isomorphic flop with $G/P \simeq G/Q$. In particular, we have
		\[
		\Eff(\mathbf{P} (T^*_{G/P}))=\overline{\rm Eff}(\mathbf{P} (T^*_{G/P}))=\Mov(\mathbf{P} (T^*_{G/P}))=\overline{\rm Mov}(\mathbf{P} (T^*_{G/P}))=\left\langle [\mu^*\pi_2^*H],[\pi_1^*H]\right\rangle.
		\]
		
		\item If $\Gamma$ has codimension $1$, i.e., $\varepsilon$ is divisorial, then $\Gamma$ is a prime divisor such that 
		\[
		\Mov(\mathbf{P} (T^*_{G/P}))=\overline{\rm Mov}(\mathbf{P} (T^*_{G/P}))=\left\langle [\Lambda], [\pi^*H] \right\rangle
		\]
	    and
	    \[
	    \rm{Eff}(\mathbf{P} (T^*_{G/P}))=\overline{\rm Eff}(\mathbf{P} (T^*_{G/P}))=\left\langle [\Gamma],[\pi^*H] \right\rangle.
	    \]
	\end{enumerate}
\end{theorem}

%Theorem \ref{Thm:Rational-Homogeneous-Spaces} follows directly from the results proved in Section \ref{Section:Semisimple} and the existence of the so-called stratified Mukai flops.

\begin{proof}
	Since $G/P$ is not isomorphic to projective spaces, the birational contraction $\varepsilon: \mathbf{P} (T^*_{G/P}) \to \mathbf{P}(\widetilde{\sO}) $ is not an isomorphism and $\Nef(\mathbf{P} (T^*_{G/P}) )=\langle [\Lambda],[\pi^*H]\rangle$. Let $\alpha(G/P, H)$ be the pseudoeffective threshold of $G/P$ with respect to $H$, then 
	 $[\Lambda-\alpha(G/P, H)\pi^*H]$ generates an extremal ray of $\overline{\Eff}(\mathbf{P} (T^*_{G/P}) )$. The statement (2) follows directly from Proposition \ref{Prop:Pseff-Cone-Divisorial-Contraction}. Thus it remains to prove the statement (1).
	
	By our assumption, the birational contraction $\varepsilon$ is small.  By Proposition \ref{p.stratiMukai} and  Proposition \ref{p.stratiMukai2}, there exists a flop $\mu:   \mathbf{P} (T^*_{G/P}) \dashrightarrow  \mathbf{P} (T^*_{G/Q})$ , with $G/P \simeq G/Q$ as projective varieties.  It follows that the pull-back $\mu^*\Nef( \mathbf{P} (T^*_{G/Q}))$ is contained in $\overline{\Mov}( \mathbf{P} (T^*_{G/P}))$ (cf. Lemma \ref{Lemma:Big+Nef=MDS}). Moreover, it is clear that we have $\mu^*\Lambda=\Lambda$ since $-K_{\mathbf{P} (T^*_{G/P})}=n\Lambda$ and $\mu$ is a SQM. This implies that the pull-back $\mu^*\Nef(\mathbf{P} (T^*_{G/P}))$ is contained in the cone $\langle [\Lambda-\alpha(G/P, H)\pi_1^*H], [\Lambda]\rangle$. Nevertheless, as $\pi_2^*H$ is not big, it follows that $[\mu^*\pi_2^*H]$ is not contained in the interior of $\overline{\Eff}(\mathbf{P} (T^*_{G/P}))$. So we get
	\[
	\overline{\Eff}(\mathbf{P} (T^*_{G/P}))=\overline{\Mov}(\mathbf{P} (T^*_{G/P}))=\langle [\mu^*\pi_2^*H], [\pi_1^*H]\rangle.
	\]
	On the other hand, as $\pi_2^*H$ is globally generated and $\mu$ is a SQM, the stable base locus $\bB(\mu^*\pi_2 H)$ has codimension at least $2$. Hence, we obtain
	\[
	\overline{\Eff}(\mathbf{P} (T^*_{G/P}))=\Eff(\mathbf{P} (T^*_{G/P}))=\overline{\Mov}(\mathbf{P} (T^*_{G/P}))=\Mov(\mathbf{P} (T^*_{G/P})).
	\]
	This finishes the proof.
\end{proof}

While Theorem \ref{Thm:Rational-Homogeneous-Spaces} already gives a very nice geometric description of the cones of divisors of $\mathbf{P} (T^*_{G/P})$, it is not very easy to apply it to compute explicitly the cones in terms of $\Lambda$ and $\pi^*H$.   We introduce the following notion to divide $G/P$ into several types in order to carry out this computation.

\begin{definition} \label{d.types}
	Let $G/P$ be a rational homogeneous space of Picard number $1$ corresponding to a single marked Dynkin diagram.
	\begin{enumerate}
		\item $G/P$ is said of the first type (I) if $s$ is a birational small morphism (cf. Proposition \ref{p.stratiMukai}).
			
		\item $G/P$ is said of type (II-s) if $s$ is not birational and $\varepsilon$ is small (cf. Proposition \ref{p.stratiMukai2}).
			
		\item $G/P$ is said of type (II-d-d) if  $\varepsilon$ is divisorial and  the VMRT of $G/P$ is dual defective.

		\item $G/P$ is said of type (II-d-A1) (resp. (II-d-A2) )   if  $\varepsilon$ is divisorial but the VMRT of $G/P$ is not dual defective,  and  $\mathbf{P}(\widetilde{\mathcal{O}})$ is of type $A_1$  (resp. $A_2$) (cf. Definition \ref{Def:Singularities-A1-A2}).
					
	\end{enumerate}
\end{definition}

\begin{remark}
Recall there are following isomorphisms between different rational homogeneous spaces:  $C_n/P_1 \simeq A_{2n}/P_1, B_n/P_n \simeq D_{n+1}/P_n$ and $G_2/P_1 \simeq B_3/P_1$.  Their types are the same except  the following cases: $C_n/P_1$ is of type (II-s), while $A_{2n}/P_1$ is of type (I); and for $n$ even,  $B_n/P_n$ is of type (II-s)   while $D_{n+1}/P_n$ is of type (I). In fact, note that that $C_n/P_1 \simeq \mathbb{P}^{2n-1}$.  Let $\mathcal{O}_{min} \subset \mathfrak{sl}_{2n}$ be the minimal nilpotent orbit (corresponding to the partition $[2, 1^{2n-2}]$), then there exists a generically 2-to-1 morphism $\overline{\mathcal{O}}_{min} \to  \mathcal{O}_{\bf d}$.
The flop $C_{n,1}$ is nothing else but the Mukai flop $A_{2n,1}$. Moreover, by \cite[Example 3.3]{Namikawa2008}, $B_{n, n}$-flop is the same as $D_{n+1}$-flop for $n$ even.
\end{remark}

\begin{proposition}
	\label{Prop:Pseff-Small}
	Under the notation and assumption as in Theorem \ref{Thm:Rational-Homogeneous-Spaces}. Assume that $\varepsilon$ is small. Let $\ell_i$ be a general line in a general fibre of $\pi_i$. If $a(H)$ and $b(H)$ are the unique positive integers such that
	\[
	[\mu^*\pi_2^*H] \equiv a(H)[\Lambda] - b(H)[\pi_1^*H],
	\]
	then we have
	\[
	a(H)=\pi_2^*H\cdot \mu_*(\ell_1)\quad {\rm and}\quad a(H)-b(H)\pi_1^*H\cdot \mu^{-1}_*(\ell_2)=0.
	\]
	Moreover, the morphism $\mu(\ell_1)\rightarrow \pi_2(\mu(\ell_1))$ is birational. In particular, we have
	\[
	a(H)=H\cdot \pi_{2*}\mu_*(\ell_1).
	\]
\end{proposition}

\begin{proof}
	Since $\mu$ is a SQM and $\ell_i$'s are general, we may assume that both $\mu$ and $\mu^{-1}$ are isomorphisms in a neighbourhood of $\ell_i$. In particular, we have
	\[
	a(H) = (a(H)\Lambda - b(H) \pi^*_1 H)\cdot \ell_1 = \mu^*\pi_2^*H\cdot \ell_1 = \pi_2^*H\cdot \mu_*(\ell_1)
	\]
	and
	\[
	0=\pi_2^*H\cdot \ell_2 = \mu^*\pi_2^*H \cdot \mu^{-1}_*(\ell_2) = a(H) - b(H)\pi_1^*H\cdot \mu^{-1}_*(\ell_2).
	\]
	Here we note that $\mu^*\Lambda=\Lambda$ and $\Lambda \cdot \ell_i=1$. Now it remains to show that the morphism $\mu(\ell_1)\rightarrow C:=\pi_2(\mu(\ell_1))$ is birational. Let $f:\PP^1\rightarrow C$ be the normalization. As $T_{G/Q}$ is nef, there exist integers $a_1\geq\dots\geq a_k>a_{k+1}=\dots=a_n=0$ (with $k\leq n$) such that
	\[
	f^*T_{G/Q}\cong \bigoplus_{i=1}^k\sO_{\PP^1}(a_i)\oplus \sO_{\PP^1}^{\oplus (n-k)} 
	\]
	Denote by $d$ the degree of $\mu(\ell_1)\rightarrow C$. Then we have
	\[
	\pi_2^*T_{G/Q}|_{\mu(\ell_1)} \cong \bigoplus_{i=1}^k\sO_{\PP^1}(d a_i) \oplus \sO_{\PP^1}^{\oplus (n-k)}.
	\]
	As $\Lambda \cdot \mu(\ell_1)=1$, if $d\geq 2$, then $\mu(\ell_1)$ is contained in
	\[
	\PP(\sO_{\PP^1}^{\oplus (n-k)}) \subset \PP(f^*T_{G/P}).
	\]
	On the other hand, as $\PP(\sO_{\PP^1}^{\oplus (n-k)})$ is dominated by curves with $\Lambda$-degree $0$, thus $\PP(\sO_{\PP^1}^{\oplus (n-k)})$ is contained in the exceptional locus of $\varepsilon$ and so is $\mu(\ell_1)$, which is absurd.
\end{proof}

\begin{proposition}
	\label{Prop:Pseff-Divisor}
	Under the notation and assumption as in Theorem \ref{Thm:Rational-Homogeneous-Spaces}. Assume that $\varepsilon: \mathbf{P} (T^*_{G/P}) \to \mathbf{P}(\widetilde{\sO})$ is divisorial. Let $a(\Gamma)$ and $b(\Gamma)$ be the unique positive integers such that 
	\[
	[\Gamma]\equiv a(\Gamma)[\Lambda] - b(\Gamma)\pi^*H.
	\]
	Then the following statements hold.
	\begin{enumerate}
		\item The projective variety $\mathbf{P}(\widetilde{\mathcal{O}})$ is a $\Q$-factorial variety of Picard number $1$. Moreover, let $\Lambda'$ and $H'$ be the push-forward of $\Lambda$ and $\pi^*H$ by $\varepsilon$. Then we have
		\[
		\frac{b(\Gamma)}{a(\Gamma)}=\frac{\Lambda^{2n-1}}{\Lambda^{2n-2}\cdot \pi^*H}\quad {\rm and}\quad H'\equiv \frac{a(\Gamma)}{b(\Gamma)}\Lambda'.
		\]
		
		\item If $G/P$ is of type (II-d-A1) or (II-d-A2), then $a(\Gamma)=\codeg(\sC_o)$, where $\sC_o$ is the VMRT of $G/P$ at a referenced point $o\in G/P$.
		
		\item $b(\Gamma)\leq 2$ with equality if and only if $G/P$ is of type (II-d-A1).

		\item  $G/P$ is of type (II-d-A2) if and only if  $\mathbf{P}(\widetilde{\mathcal{O}})$ has $cA_2$-singularities in codimension $2$.
	\end{enumerate}
\end{proposition}

\begin{proof}
	Firstly note that the morphism $\varepsilon$ is a Mori extremal contraction with respect to a klt pair $(\mathbf{P}(T^*_{G/P}),\Delta)$ (see \cite[Proposition 5.5]{MunozOcchettaSolaCondeWatanabeEtAl2015}). Thus, as $\rho(\mathbf{P}(T^*_{G/P}))=2$, $\varepsilon$ is divisorial and $\mathbf{P}(T^*_{G/P})$ is rationally connected, it follows that $\mathbf{P}(\widetilde{\mathcal{O}})$ is a $\Q$-factorial Fano variety of Picard number $1$. Moreover, note that we have $\Gamma\cdot \Lambda^{2n-2}=0$. This implies immediately
	\[
	\frac{b(\Gamma)}{a(\Gamma)}=\frac{\Lambda^{2n-1}}{\Lambda^{2n-2}\cdot \pi^*H}.
	\]
	Let $\widetilde{H}$ be a general member in $|\pi^*H|$ and set $H'=\varepsilon_*\widetilde{H}$. As $\mathbf{P}(\widetilde{\mathcal{O}})$ is $\Q$-factorial, there exists a rational number $r$ such that $H'\equiv r \Lambda'$. Moreover, by the negativity lemma, there exists a non-negative rational number $\alpha$ such that
	\[
	\varepsilon^*H' \equiv_{\Q} \widetilde{H} + \alpha \Gamma.
	\]
	As $\varepsilon^*\Lambda'=\Lambda$, we obtain
	\[
	r\Lambda \equiv \varepsilon^*H' \equiv \pi^*H + \alpha(a(\Gamma)\Lambda-b(\Gamma)\pi^*H).
	\]
	Since $\Lambda$ and $\pi^*H$ are linearly independent, comparing the coefficients shows that we have
	\[
		\alpha b(\Gamma)=1 \quad {\rm and}\quad  \alpha a(\Gamma)=r.
	\]
	This implies $r=a(\Gamma)/b(\Gamma)$ and the statement (1) is proved.
	
	If $G/P$ is of types (II-d-A1) or (II-d-A2), the the total dual VMRT is a divisor. It follows from Corollary \ref{Cor:Equality-ABL-RBL}, Theorem \ref{Thm:Class-dual-VMRT} and Proposition \ref{Prop:Pseff-Cone-Divisorial-Contraction} that we have $\check{\sC}=\Gamma$ and hence $a(\Gamma)=\codeg(\sC_o)$.
	
	For the statement (3), by Proposition \ref{Prop:Pseff-Cone-Divisorial-Contraction}, we have $b(\Gamma)\leq 2$ with equality if and only if $\mathbf{P}(\widetilde{\mathcal{O}})$ is of type $A_1$ and there exists a dominating family $\sK$ of minimal rational curves on $G/P$ such that $\check{\sC}=\Gamma$ and $H\cdot \sK=1$. Note that in our situation, there exists only one dominating family $\sK$ of minimal rational curves on $G/P$ and $H\cdot \sK=1$. Thus $b(\Gamma)=2$ if and only $\check{\sC}=\Gamma$ and $\mathbf{P}(\widetilde{\mathcal{O}})$ is of type $A_1$. The latter conditions are equivalent to say that $G/P$ is of type (II-d-A1) by definition.
	
	For the statement (4), if $G/P$ is of type (II-d-A2), it follows from definition that $\mathbf{P}(\widetilde{\mathcal{O}})$ has $cA_2$ singularities in codimension $2$. Conversely, from the proof of Proposition \ref{Prop:Pseff-Cone-Divisorial-Contraction}, it is known that $b(\Gamma)\pi^*H\cdot F\leq 2$ with equality if and only if $\Gamma$ is smooth along $F$, where $F$ is an irreducible component of a general fibre of $\Gamma\rightarrow \varepsilon(\Gamma)$. In particular, if $\mathbf{P}(\widetilde{\mathcal{O}})$ has $cA_2$-singularities in codimension $2$, then we must have $b(\Gamma)=\pi^*H\cdot F=1$. Then Claim 2 in the proof of Proposition \ref{Prop:Pseff-Cone-Divisorial-Contraction} implies that $\check{\sC}=\Gamma$ and consequently $G/P$ is of type (II-d-A2).
\end{proof}

As an immediate application of Proposition \ref{Prop:Pseff-Divisor}, one can easily derive the following result.  

\begin{corollary} \label{c.G/Ptypes}
		Under the notation and assumption as in Theorem \ref{Thm:Rational-Homogeneous-Spaces}. Assume that $\varepsilon$ is divisorial. Then the following statements hold.
	\begin{enumerate}		
		\item $G/P$ is of type (II-d-d) if and only if $\check{\sC}\not=\Gamma$, and if and only if $b(\Gamma)=1$ and $\mathbf{P}(\widetilde{\mathcal{O}})$ has only $cA_1$ singularities in codimension $2$.
		
	    \item $G/P$ is of type (II-d-A1) if and only if $b(\Gamma)=2$ and $\check{\sC}=\Gamma$, and if and only if $\mathbf{P}(\widetilde{\mathcal{O}})$ has  $cA_1$-singularities in codimension $2$ and $\check{\sC}=\Gamma$.
	    
	    \item $G/P$ is of type (II-d-A2) if and only if $b(\Gamma)=1$ and $\check{\sC}=\Gamma$, and if and only if $\mathbf{P}(\widetilde{\mathcal{O}})$ has $cA_2$-singularities in codimension $2$.
	\end{enumerate}
\end{corollary}

\subsection{Types of rational homogeneous spaces}

By Proposition \ref{Prop:Pseff-Small} and Proposition \ref{Prop:Pseff-Divisor} in the previous subsection,  to compute $a(E)$, $b(E)$, $a(H)$ and $b(H)$, we need to determine the types of $G/P$.  Proposition \ref{p.stratiMukai} and Proposition \ref{p.stratiMukai2} give respectively the classification of $G/P$ of  type (I)  and  type (II-s). In this subsection, we will determine the types of all other $G/P_k$,  where $P_k$ is the maximal parabolic subgroup associated to the $k$-th simple root of $G$.

The VMRT $\mathcal{C}_o$  of $G/P_k$ is determined in \cite[Theorem 4.8]{LandsbergManivel2003}, which is again a rational  homogeneous space if $P_k$ corresponds to a long root.  When $P_k$ corresponds to a short root, $\mathcal{C}_o$ is a two-orbit variety. The embedding $\mathcal{C}_o \subset \mathbf{P}(T_{G/P,o})$ is in general degenerated and the dual defect of $\mathcal{C}_o$ can be computed from the following  when it is homogeneous (cf.  \cite{Snow1993}, \cite[Theorem 7.54 and Theorem 7.56]{Tevelev2005}).
\begin{proposition}\label{p.dualdectiveG/P}
Let $G/P \subset \mathbb{P}^N$ be the minimal $G$-equivariant embedding. Then it is dual defective if and only if $G/P$ is one of the following:
\begin{itemize}
\item[(a)]  $ \mathbb{P}^n$ with ${\rm def} = n$.
\item[(b)] ${\rm Gr}(2, 2m+1) $ with ${\rm def} = 2$.
\item[(c)] the 10-dimensional spinor vareity $\mathbb{S}_5$ with ${\rm def} = 4$.
\item[(d)] a product  $G_1/P_1 \times G_2/P_2$ with $G_1/P_1$ as above such that ${\rm def}(G_1/P_1) > \dim G_2/P_2$. In this case, the dual defect is   ${\rm def} = {\rm def}(G_1/P_1) - \dim G_2/P_2 $.
\end{itemize}
\end{proposition}

\begin{proposition}
	\label{Prop:Defect-short-root}
	Let $X=G/P_k$ be a rational homogeneous space such that $P_k$ corresponds to a short root. Then the VMRT $\sC_o$ of $X$ is dual defective if and only if $X$ is one of the following:
	\[
	B_n/P_n\ (n\geq 3\ {\rm odd}),\quad C_n/P_k\ (2n\geq 3k)\quad \text{and}\quad F_4/P_4.
	\]
\end{proposition}

\begin{proof}
	If $X=G/P_k$ is one of the following: $B_n/P_n$, $C_n/P_1$ and $G_2/P_1$, then it  is isomorphic respectively to $D_{n+1}/P_{n+1}$, $A_{2n-1}/P_1$ and $B_3/P_1$. In particular, the VMRT of $X$ is still a rational homogeneous space in these cases and we can apply Proposition \ref{p.dualdectiveG/P}. If $X=G/P_k$ is the rational homogeneous space of type $C_n/P_k$ with $k\geq 2$, it is shown in Lemma \ref{Lemma:C_n-VMRT} that the VMRT $\sC_o\subset \PP(\Omega_{X,o})$ is dual defective if and only if $2n\geq 3k$. If $X=G/P_k$ is the variety $F_4/P_3$, then it is shown in Lemma \ref{Lemma:Codegree-VMRT-F4/P3} that the VMRT $\sC_o\subset \PP(\Omega_{X,o})$ is not dual defective with codegree $8$. If $X=G/P_k$ is the variety $F_4/P_4$, then the VMRT $\sC_o\subset \PP(\Omega_{X,o})$ is a hyperplane section of $\mathbb{S}_5\subset \PP^{15}$. Recall that the dual defect of $\mathbb{S}_5\subset \PP^{15}$ is equal to $4$, thus the dual defect of the VMRT $\sC_o\subset \PP(\Omega_{X,o})$ is $3$ by \cite[Theorem 5.3]{Tevelev2005}.
\end{proof}

Now we determine the singularity type of $\mathbf{P}(\widetilde{\mathcal{O}})$. 
\begin{proposition}
	\label{Prop:Singularity-type-Ai}
Assume that $\varepsilon: \mathbf{P}(T^*_{G/P}) \to \mathbf{P}(\widetilde{\mathcal{O}})$ is divisorial and the VMRT of $G/P$ is not dual defective.  Then 
$\mathbf{P}(\widetilde{\mathcal{O}})$ is of type $A_1$ except for $G/P = E_7/P_4$, which is of type $A_2$.
\end{proposition}
\begin{proof}

 Consider first the case where  $\widehat{s}$ is birational,  then $\mathbf{P}(\widetilde{\mathcal{O}})$ is just the normalization of $\mathbf{P}(\overline{\mathcal{O}})$, whose generic singularity type is determined in \cite{FJLS}.  It turns out only for $E_7/P_4$, the generic singularity is of type $A_2$ while all others are of type $A_1$.

Assume now $\widehat{s}$ is not birational.  By Proposition \ref{p.birSpringer},  $G/P$ is either $G_2/P_1$ or $F_4/P_3$ as  $\varepsilon$ is divisorial.
Consider first the case of $G_2/P_1$, which is isomorphic to the 5-dimensional quadric $\mathbb{Q}^5$.  Let $\mathcal{O}$ be the 10-dimensional nilpotent orbit in $\mathfrak{g}_2$ and $\mathcal{O}' \subset \mathfrak{so}_7$ the nilpotent orbit corresponding to the partition $[3, 1^4]$.  Then there is a generically 2-to-1 morphism  $\nu: \overline{\mathcal{O}'} \to \overline{\mathcal{O}}$ which is induced from the projection $\mathfrak{so}_7 \to \mathfrak{g}_2$. The map $\widehat{s}: T^*_{G_2/P_1} \to \overline{\mathcal{O}}$  factorizes through $\nu$.  As $\overline{\mathcal{O}'}$ is normal, we have 
$\widetilde{\mathcal{O}} =  \overline{\mathcal{O}'}$ which has generic singularity type $A_1$.

Now consider the case of $F_4/P_3$.  In this case, the Springer map $\widehat{s}: T^*_{F_4/P_3} \to \overline{\mathcal{O}}_{F_4(a_3)}$ has degree 4  (\cite[Appendix]{Fu07}).   By Theorem 1.3 in \cite{FJLS},  the transverse slice $\mathcal{T}$ from the codimension 6 orbit $\mathcal{O}_{A_2+\tilde{A}_1}$ to $\overline{\mathcal{O}}_{F_4(a_3)}$ is isomorphic to the quotient $(\mathbb{C}^3 \oplus \mathbb{C}^{3*})/\mathfrak{S}_4$, where $\mathfrak{S}_4$ acts on $\mathbb{C}^3$ by reflection representation.   The only index 4 subgroup of $\mathfrak{S}_4$ is $\mathfrak{S}_3$, hence  the degree 4 map
$\widetilde{\tau}:  \widetilde{\mathcal{O}} \to \overline{\mathcal{O}}_{F_4(a_3)}$  is locally the quotient $(\mathbb{C}^3 \oplus \mathbb{C}^{3*})/\mathfrak{S}_3 \to (\mathbb{C}^3 \oplus \mathbb{C}^{3*})/\mathfrak{S}_4$.   Hence the generic singularity of $\widetilde{\mathcal{O}} $ is the same as that of $(\mathbb{C}^3 \oplus \mathbb{C}^{3*})/\mathfrak{S}_3$, which is of type $A_1$.
\end{proof}

We can summarize the types of $G/P$ in the following table.  By Proposition  \ref{Prop:Pseff-Divisor}  and Corollary \ref{c.G/Ptypes}, we get the number $a, b$ for $G/P$ not of type (I) and (II-s), the latter cases will be done in the next subsection by applying Proposition \ref{Prop:Pseff-Small}.

\footnotesize
\renewcommand*{\arraystretch}{1.6}
\begin{longtable}{|M{0.5cm}|M{5.5cm}|M{6cm}|M{2.5cm}|}
	\caption{Types of rational homogeneous spaces}
	\label{Table:Types-RHS}
	\\
	\hline

	%\multirow{2}{*}{}   
	   %                    & \multirow{2}{*}{II-d-d}   &
	%\multicolumn{2}{c|}{II-d-A}
	%\\
	%\cline{3-4}

	                      &          II-d-d                 &

	II-d-A1              &       II-d-A2
	\\
	\hline
	
	    $A_n$              &                  -            &
	$k=\frac{n+1}{2}$      &              -
	\\
	\hline
	
	    $B_n$        &       $\frac{2n+1}{3}\leq k \leq n-1$ and $k$ odd     &
	$\begin{cases}
		k\leq \frac{2n}{3} \\
		k=n\ {\rm and}\ n\geq 3\ {\rm odd}
	\end{cases}$
   &            -
	\\
	\hline

        $C_n$                            &     $2\leq k\leq \frac{2n}{3}$ and $k$ even    &
    $k\geq \frac{2n+1}{3}$                     &              -
    \\
    \hline
    
        $D_n$                              &   $\frac{2n}{3}\leq k\leq n-2$ and $k$ even  &
        $\begin{cases}
        	k\leq \frac{2n-1}{3} \\
        	k=n-1\ \text{or}\ n,\ {\rm and}\ n\geq 4\ {\rm even}
        \end{cases}$               &              -
    \\
    \hline
    
          $E_n$                 &  
           $E_6/P_2$, $E_7/P_6$, 
           $E_8/P_k$ $(k=3,4,6)$      &   otherwise    &       $E_7/P_4$
    \\
    \hline

          $F_4$                 &      $k=4$     &   
          $k=1,2,3$             &      -
    \\
    \hline
    
           $G_2$                 &      -         &   
        $k=1$, $2$               &      -
    \\
    \hline
\end{longtable}

\normalsize

As an immediate application, we obtain:

\begin{proposition}
	\label{Prop:Fano-contact}
	Let $X=G/P$ be a rational homogeneous space of Picard number $1$. Denote by $H$ the ample generator of $\pic(X)$ and by $\pi:\PP(T_X)\rightarrow X$ the natural projection.
	\begin{enumerate}
		\item If $X$ is isomorphic to one of the varieties listed in Conjecture \ref{Conj:normalized-Tangent}, then the normalized tangent bundle of $X$ is pseudoeffective but not big.
		
		\item If $X$ is a homogeneous Fano contact manifold different from a projective space, then the total dual VMRT $\check{\sC}\subset \PP(T_X)$ is a prime divisor satisfying
		\[
		[\check{\sC}] \equiv 4\Lambda - 2\pi^*H.
		\]
	\end{enumerate}
\end{proposition}

\begin{proof}
	For the statement (1), this is already proved in \cite{Shao2020}. Here we use the total dual VMRT to give a new proof. In fact, this can be easily derived from the table below:
	
	\renewcommand*{\arraystretch}{1.6}
	\begin{longtable*}{|M{2.2cm}|M{2.4cm}|M{2.5cm}|M{1.5cm}|M{3cm}|M{1.1cm}|}
		\hline
		
		$G/P$                 &   $\bQ^n$
		&  $\Gr(n,2n)$             &   $\mathbb{S}_{2n}$
		&  $\LG(n,2n)$              
		&  $E_7/P_7$
		\\
		\hline
		
		VMRT $\sC_o$                      &   $\bQ^{n-2}$
		&  $\PP^{n-1}\times \PP^{n-1}$    &   $\Gr(2,2n)$
		&  $\PP^{n-1}$                   
		&  $E_6/P_1$
		\\
		\hline
		
		embedding                         &   Hyperquadric
		&  Segre                          &   Pl\"ucker
		&  second Veronese                
		&  Severi
		\\
		\hline
		
		codegree $a$                      &   $2$
		&  $n$                            &   $n$
		&  $n$                
		&  $3$
		\\
		\hline
	\end{longtable*}

	Note that the VMRT of $X$ is not dual defective and its codegree is given in the last row of the table above. Moreover, by Proposition \ref{Prop:Singularity-type-Ai} and Corollary \ref{c.G/Ptypes} that we have
	\[
	[\check{\sC}] \equiv a \Lambda -2 \pi^*H.
	\]
	Then one can easily check case by case that we have $a\cdot \text{index}(X)-2\cdot \dim(X)=0$. Hence, the normalized tangent bundle of $X$ is pseudoeffective but not big by Theorem \ref{Thm:Class-dual-VMRT}.
	
	For the statement (2), it can be derived from the table below by the same argument as above
	\renewcommand*{\arraystretch}{1.6}
	\begin{longtable*}{|M{2.2cm}|M{2.4cm}|M{1.5cm}|M{1.5cm}|M{1.5cm}|M{1.5cm}|M{1.5cm}|}
		\hline
		
		$G/P$                 &   ${\rm OG}(2,n+6)$
		&  $E_6/P_2$               &   $E_7/P_1$
		&  $E_8/P_8$               &   $F_4/P_1$ 
		&  $G_2/P_2$
		\\
		\hline
		
		VMRT $\sC_o$                      &   $\PP^1\times \mathbb{Q}^n$
		&  $\Gr(3,6)$                     &   $\mathbb{S}_6$
		&  $E_7/P_7$                      &   $\LG(3,6)$                   
		&  $\PP^1$
		\\
		\hline
		
		embedding                         &   Segre
		&  Pl\"ucker                      &   Spinor
		&  $\sO(1)$                       &   $\sO(1)$                
		&  $\sO(3)$
		\\
		\hline
	\end{longtable*}
	Note that all the VMRTs above are not dual defective with codegree $4$ (see \cite[p.169]{Tevelev2005}). In particular, by Proposition \ref{Prop:Singularity-type-Ai} and Corollary \ref{c.G/Ptypes}, we have
	$[\check{\sC}]\equiv 4\Lambda - 2\pi^*H$.
\end{proof}

\subsection{Geometry of stratified Mukai flops} \label{s.GeomMukaiFlop}

This subsection is devoted to explicitly calculate the positive integers $a(H)$ and $b(H)$ in Proposition \ref{Prop:Pseff-Small}.  It turns out that the flops are symmetric. In particular, according to Proposition \ref{Prop:Pseff-Small}, after exchanging $\mu$ and $\mu^{-1}$, we get $a(H)=\pi^*_1 H \cdot \mu_*^{-1}(\ell_2)$ and hence we always have $b(H)=1$.  It remains to determine $a(H)$, which by Proposition \ref{Prop:Pseff-Small}  can be interpreted as the degree of the image under the flop of a general line in the projectivized cotangent space. We will describe in details the flops which will enable us to determine this degree. 

For a stratified Mukai flop $\widehat{\mu}: T^*_{G/P} \dashrightarrow T^*_{G/Q}$ (where $P$ may coincide with $Q$), it induces a rational map
$\nu: \mathbf{P}(T^*_{G/P,o}) \dashrightarrow G/Q$ by composing the projectivization of $\widehat{\mu}$ with the projection $\mathbf{P}(T^*_{G/Q})\to G/Q$.  The aim of this section is to describe the rational map $\nu$ and then compute the degree of $\nu(\ell)$ for a general line $\ell$ in $\mathbf{P}(T^*_{G/P,o})$.
The result is summarised in the following table: \\

\footnotesize
\renewcommand*{\arraystretch}{1.6}
\begin{longtable}{|M{2cm}|M{1.5cm}|M{0.9cm}|M{0.5cm}|M{0.7cm}|M{2.8cm}|M{2.5cm}|M{2cm}|}
	\caption{Degree of lines under stratified Mukai flops}
	\label{Table:Degree-of-lines}
	\\
	\hline
	
	Type & $A_{n,k}$ & $D_{2n+1}$ & $E_{6,I}$ &  $E_{6,II}$ & $B_{n,k}$&  $D_{n,k}$&  $C_{n,k}$ \\
	\hline
	degree $\nu(\ell)$ & $k$  &  $n$  & 2  &  4 & $2n-k$  & $2n-1-k$ & $2k-2$ \\
	\hline
	condition for $k$  &  $2k < n$  &    &    &  &  $\frac{2n+1}{3}\leq k\leq n-1$, $k$ even  & $\frac{2n}{3}\leq k\leq n-2$, $k$ odd & $2\leq k\leq \frac{2n}{3}$, $k$ odd\\
	\hline
\end{longtable}
\normalsize

\subsubsection{Preliminary}

Recall that for a simple Lie algebra $\mathfrak{g}$, there exist only finitely many nilpotent orbits in $\mathfrak{g}$.  In classical types, these orbits are parametrized by certain partitions, which correspond to sizes of the Jordan blocks in each conjugacy class.

Now we consider classical B-C-D types.
Let $\epsilon \in \{0,1\}$ and $V$ a $d$-dimensional vector space with a non-degenerate bilinear form such that $\langle v, w \rangle = (-1)^\epsilon \langle w, v \rangle $ for all $v, w \in V$.

Given a nilpotent element $\phi: V \to V$ preserving the bilinear form,  we can associate to it a partition ${\bf d}=[d_1, \cdots, d_l]$ of $d$.   Except a few cases in type $D$, this partition uniquely determines the conjugacy class of $\phi$, denoted by $\mathcal{O}_{\bf d}$.

We identify the partition ${\bf d}$ with a Young table consisting of $d$ boxes, where the $i$-th row consists of $d_i$ boxes for each $i$.  We denote by $(i,j)$ the box of ${\bf d}$ lying on the $i$-th column and $j$-th row.   Let us recall the following classical result (cf. Proof of \cite[Theorem 4.5]{Namikawa2006}).
\begin{proposition} \label{p.Jordan}
	For an element $\phi \in \mathcal{O}_{\bf d}$, there exists a basis $e(i,j)$ of $V$ indexed by the Young diagram ${\bf d}$ with the following properties:
	
	(a) $\phi (e(i, j)) = e(i-1,j)$ for all $(i,j) \in {\bf d}$.
	
	(b) $\langle e(i, j), e(p, q)\rangle \neq 0 $ if and only if $p=d_j-i+1$ and $q=\beta(j)$, where $\beta$ is a permutation of $\{1,2, \cdots, l\}$  ($l$ is the length of the partition) which satisfies $\beta^2=id$, $d_{\beta(j)} = d_j$, and $\beta(j) \not\equiv j (\text{mod } 2)$ if $d_j \not\equiv \epsilon (\text{mod } 2)$.  One can choose an arbitrary $\beta$ within these restrictions.
\end{proposition}

We start with the following elementary result.
\begin{proposition} \label{p.localmodel}
	Let $a < b$ be two integers and $m$ an odd integer.
	Let $A, B, W$ be vector spaces of dimension $a, b, m$ respectively.
	
	1)  Consider the rational map $\nu_1: \mathbf{P}({\rm Hom}(A, B)) \dashrightarrow {\rm Gr}(a, B)$ by sending a general element $\psi \in {\rm Hom}(A, B)$ to its image ${\rm Im}(\psi) \subset B$.   Then $\nu_1$ sends a general line in $\mathbf{P}({\rm Hom}(A, B))$ to a curve of degree $a$ in ${\rm Gr}(a, B)$.
	
	2) Consider the rational map $\nu_2: \mathbf{P}(\wedge^2 W) \dashrightarrow \mathbf{P} W^*$ by sending a general element $\psi \in \wedge^2 W$ to its kernel ${\rm Ker}(\psi)$ (by viewing $\psi$ as a map from $W^*$ to $W$).   Then $\nu_2$ sends a general line in $\mathbf{P}(\wedge^2 W)$ to a curve of degree $m-1$ in $\mathbf{P}(W^*)$.
\end{proposition}
\begin{proof}
	1) Take a general (parameterised) line $[\psi_\lambda] \in \mathbf{P}({\rm Hom}(A, B))$ (with $\lambda \in \mathbf{P}^1$), then $\psi_\lambda: A \to B$ is injective. Take a basis  $e_1, \cdots, e_a$ of $A$, then ${\rm Im}(\psi_\lambda) \subset {\rm Gr}(a, B)$ corresponds to the curve (under the Pl\"ucker embedding)
	$$ \lambda \mapsto \psi_\lambda(e_1) \wedge \cdots \wedge \psi_\lambda(e_a),$$
	which is of degree $a$ as $\psi_\lambda$ is linear in $\lambda$. 
	
	2) For a general element $\psi \in \wedge^2V$, it has the maximal rank $m-1$ as $m$ is odd. Take a general subspace $W_0^* \subset W^*$ of codimension 1, then
	$\psi: W_0^* \to {\rm Im}(\psi)$ is an isomorphism.  By taking a basis of $W_0^*$ and using a similar argument as in 1), we see that $\nu_2$ maps a general line to a degree $m-1$ curve in
	${\rm Gr}(m-1, W) \simeq \mathbf{P} W^*$.
\end{proof}

\subsubsection{Type $A_{n,k}$}

Let $V$ be an $(n+1)$-dimensional vector space and $k <(n+1)/2$ an integer.  The $A_{n,k}$ flop is the birational map
$\widehat{\mu}: T^*{\rm Gr}(k, V) \dashrightarrow T^* {\rm Gr}(k, V^*)$ which is given as follows:

For any $[F] \in {\rm Gr}(k, V)$,  there exists a natural isomorphism $T_{[F]}^*{\rm Gr}(k, V) \simeq {\rm Hom}(V/F, F)$.  An element $\phi \in {\rm Hom}(V/F, F)$ gives naturally an element $$\phi^* \in {\rm Hom}(F^*, (V/F)^*) \subset {\rm Hom}(F^*, V^*).$$ If $\phi$ is general, then $\phi: V/F \to F$ is surjective as $\dim F < \dim V/F$.  This gives an injective map $\phi^*: F^* \to (V/F)^*$, whose image gives an element $[{\rm Im}(\phi^*)] \in {\rm Gr}(k, (V/F)^*) \subset {\rm Gr}(k, V^*)$.   The flop $\widehat{\mu}$ sends $([F], \phi)$ to $([{\rm Im}(\phi^*)], \phi^*)$.  Hence the rational map $\nu$ is given by
$$
\nu: \mathbf{P}(T^*_{[F]}{\rm Gr}(k, V)) \dashrightarrow {\rm Gr}(k, (V/F)^*) \subset {\rm Gr}(k, V^*), \quad [\phi] \mapsto [{\rm Im}(\phi^*)].
$$
By Proposition \ref{p.localmodel}, $\nu$ maps a general line to a curve of degree $k$ on ${\rm Gr}(k, V^*)$.

\subsubsection{Type $D_{2n+1}$} \label{s.D2n+1}

Let $(V, \langle, \rangle)$ be an orthogonal space of dimension $4n+2$. The spinor variety $\mathbb{S}:=\mathbb{S}_{2n+1}$, which parametrizes $(2n+1)$-dimensional isotropic subspaces of $(V, \langle, \rangle)$, consists of two irreducible components
$\mathbb{S}^+, \mathbb{S}^- $.   It turns out the Richardson orbits in  $ \mathfrak{so}_{4n+2}$ associated to $\mathbb{S}^+$ and $ \mathbb{S}^- $ are the same, which corresponds to  the  partition $[2^{2n}, 1^2]$. The two Springer maps $T^* \mathbb{S}^+ \xrightarrow{\widehat{s}^+} \overline{\mathcal{O}} \xleftarrow{\widehat{s}^-} T^*\mathbb{S}^-$ are birational, which gives the
$D_{2n+1}$ flop  $\widehat{\mu}: T^* \mathbb{S}^+ \dashrightarrow T^* \mathbb{S}^-$.

The flop $\widehat{\mu}$ can be described as follows (cf. \cite[Lemma 5.6]{Namikawa2006}): given a general element $\phi \in \mathcal{O}$, the kernel ${\rm Ker}(\phi)$ is of dimension $2n+2$ which contains the $2n$-dimensional vector subspace ${\rm Im}(\phi)$. The quotient $\bar{V}:={\rm Ker}(\phi)/{\rm Im}(\phi)$ is a 2-dimensional orthogonal vector space, which has exactly two isotropic lines (say $L^+, L^-$). Then their pre-images in ${\rm Ker}(\phi)$ give two $(2n+1)$-dimensional isotropic subspace $F^+, F^-$ of $V$. This gives two points $[F^{\pm}] \in \mathbb{S}^{\pm}$.  The flop $\mu$ maps $([F^+], \phi)$ to
$([F^-], \phi)$.  Note that we have a natural isomorphism $F^-/{\rm Im}(\phi) \simeq {\rm Ker}(\phi)/F^+$, which shows that $F^-$ is the linear span of ${\rm Im}(\phi)$ and ${\rm Ker}(\phi)/F^+$.

For an element $[F] \in \mathbb{S}^+$,  we have a natural isomorphism $V/F \simeq F^*$ induced from the pairing $\langle, \rangle$ on $V$ as $F = F^\perp$.
Further more $T_{[F]}^* \mathbb{S}^+ \simeq \wedge^2 F$.   We fix a (non-canonical) isomorphism $V \simeq F \oplus F^*$ such that the pairing $\langle, \rangle$ on $V$ corresponds to the natural pairing on $F \oplus F^*$.

For general $\phi \in \wedge^2 F$, its kernel is one-dimensional (as $\dim F$ is odd), which defines a point $[f_\phi^*] \in \mathbf{P}F^*$.  Then ${\rm Im}(\phi)$ is just the hyperplane $H_\phi$ in $F$ annihilating $f_\phi^*=0$.
Thus the rational map $\nu$ is the composition of maps
$$\mathbf{P}(\wedge^2 F) \dashrightarrow \mathbf{P} F^* \subset  \mathbb{S}^-, \quad [\phi] \mapsto  [f_\phi^*] \mapsto <H_\phi, f_\phi^*>.$$
By Proposition \ref{p.localmodel}, $\nu$ maps a general line in $\mathbf{P}(\wedge^2 F)$ to a curve of degree $2n$ in the Pl\"ucker embedding of  $\mathbb{S}^-$.

Note that the composition $\mathbb{S}^- \subset {\rm Gr}(2n+1, V) \subset \mathbf{P} (\wedge^{2n+1} V)$ is induced by $\mathcal{O}_{\mathbb{S}^-}(2)$, hence this gives a degree $n$ curve on $\mathbb{S}^-$.

\subsubsection{Type $E_{6,I}$}

Consider the $E_{6,I}$ flop $\widehat{\mu}: T^* (E_6/P_1) \dashrightarrow T^*(E_6/P_6)$. Fix a point $o \in E_6/P_1$, then the cotangent space $T_o^* (E_6/P_1)$ can be identified with the spinor representation $\mathcal{S}$ of ${\rm Spin}_{10}$. Let $\mathbb{Q}^8$ be the smooth 8-dimensional hyperquadric. By \cite[Proposition 1.5 ]{Chaput2006}, there exists a unique $\mathbb{C}^* \times {\rm Spin}_{10}$-equivariant rational map $\hat{\nu}: \mathcal{S} \dashrightarrow \mathbb{Q}^8$, which is defined as follows: the affine cone of the 10-dimensional spinor variety $\hat{\mathbb{S}}_5 \subset \mathcal{S}$ is defined by 10 quadratic equations
$Q_1=\cdots=Q_{10}=0$, and  the map $\hat{\nu}$ is given by $z \mapsto [Q_1(z):\cdots:Q_{10}(z)] \in \mathbb{P}^{9}$ whose image is contained in $\mathbb{Q}^8$.
This implies that if we take a general line $\ell$ in $\mathbb{P} \mathcal{S}$, then $\nu (\ell)$ is a conic on $\mathbb{Q}^8$.

By  \cite[Theorem 3.3]{Chaput2006}, the map $\hat{\nu}$ is the composition of $\widehat{\mu}$ with the projection $T^* (E_6/P_6) \to E_6/P_6$ (under the natural embedding $\mathbb{Q}^8 \subset E_6/P_6$). This shows that the rational map $\nu: \mathbf{P} T_o^*(E_6/P_1) \dashrightarrow E_6/P_6$ maps a general line  to a conic.

\subsubsection{Type $E_{6, II}$}

Let $F$ be a 5-dimensional vector space.  By  \cite[Proposition 2.1]{Chaput2006}, there exists a unique ${\rm GL}_2 \times {\rm GL}(F)$-equivariant rational map
$$g:  \wedge^2 F^* \oplus \wedge^2 F^* \dashrightarrow {\rm Gr}(3, F)$$
which maps a general element $(\omega_1, \omega_2) \in \wedge^2 F^* \oplus \wedge^2 F^*$ to ${\rm Ker}(\omega_1)^{\perp_{\omega_2}} \cap {\rm Ker}(\omega_2)^{\perp_{\omega_1}}$.  Here $\omega_i \in \wedge^2 F^*$ is viewed as a two form on $F$ and ${\rm Ker}(\omega_1)^{\perp_{\omega_2}}$ means the orthogonal space with respect to $\omega_2$ of the subspace ${\rm Ker}(\omega_1)$.
 As $\omega_i$ is general, it has rank 4, hence ${\rm Ker}(\omega_i)$ is 1-dimensional, which shows that ${\rm Ker}(\omega_1)^{\perp_{\omega_2}} \cap {\rm Ker}(\omega_2)^{\perp_{\omega_1}}$ is a 3-dimensional vector subspace in $F$.  By \cite[Lemma 2.3]{Chaput2006}, we have $g(a \omega_1 + b \omega_2, a' \omega_1 + b' \omega_2) = g(\omega_1, \omega_2)$ for a general element $\begin{pmatrix} a & b \\ a' & b' \end{pmatrix}$ in ${\rm GL}_2$.  By \cite[Lemma 2.4]{Chaput2006}, a general element   $\phi=(\omega_1, \omega_2) \in \wedge^2 F^* \oplus \wedge^2 F^*$ can be co-diagonalised as follows (under a suitable basis $f_1^*, \cdots, f_5^*$ of $F$):
$$
\omega_1 = f_2^* \wedge f_4^* + f_3^* \wedge f_5^*, \qquad  \omega_2 = f_1^* \wedge f_5^* + f_3^* \wedge f_4^*.
$$

Take another element $\phi'=(\omega'_1, \omega'_2)$ defined as follows:
$$
\omega'_1 = f_1^* \wedge f_4^* + f_3^* \wedge f_5^*, \qquad \omega'_2 = f_1^* \wedge f_2^* + f_3^* \wedge f_4^*.
$$

Consider the following plane in $\wedge^2 F^* \oplus \wedge^2 F^*$ given by $\phi_{s,t} = s \phi + t \phi' =(\omega_1^{s,t}, \omega_2^{s,t})$ for $(s,t) \in \mathbb{C}^2$. By a direct computation, we have
$$
{\rm Ker}(\omega_1^{s,t}) = \mathbb{C} (sf_1 - t f_2), \qquad  {\rm Ker}(\omega_2^{s,t}) = \mathbb{C} (sf_2 - t f_5).
$$

One remarks that for any $(s, t) \neq (0,0)$, the subspaces ${\rm Ker}(\omega_1^{s,t})$ and   ${\rm Ker}(\omega_2^{s,t})$ are 1-dimensional and they intersect only at $(0,0)$. Moreover, one shows directly that
$$\omega_2({\rm Ker}(\omega_1^{s,t}), \cdot) \cap   \omega_1({\rm Ker}(\omega_2^{s,t}), \cdot) =\{0\}.$$
This shows that $g(\phi_{s,t})$ is well-defined for $(s,t) \neq (0,0)$. By a direct computation, we have
$$
g(\phi_{s,t}) = \{ \sum_i x_i f_i | x_5 = \frac{t^2}{s^2} x_1 - \frac{t}{s} x_2,  x_4 = \frac{(s+t)t}{s^2} x_3\}.
$$

This gives a basis for $g(\phi_{s,t})$, which, under the Pl\"ucker embedding, is mapped to the following curve on ${\rm Gr}(3, F)$
$$
[s:t] \mapsto [(f_1 + \frac{t^2}{s^2} f_5) \wedge (f_2  - \frac{t}{s} f_5) \wedge  (f_3 + \frac{(s+t)t}{s^2} f_4)].
$$

Note that this gives a degree 4 curve on ${\rm Gr}(3, F)$.

Consider the $E_{6,II}$ flop $\widehat{\mu}: T^*(E_6/P_3) \dashrightarrow T^*(E_6/P_5)$. By  \cite[Theorem 4.3]{Chaput2006}, the composition $T_o^*(E_6/P_3) \dashrightarrow T^*(E_6/P_5) \to E_6/P_5$ can be identified with the composition of $g$ with the natural embedding ${\rm Gr}(3, F) \subset E_6/P_5$. The precedent argument shows that a general line in $\mathbf{P} T_o^*(E_6/P_3)$ is mapped to a degree 4 curve on $E_6/P_5$.

\subsubsection{Type $B_{n, k}$} \label{s.Bnk}

Let $(V, \langle, \rangle)$ be an orthogonal space of dimension $2n+1$.  A vector subspace $F \subset V$ is said orthogonal if $F \subset F^\perp$.
The  $k$-th orthogonal Grassmann variety $B_n/P_k$ parametrizes $k$-dimensional orthogonal vector subspaces in $V$.
There exists an isomorphism:
$$
T^*(B_n/P_k) \simeq \{([F], \phi) \in B_n/P_k \times \mathfrak{so}(V)|  \phi(V) \subset F^\perp,  \phi(F^\perp) \subset F \subset {\rm Ker}(\phi)  \}.
$$
Under this isomorphism,  the Springer map $\widehat{s}: T^* (B_n/P_k) \to \overline{ \mathcal{O}}_{\bf d}$ sends $([F], \phi)$ to $\phi$.

When $k$ is even such that $k > \frac{2n+1}{3}$, the Springer map $\pi$ is generically finite of degree 2 and ${\bf d} = [3^{2n+1-2k}, 2^{3k-2n-2}, 1^2]$.   The involution on  general fibers of $\pi$ gives the $B_{n, k}$-flop: $\widehat{\mu}: T^*(B_n/P_k) \dashrightarrow T^*(B_n/P_k)$.

For $\phi \in \mathcal{O}_{\bf d}$,  we choose a basis $e(i, j)$ of $V$ as described by Proposition \ref{p.Jordan} (by taking $\beta$ satisfying $\beta(k)=k+1$).  Then ${\rm Ker}(\phi)$ has dimension $k+1$ and is generated by $e(1,1),  e(1, 2), \cdots, e(1, k+1)$. The two fibers $\pi^{-1}(\phi)$ are given by the following two orthogonal subspaces  (cf. proof of \cite[Theorem 4.5]{Namikawa2006})
$$F_1 =  \sum_{1 \leq j \leq k} \mathbb{C} e(1, j) \ \text{and } F_2 = \sum_{1 \leq j \leq k-1} \mathbb{C} e(1, j) + \mathbb{C} e(1, k+1).$$

One notes that $F_0:=F_1 \cap F_2 = F_1 \cap {\rm Im}(\phi) ={\rm Im}(\phi) \cap {\rm Ker}(\phi)$ is of dimension $k-1$ and $F_2/F_0$ is naturally isomorphic to  ${\rm Ker}(\phi)/F_1$.
The flop $\widehat{\mu}$ interchanges the two fibres. Namely the flop $\widehat{\mu}$ sends $([F_1], \phi)$ to $([F_2], \phi)$ where $F_2$ is the linear span of $F_1 \cap {\rm Im}(\phi)$ and ${\rm Ker}(\phi)/F_1$, the latter being one-dimensional.  Furthermore, $\langle {\rm Ker}(\phi)/F_1, F_0 \rangle =0$ and as $F_0 \subset F_1$ is a hyperplane, it is exactly the  orthogonal part in $F$ of ${\rm Ker}(\phi)/F_1$.  This implies that $F_2$ is in fact uniquely determined by ${\rm Ker}(\phi)/F_1$.  We summarize these in the following picture on Young table.

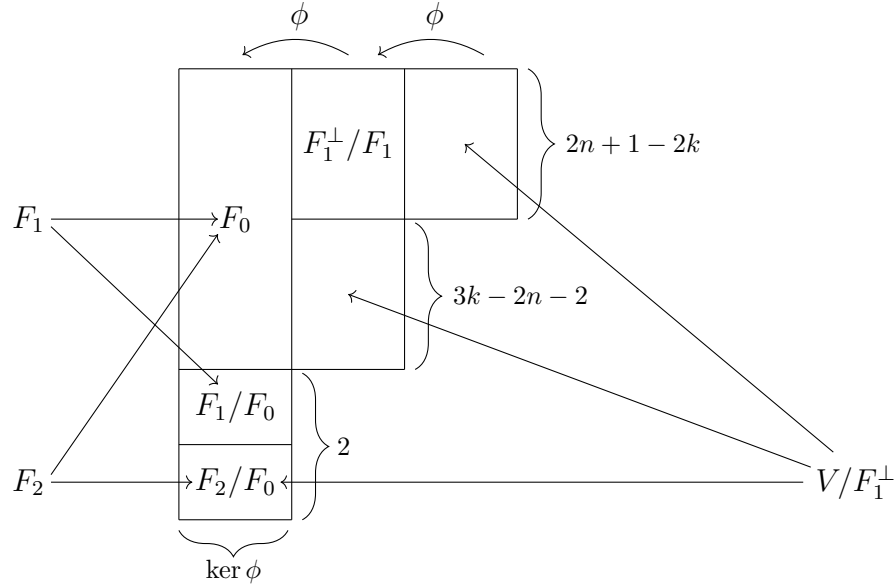
\begin{center}
	\begin{tikzpicture}
		\draw (0,6)--(4.5,6);
		\draw (1.5,4)--(4.5,4);
		\draw (0,2)--(3,2);
		\draw (0,1)--(1.5,1);
		\draw (0,0)--(1.5,0);
		
		\draw (0,6)--(0,0);
		\draw (1.5,6)--(1.5,0);
		\draw (3,6)--(3,2);
		\draw (4.5,6)--(4.5,4);
		
		\node (a) at (4.2,6.1)   {};
		\node (b) at (2.5,6.1) {};
		\draw[->] (a) to [bend right=30] (b);
		\node at (3.4,6.7) {$\phi$};
		
		\node (c) at (0.7,6.1)   {};
		\node (d) at (2.4,6.1)   {};
		\draw[->] (d) to [bend right=30] (c);
		\node at (1.6,6.7) {$\phi$};
		
		\draw[decorate,decoration={brace,amplitude=10pt},xshift=4pt,yshift=0pt] (4.5,6) -- (4.5,4) node [black,midway,xshift=1.4cm] {\footnotesize $2n+1-2k$};
		\draw[decorate,decoration={brace,amplitude=10pt},xshift=4pt,yshift=0pt] (3,3.95) -- (3,2) node [black,midway,xshift=1.4cm] {\footnotesize $3k-2n-2$};
		\draw[decorate,decoration={brace,amplitude=10pt},xshift=4pt,yshift=0pt] (1.5,1.95) -- (1.5,0) node [black,midway,xshift=0.55cm] {\footnotesize $2$};
		\draw[decorate,decoration={brace,amplitude=10pt},xshift=0pt,yshift=-4pt] (1.45,0) -- (0,0) node [black,midway,yshift=-0.55cm] {\footnotesize $\ker{\phi}$};
		
		\node at (0.75,4) {$F_0$};
		\node at (0.75,1.5) {$F_1/F_0$};
		\node at (0.75,0.5) {$F_2/F_0$};
		\node at (2.25,5) {$F_1^{\perp}/F_1$};
		
		\node at (-2,4) {$F_1$};
		\draw[->] (-1.7,4) to (0.52,4);
		\draw[->] (-1.7,3.9) to (0.52,1.8);
		
		\node at (-2,0.5) {$F_2$};
		\draw[->] (-1.7,0.5) to (0.18,0.5);
		\draw[->] (-1.7,0.6) to (0.52,3.8);

		\node at  (9,0.5) {$V/F_1^{\perp}$};
		\draw[->] (8.3,0.5) to (1.35,0.5);
		\draw[->] (8.4,0.7) to (2.25,3);
		\draw[->] (8.7,0.9) to (3.8,5);
	\end{tikzpicture}
\end{center}

Fix an orthogonal space $[F] \in B_n/P_k$, then $F^\perp/F$ is an orthogonal space of dimension $2n+1-2k$ and $V/F^\perp$ is isomorphic to $F^*$ via the pairing $F \times V/F^\perp \to \mathbb{C}$ induced from the bilinear form on $V$.  We fix a (non-canonical) isomorphism $V \simeq F \oplus F^* \oplus F^\perp/F$ such that the orthogonal form on $V$ is given by that induced  on $F^\perp/F$ and the natural one on $F \oplus F^*$.

By \cite[ Proposition 5.1]{LandsbergManivel2003}, we have
$$\iota_F:  T_{[F]}^* (B_n/P_k)  \simeq {\rm Hom}(F^\perp/F, F) \oplus \wedge^2 F.$$
This isomorphism is given as follows:  for $\phi \in T_{[F]}^* (B_n/P_k)$, it induces a map $\phi_0 \in {\rm Hom}(F^\perp/F, F) $ as $F \subset {\rm Ker}(\phi)$ and $\phi(F^\perp) \subset F$.  As $\phi(V) \subset F^\perp$, it induces a map $(\phi_1, \phi_2): V/F^\perp \to F^\perp/F   \oplus F$.  It turns out that $\phi \in \mathfrak{so}(V)$ is equivalent to the following: (1) the map $\phi_1: V/F^\perp \simeq F^* \to F^\perp/F$ is the dual $-\phi_0^*$ of the map $-\phi_0$ (here $F^\perp/F$ is self-dual).   (2) the map $\phi_2: V/F^\perp \simeq F^* \to F$ is in fact an element in $\wedge^2 F$.  Then the isomorphism $\iota_F$ sends $\phi$ to $(\phi_0, \phi_2)$.

Conversely, given $(\phi_0, \phi_2) \in {\rm Hom}(F^\perp/F, F) \oplus \wedge^2 F$, 
we construct $\bar{\phi}$ as a map from $V/F \simeq V/F^\perp \oplus F^\perp/F$ to $F \oplus F^\perp/F $, which is given  as follows
$$
\bar{\phi} = \begin{pmatrix}  \phi_2  & \phi_0  \\ -\phi_0^*  & 0 \end{pmatrix} 
$$

Thus $\bar{\phi}$ is represented as an anti-symmetric matrix of size $\dim V/F = 2n+1-k$.  Note that $\dim V/F$ is odd as $k$ is even, so for a general choice of $(\phi_0, \phi_2)$, the map  $\bar{\phi}$ is of maximal rank $2n-k$ and ${\rm Ker}(\bar{\phi})$ is one-dimensional.  Note that ${\rm Ker}(\bar{\phi}) = {\rm Ker}(\phi)/F$. 
By the natural quotient $V/F \to V/F^\perp \simeq F^*$, the image of ${\rm Ker}(\bar{\phi})$ gives a line $\mathbb{C} f^* \subset F^*$. Then the flop
$\mu$ maps $([F], \phi)$ to $([F'], \phi)$, where $F'  \subset F \oplus F^* \subset V$ is the subspace generated by $H_{f^*}$ and $f^*$, here $H_{f^*}$ is the hyperplane in $F$ defined by $f^*=0$.   Then  the map $\nu: {\bf P}(T_{[F]}^* (B_n/P_k)) \dasharrow B_n/P_k$ is then given by $[\phi_0, \phi_1] \mapsto [F']$.

Note that $H_{f^*}$ is uniquely determined by $f^*$, while $f^*$ is given by the kernel ${\rm Ker}(\bar{\phi}^*)$.
By Proposition \ref{p.localmodel}, $\nu$ maps a line to a curve of degree $2n-k$ on $B_n/P_k$ for the Pl\"ucker embedding of $B_n/P_k$. 
Thus for $k \neq n$, this gives a degree $2n-k$ curve on $B_n/P_k$, while for $k=n$, this gives a curve of degree $n/2$ on $B_n/P_n$ as $B_n/P_n \subset {\rm Gr}(n, 2n+1) \subset \mathbb{P}^N$ is induced by $\mathcal{O}(2)$.

\begin{remark}
	By \cite[Example 3.3]{Namikawa2008}, $B_{2n, 2n}$-flop is the same as $D_{2n+1}$-flop.  Hence we recover the result in Section \ref{s.D2n+1}.
\end{remark}

\subsubsection{Type $D_{n, k}$}

Let $(V, \langle, \rangle)$ be an orthogonal space of dimension $2n$.  As in the $B_{n, k}$-flop case, we have the following isomorphism of the cotangent bundle of 
the  $k$-th orthogonal Grassmann variety $D_n/P_k$:
$$
T^*(D_n/P_k) \simeq \{([F], \phi) \in D_n/P_k \times \mathfrak{so}(V)|  \phi(V) \subset F^\perp,  \phi(F^\perp) \subset F \subset {\rm Ker}(\phi)  \}.
$$
Under this isomorphism,  the Springer map $\widehat{s}: T^* (B_n/P_k) \to \overline{ \mathcal{O}}_{\bf d}$ sends $([F], \phi)$ to $\phi$.

When $k$ is odd such that $n-2 \geq k > \frac{2n}{3}$, the Springer map $\widehat{s}$ is generically finite of degree 2 and ${\bf d} = [3^{2n-2k}, 2^{3k-2n-1}, 1^2]$.   The involution on the general fibres of $\widehat{s}$ gives the $D_{n, k}$-flop: $\widehat{\mu}: T^*(D_n/P_k) \dashrightarrow T^*(D_n/P_k)$.

This flop is similar to the $B_{n, k}$-flop.  By the similar argument, we see that a general line in ${\bf P} (T^*_{[F]} (D_n/P_k))$ is mapped to a curve of degree $2n-k-1$.

\subsubsection{Type $C_{n, k}$}

Let $(V, \omega)$ be a symplectic vector space of dimension $2n$.  A vector subspace $F \subset V$ is said isotropic if $F \subset F^\perp$.
The  $k$-th symplectic Grassmann variety $C_n/P_k$ parametrizes $k$-dimensional isotropic vector subspaces in $V$.
There exists an isomorphism:
$$
T^*(C_n/P_k) \simeq \{([F], \phi) \in C_n/P_k \times \mathfrak{sp}(V)|  \phi(V) \subset F^\perp,  \phi(F^\perp) \subset F \subset {\rm Ker}(\phi)  \}.
$$
Under this isomorphism,  the Springer map $\widehat{s}: T^* (C_n/P_k) \to \overline{ \mathcal{O}}_{\bf d}$ sends $([F], \phi)$ to $\phi$.

When $k$ is odd such that $k \leq \frac{2n}{3}$, the Springer map $\pi$ is generically finite of degree 2 and ${\bf d} = [3^{k-1}, 2^{2}, 1^{2n-3k-1}]$.    The involution on the general fibres of $\pi$ gives the $C_{n, k}$-flop: $\widehat{\mu}: T^*(C_n/P_k) \dashrightarrow T^*(C_n/P_k)$.

When $k=1$, then ${\bf d} = [2^{2}, 1^{2n-4}]$ and  an element $\phi \in \mathcal{O}_{\bf d}$ has rank 2, so ${\rm Im}(\phi)$ is two dimensional. The flop $\mu$ sends $([F], \phi)$ to $([{\rm Im}(\phi)/F], \phi)$.   In this case, if we take a general pencil $\phi_\lambda \in T_{[F]}^* (C_n/P_1)$, then the flop $\widehat{\mu}$ maps it to a line in $C_n/P_1$.

Now we consider the case $3 \leq k \leq \frac{2n}{3}$.  Fix $[F] \in C_n/P_k$ and take a general pencil $\phi_\lambda \in T^*_{[F]}(C_n/P_k)$.  Note that $\phi_\lambda^2=\phi_\lambda \circ \phi_\lambda$ has rank $k-1$, hence ${\rm Im}( \phi_\lambda^2)$ is a vector subspace of dimension $k-1$ in $F$. It defines an element $f_\lambda^*$ in $F^* \simeq V/F^\perp$, which is unique up to a scalar.  Then the image of $([F], \phi_\lambda)$ under the flop $\widehat{\mu}$ is $([F_\lambda], \phi_\lambda)$ where $F_\lambda$ is spanned by  ${\rm Im}( \phi_\lambda^2)$ and $f_\lambda^*$.  This gives a curve on $C_n/P_k$, which is given in the Pl\"ucker embedding $\phi_\lambda^2(v_1) \wedge \phi_\lambda^2(v_2) \cdots \wedge \phi_\lambda^2(v_{k-1})$ for a general chosen $k-1$ vectors $v_1, \cdots, v_{k-1}$ of $V$, as $\phi_\lambda^2$ is quadratic in $\lambda$.  This gives a curve of degree $2(k-1)$.

\subsection{Proof of Theorem \ref{t.RationalHomSpace}}

If the morphism $\varepsilon: \mathbf{P}(T^*_{G/P})\rightarrow \mathbf{P}(\widetilde{\sO})$ is a divisorial contraction, then Proposition \ref{Prop:Pseff-Divisor} and Corollary \ref{c.G/Ptypes} can be applied to determine $a$ and $b$. Nevertheless, in general it is not easy to compute the Segre classes $\Lambda^{2n-1}$ and $\Lambda^{2n-1}\cdot \pi^*H$. In the following, we shall use a similar method as the previous subsection to determine $a$ and $b$ in the classical cases. We start with the following result which computes the pseudoeffective threshold for $G/P$ of type (II-d-d).

\begin{proposition}\label{p.a(Gamma)classic}
Let $G/P$ be a rational homogeneous space of type (II-d-d) of classical type.  Then the pseudoeffective threshold of $G/P$  is given by $\alpha_{G/P} = 1/a$ where $a$ is the integer given by the following table. 

\renewcommand*{\arraystretch}{1.6}
\begin{longtable}{|M{0.5cm}|M{5cm}|M{3.8cm}|M{2.5cm}|}
	\caption{values of $a$ in the case (II-d-d)}
	\label{Table:a-II-d-d}
	\\
	\hline
	$\mathfrak{g}$     &   node  & nilpotent orbit $\mathcal{O}$ &  $a$ \\
	\hline
	$B_n$                           &       $\frac{2n+1}{3}\leq k\leq n-1$ and $k$ odd     & $[3^{2n+1-2k}, 2^{3k-2n-1}]$ &  $2n+1-k$
	\\
	\hline

	$C_n$                            &  $2\leq k\leq \frac{2n}{3}$ and $k$ even    & $[3^{k}, 1^{2n-3k}]$ &  $2k$
	\\
	\hline
	
	$D_n$                              &   $\frac{2n}{3}\leq k\leq n-2$ and $k$ even  & $[3^{2n-2k}, 2^{3k-2n}]$ &  $2n-k$
	\\
	\hline
\end{longtable}
\end{proposition}

\begin{proof}
 Note that for $G/P$ of type (II-d-d), the Springer map  $\widehat{s}: T_{G/P}^* \to \overline{\mathcal{O}}$ is birational by Proposition \ref{p.birSpringer}.  Let $\Gamma$ be the exceptional divisor and write $[\Gamma] \equiv a(\Gamma) \Lambda - b(\Gamma) \pi^*H$.  By Corollary \ref{c.G/Ptypes},  we have $b(\Gamma)=1$.   By Theorem \ref{Thm:Rational-Homogeneous-Spaces},  we have $\alpha_{G/P} = 1/a(\Gamma)$.   By Proposition \ref{Prop:Pseff-Divisor}, we have 
$a(\Gamma)\Lambda' \equiv H'$, which can be used to determine  $a(\Gamma)$ for the classical cases.

As $\mathbf{P}(T_{G/P}^*) \to \mathbf{P}( \overline{\mathcal{O}})$ is birational, this gives a rational map $\eta: \mathbf{P}( \overline{\mathcal{O}}) \dasharrow G/P$.  For any point $x \in \mathcal{O}$, there exists an $\mathfrak{sl}_2$-triplet $(x, y, h)$ by the Jacobson-Morozov theorem.  The nilpotent elements in this $\mathfrak{sl}_2$ give a conic $C$ on $\mathbf{P}(\mathcal{O})$ passing through $[x]$. In other words, $\mathbf{P}(\mathcal{O})$ is covered by conics.   Now we show that $\eta: C \to \eta(C)$ is birational:  let $\mathfrak{n} \subset \mathfrak{g}$ be the nilradical of $\mathfrak{p}$, which is naturally identified with $T_{G/P,o}^*$.  As $\mathcal{O}$ is Richardson, the intersection $\mathfrak{n} \cap \mathcal{O}$ is dense in $\mathfrak{n}$, thus $\mathfrak{n} \subset \overline{\mathcal{O}}$.  
This implies that fibers of $T^*_{G/P} \to G/P$ are mapped to linear subspaces in $\overline{\mathcal{O}}$.  For any $y \in G/P$, denote by $\mathfrak{n}_y$ this linear subspace. 
Then $\mathbf{P} (\mathfrak{n}_y) \cap C =  \mathbf{P} (\mathfrak{n}_y)  \cap \mathbf{P} (\mathfrak{sl}_2) \subset C$.  As $C$ is a conic, while   $\mathbf{P} (\mathfrak{n}_y)  \cap \mathbf{P} (\mathfrak{sl}_2)$ is linear, we have $\mathbf{P} (\mathfrak{n}_y)  \cap \mathbf{P} (\mathfrak{sl}_2)$ is just a point, which shows $\eta: C \to \eta(C)$ is birational.
It follows that $a(\Gamma) = \frac{\eta_*(C) \cdot H}{2}$.  Thus we only need to compute the degree of the curve $\eta_*(C)$. 

Consider the case of  $B_n/P_k$ with $k$ odd and  $k\geq \frac{2n+1}{3}$.  Then $\mathcal{O}$ corresponds to the partition $[3^{2n+1-2k}, 2^{3k-2n-1}]$.   Take an element $\phi \in \mathcal{O} \subset \mathfrak{so}(V)$, then ${\rm Ker}(\phi)$ has dimension $k$ as ${\rm rk}(\phi) = 2n+1-k$.   Using the identification 
$$
T^*(B_n/P_k) \simeq \{([F], \phi) \in B_n/P_k \times \mathfrak{so}(V)|  \phi(V) \subset F^\perp,  \phi(F^\perp) \subset F \subset {\rm Ker}(\phi)  \},
$$
it follows that the map $\eta$ is given by $\eta(\phi) = [{\rm Ker}(\phi)]$.   By Proposition \ref{p.localmodel}, $\eta_*(C)$ is a curve of degree $2(2n+1-k)$, which gives $a = 2n+1-k$. 
The case of $D_n/P_k$ is completely similar.

Consider $C_n/P_k$ with $k$ even and  $k\leq \frac{2n}{3}$.  The nilpotent orbit $\mathcal{O}$ corresponds to the partition $[3^{k}, 1^{2n-3k}]$.   Take an element $\phi \in \mathcal{O}$, then it is easy to see that $\eta(\phi) = [{\rm Im}(\phi^2)]$.   This shows that $\eta(C)$ is a curve of degree $4k$, hence $a=2k$.
\end{proof}

There are five $G/P$ of type (II-d-d) in exceptional Lie algebras. Although the similar approach works, but the map $\eta$ is not explicit, which prevents us to do the computation.  In a similar way, we can get the following:

\begin{lemma}
	\label{Lemma:C_n-II-d-A}
The pseudoeffective threshold of $C_n/P_k$ with $k \geq \frac{2n+1}{3}$ is $\frac{2}{2n-k}$.
\end{lemma}

\begin{proof}
Note that $C_n/P_k$ with $k \geq \frac{2n+1}{3}$ is of type (II-d-A1) and  the Springer map  $\widehat{s}: T_{C_n/P_k}^* \to \overline{\mathcal{O}}_{[3^{2n-2k},2^{3k-2n}]}$ is birational by Proposition \ref{p.birSpringer}.  Let $\Gamma$ be the exceptional divisor and write $[\Gamma] \equiv a(\Gamma) \Lambda - b(\Gamma) \pi^*H$.  By Corollary \ref{c.G/Ptypes},  we have $b(\Gamma)=2$.   By Theorem \ref{Thm:Rational-Homogeneous-Spaces},  we have $\alpha_{C_n/P_k} = 2/a(\Gamma)$.   To compute $a(\Gamma)$, we consider the rational map 
$$\eta: \mathbf{P} \overline{\mathcal{O}}_{[3^{2n-2k},2^{3k-2n}]} \dasharrow C_n/P_k, \quad \quad \phi \mapsto [{\rm Ker}(\phi)]. $$

As in the proof of  Proposition \ref{p.a(Gamma)classic}, take a conic curve $C$ on $\mathbf{P}(\mathcal{O})$, then its image $\eta(C)$ is a curve of degree  $2(2n-k)$. This gives $a(\Gamma) = 2n-k$.
\end{proof}

Now we are ready to prove our main result. 
\begin{proof}[Proof of Theorem \ref{t.RationalHomSpace}]
	The statement (1) is a direct consequence of Theorem \ref{Thm:Rational-Homogeneous-Spaces}, Proposition \ref{Prop:Pseff-Small} and Proposition \ref{Prop:Pseff-Divisor}.
	
	For the statement (2), let $r$ and $d$ be two positive integers. Then there exists an effective divisor $D\subset \PP(T_{G/P})$ such that $D\sim r\Lambda - d\pi^*H$ if and only if
	\[
	H^0(G/P,({\rm Sym}^r T_{G/P})\otimes \sO_{G/P}(-dH))\not=0.
	\]
	Firstly we assume that the morphism $\varepsilon:\mathbf{P}(T^*_{G/P})\rightarrow \mathbf{P}(\widetilde{\sO})$ is divisorial with exceptional divisor $\Gamma$. Then $\Gamma$ is dominated by curves with $\Lambda$-degree $0$ and $\Gamma\equiv a\Lambda - b\pi^*H$ by our definition of $a$ and $b$. Let $m$ be the multiplicity of $D$ along $\Gamma$. Then the restriction of the following effective divisor
	\[
	D-m\Gamma \equiv (r-am)\Lambda + (bm-d)\pi^*H
	\]
	to $\Gamma$ is pseudoeffective. Then we obtain $r-am\geq 0$ and $bm-d\geq 0$. This yields
	\[
	d\leq bm \leq b\left\lfloor \frac{r}{a} \right\rfloor,
	\]
	where the second inequality follows from the fact that $m$ is an integer. Conversely, if $r$ and $d$ are two positive integers satisfying $ d\leq b\lfloor \frac{r}{a} \rfloor$.  We define $m=\lfloor\frac{r}{a}\rfloor$. Then we get
	\[
	r\Lambda - d\pi^*H \sim m\Gamma  + (r-am)\Lambda + (bm-d)\pi^*H.
	\]
	Note that $r-am\geq 0$ and $bm-d\geq 0$ by our assumption. As $\Lambda$ and $H$ are globally generated, it follows that there exists an effective divisor $D'$ such that 
	\[
	r\Lambda - d\pi^*H\sim m\Gamma + D'\geq 0.
	\]
	
	Next we assume that the morphism $\varepsilon:\mathbf{P}(T^*_{G/P})\rightarrow \mathbf{P}(\widetilde{\sO})$ is small. We consider the stratified Mukai flop $\mu:\mathbf{P}(T^*_{G/P})\dashrightarrow \mathbf{P}(T^*_{G/Q})$. Let $D\subset \mathbf{P}(T^*_{G/P})$ be an effective divisor such that 
	\[
	D\sim r\Lambda - d\pi_1^*H.
	\]
	By Proposition \ref{Prop:Pseff-Small}, the push-forward by $\mu$ shows
	\[
	\mu_*D \sim r\Lambda' - d\mu_* \pi_1^*H \sim r\Lambda' - d( a\Lambda' - \pi_2^*H) \sim (r-da)\Lambda' + d\pi_2^*H,
	\]
	where $\Lambda'$ is the tautological divisor of $\mathbf{P}(T^*_{G/Q})$ and we use the fact that $b=1$ in this case. As $\mu_*D$ is effective, we obtain $r-da\geq 0$. Conversely, if $r$ and $d$ are two positive integers satisfying $ d\leq b\lfloor \frac{r}{a} \rfloor$. Then we get $ad\leq r$ as $b=1$. In particular, as $\Lambda'$ and $H$ are globally generated, there exists an effective divisor $D'$ such that $D'\sim (r-ad)\Lambda' + d\pi_2^*H$. Then the pull-back $\mu^*D'$ is an effective divisor such that 
	\[
	\mu^*D' \sim (r-ad)\Lambda + d\mu^*\pi_2^*H \sim r\Lambda - d\pi_1^*H.
	\] 
	
	For the statement (3), note first that the tangent bundle $T_{G/P}$ is semi-stable. Thus by Lemma \ref{l.semistable} we  have
	\[
	\frac{b}{a} = \alpha_{G/P} = {\rm index}(G/P) \cdot \alpha(G/P,-K_{G/P}) \leq  \frac{{\rm index}(G/P)}{\dim(G/P)}.
	\]
	In particular, the normalized tangent bundle of $G/P$ is pseudoeffective if and only if 
	\[
	a\cdot {\rm index}(G/P)= b \cdot \dim(G/P).
	\]
	Consequently, as $b\leq 2$, it follows that $2\dim(G/P)$ is divided by ${\rm index}(G/P)$. Thus, for $G$ of exceptional type, one can check by Appendix \ref{Appendix} that the normalized tangent bundle of $G/P$ is pseudoeffective if and only if $G/P$ is isomorphic to either $E_7/P_7$ or $G_2/P_1=\mathbb{Q}^5$. \\
	
	\textit{Type $A_n/P_k$.} Note that $A_n/P_k$ is isomorphic to $A_n/P_{n+1-k}$. Thus we may assume that $2k\leq n+1$. Firstly we assume that $2k\leq n$, then $a=k$ and $b=1$. Then we have
	$$
		a\cdot \text{index}(A_n/P_k) - b\cdot \dim(A_n/P_k)  = k(n+1) - k(n-k+1) = k^2 >0. $$
    Hence, the normalized tangent bundle of $A_n/P_k$ is not pseudoeffective if $2k\not= n+1$. Next we assume that $2k=n+1$, then we have $a=k$ and $b=2$ and we have
	\[
	a\cdot \text{index}(A_{n}/P_{k}) - b\cdot \dim(A_n/P_k) = k(n+1) - 2k(n-k+1) = k(2k-n-1) =0.
	\]
	Hence, the normalized tangent bundle of $X=A_n/P_k$ is pseudoeffective if $2k=n+1$ and $X$ is isomorphic to the Grassmann variety $\Gr(k,2k)$ in this case. \\
	
	\textit{Type $B_n/P_k$.} Firstly we assume that $3k\leq 2n$. Then $a=2k$ and $b=2$. In particular,  we have
	$$
		a\cdot \text{index}(B_n/P_k) - b\cdot \dim(B_n/P_k) = 2k(2n-k) - k(4n-3k+1) = k(k-1) \geq 0 \
		$$
    with equality if and only if $k=1$. Hence, if $3k\leq 2n$, then the normalized tangent bundle of $B_n/P_k$ is pseudoeffective if and only if $k=1$, in which case $B_n/P_k$ is isomorphic to the $(2n-1)$-dimensional quadric $\bQ^{2n-1}$. Next we assume that $2n+1\leq 3k\leq 3(n-1)$. Then $a=2n-k$ ($k$ even) or $2n-k+1$ ($k$ odd), and $b=1$. Nevertheless note that we have
	\begin{align*}
		2a\cdot \text{index}(B_n/P_k) - 2b\cdot \dim(B_n/P_k) & \geq  2(2n-k)^2 - k(4n-3k+1) \\
		                       & = 8n^2-12nk +5k^2 -k \\
		                       & = (2n-2k)(4n-2k) + k^2-k >0. 
	\end{align*}
    Therefore, if $2n+1\leq 3k\leq 3(n-1)$, then the normalized tangent bundle of $B_n/P_k$ is not pseudoeffective. Finally, we assume that $k=n$. Then $a=\lfloor \frac{n+1}{2}\rfloor$, and $b=1$ ($n$ even) or $b=2$ ($n$ odd). On the other hand, note that $B_n/P_n$ is the $\frac{n(n+1)}{2}$-dimensional spinor variety $\mathbb{S}_{n+1}$ with index $2n$. In particular, one can easily obtain that the normalized tangent bundle of $B_n/P_n$ is pseudoeffective if and only if $n$ is odd.\\
    
    \textit{Type $C_n/P_k$.} If $k=1$, then $C_n/P_1$ is isomorphic to $\PP^{2n-1}$ whose normalized tangent bundle is known to be non-pseudoeffective. Now we assume that $6\leq 3k\leq 2n$, then $a=2k-2$ ($k$ odd) or $a=2k$ ($k$ even) and $b=1$. If $k\geq 3$, then we have
    \begin{align*}
    	2a\cdot \text{index}(C_n/P_k) - 2b\cdot \dim(C_n/P_k) & \geq 2(2k-2)(2n-k+1) - k (4n-3k+1)  \\
    	                                       & = 4nk-k^2+7k-8n-4 \\
    	                                       & =\frac{2nk}{3}-k^2 + \frac{10nk}{3} - 8n +7k-4 >0.
    \end{align*}
    Hence, the normalised tangent bundle of $B_n/P_k$ is not pseudoeffective if $9\leq 3k\leq 2n$.   For $k=2$ and $n \geq 3$,  one can easily check that the normalized tangent bundle of $B_n/P_2$ is not pseudoeffective in the same way. Finally we assume that $3k\geq 2n+1$. Then $a=2n-k$ and $b=2$. Then we obtain
    \begin{align*}
    	a\cdot \text{index}(C_n/P_k) - b\cdot \dim(C_n/P_k) & = (2n-k)(2n-k+1) - k(4n-3k+1) \\
    	                                            & \geq (2n-2k+1)(2n-2k) \geq 0                                           
    \end{align*}
    with equality if and only if $k=n$. In particular, if $3k\geq 2n+1$, then the normalized tangent bundle of $C_n/P_k$ is pseudoeffective if and only if  $k=n$, which is the Lagrangian Grassmann variety $\LG(n,2n)$. \\
    
    \textit{Type $D_n/P_k$.} Firstly we assume that $3k\leq 2n-1$. Then $a=2k$ and $b=2$. Then we obtain
  $$a\cdot \text{index}(D_n/P_k) - b\cdot \dim(D_n/P_k) = 2k(2n-k-1) - k(4n-3k-1)  \geq k^2 -k \geq 0,     $$
    with equality if and only if $k=1$. Hence, if $3k\leq 2n-1$, then the normalized tangent bundle of $D_n/P_k$ is pseudoeffective if and only if $k=1$, which is the $(2n-2)$-dimensional quadric $\bQ^{2n-2}$. Next we assume that $2n\leq 3k\leq 3(n-2)$, then $a=2n-k-1$ ($k$ odd) or $a=2n-k$ ($k$ even), and $b=1$. In particular, we have
    \begin{align*}
    	2a\cdot \text{index}(D_n/P_k) - 2b\cdot \dim(D_n/P_k) & \geq 2(2n-k-1)^2 - k(4n-3k-1) \\ 
    	                                                      & = 8n^2-12nk+5k^2+5k-8n+2 \\
    	                                                      & = (2n-2k)(4n-2k)+k^2+5k-8n+2 \\
    	                                                      & \geq 4(2n+4) + k^2+5k-8n+2  > 0, 
 \end{align*}
    where the fourth inequality follows from the fact that $k\leq n-2$. Hence the normalized tangent bundle is not pseudoeffective if $2n\leq 3k\leq 3(n-2)$. Finally, if $k\geq n-1$, then $D_n/P_k$ is isomorphic to the spinor variety $\mathbb{S}_{n}=B_{n-1}/P_{n-1}$ and the normalized tangent bundle of $D_n/P_k$ with $k\geq n-1$ is pseudoeffective if and only if $n$ is even.
\end{proof}

\section{Two non-homogeneous examples}

As mentioned in the introduction, besides rational homogeneous spaces, there are only two known examples of Fano manifolds with Picard number $1$ and big tangent bundle: the del Pezzo threefold $V_5$ of degree $5$ and the horospherical $G_2$-variety $\mathbb{X}$. In this subsection, we describe the pseudoeffective cones of the projectivized tangent bundle of $V_5$ and $\mathbb{X}$. Recall that $V_5$ is actually a codimension $3$ linear section of $\Gr(2,5)\subset \PP^9$ and the bigness of $T_{V_5}$ is proved in \cite{HoeringLiuShao2020} using the total dual VMRT. In particular, this gives the pseudoeffective cone of $\PP(T_{V_5})$ by applying Theorem \ref{Thm:Class-dual-VMRT}. Actually, we have the following complete descriptions of the cones of divisors of $\PP(T_{V_5})$.

\begin{proposition}
	\label{Prop:Cones-V5}
	Let $X$ be the del Pezzo Fano threefold $V_5$ of degree $5$. Denote by $\pi:\sX:=\PP(T_X)\rightarrow X$ the projectivized tangent bundle of $X$. Let $H$ be the ample generator of $\pic(X)$ and let $\Lambda$ be the tautological divisor  of $\PP(T_X)$. Then we have
	\[
		\begin{cases}
			\ \,\Eff(\sX)=\left\langle3\Lambda-\pi^*H,\pi^*H \right\rangle  \\ 
			\Mov(\sX) =\left\langle \Lambda,\pi^*H\right\rangle             \\
			\ \Nef(\sX)=\left\langle \Lambda+\pi^*H,\pi^*H\right\rangle.
		\end{cases}
	\]
	In particular, the cones of divisors $\Eff(\sX)$, $\Mov(\sX)$ and $\Nef(\sX)$ are closed rational cones in $N^1(\sX)$.
	
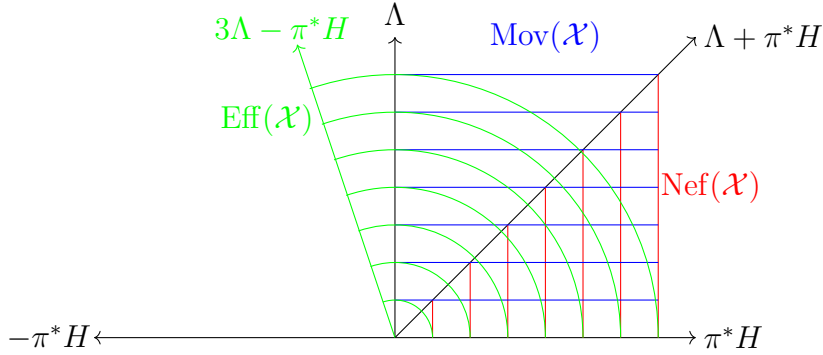
	\begin{center}
		\begin{tikzpicture}
			\draw[<->] (-4,0)--(4,0);
			\draw[->]  (0,0)--(0,4);
			
			\draw (4.5,0) node {$\pi^*H$};
			\draw (-4.6,0) node {$-\pi^*H$};
			\draw (0,4.3) node {$\Lambda$};
			
			\draw[->] (0,0)--(4,4);
			\draw (4.9,4) node {$\Lambda+\pi^*H$};
			\draw[red] (0.5,0)--(0.5,0.5);
			\draw[red] (1,0)--(1,1);
			\draw[red] (1.5,0)--(1.5,1.5);
			\draw[red] (2,0)--(2,2);
			\draw[red] (2.5,0)--(2.5,2.5);
			\draw[red] (3,0)--(3,3);
			\draw[red] (3.5,0)--(3.5,3.5);
			\draw[red] (4.2,2) node {$\Nef(\sX)$};
			
			\draw[blue] (0,0.5)--(3.5,0.5);
			\draw[blue] (0,1)--(3.5,1);
			\draw[blue] (0,1.5)--(3.5,1.5);
			\draw[blue] (0,2)--(3.5,2);
			\draw[blue] (0,2.5)--(3.5,2.5);
			\draw[blue] (0,3)--(3.5,3);
			\draw[blue] (0,3.5)--(3.5,3.5);
			\draw[blue] (2,4) node {$\Mov(\sX)$};

			\draw[green,->] (0,0)--(-1.3,3.9);
			\draw[green] (0.5,0) arc (0:108.9:0.5);
			\draw[green] (1,0) arc (0:108.9:1);
			\draw[green] (1.5,0) arc (0:108.9:1.5);
			\draw[green] (2,0) arc (0:108.9:2);
			\draw[green] (2.5,0) arc (0:108.9:2.5);
			\draw[green] (3,0) arc (0:108.9:3);
			\draw[green] (3.5,0) arc (0:108.9:3.5);
			\draw[green] (-1.5,4.1) node {$3\Lambda-\pi^*H$}; 
			\draw[green] (-1.7,2.9) node {$\Eff(\sX)$};
			\end{tikzpicture}
	\end{center}
\end{proposition}

\begin{proof}	
	The description of the effective cone $\Eff(\sX)$ of $\sX$ follows from Theorem \ref{Thm:Class-dual-VMRT} and \cite[Theorem 5.4]{HoeringLiuShao2020}.
	
	Note that $\check{\sC}$ is dominated by curves with $\Lambda$-degree $0$. It follows that $\check{\sC}\subset \bB_+(\Lambda)$. In particular, by Lemma \ref{Lemma:Base-Locus-DZD} that $[\Lambda]$ is not contained in the interior of $\Mov(\sX)$. Thus, it remains to show that $\Lambda$ is actually movable. Note that $X$ is quasi-homogeneous under the action of $\Aut(X)=\PGL_2(\mathbb{C})$ and there are exactly three orbits $X_0\sqcup X_1\sqcup X_2$, where $X_0$ is the open orbit and $X_i$ has codimension $i$ for $0\leq i\leq 2$. Moreover, the closure $\overline{X}_1=X_1\sqcup X_2$ of $X_1$ is a prime divisor in the complete linear system $|2H|$. In particular, the base locus of $|\Lambda|$ is contained in $\pi^{-1}(\overline{X}_1)$. Let $D\in |\Lambda|$ be an arbitrary element. If $\Lambda$ is not movable, then $\pi^*\overline{X}_1$ is contained in $\supp(D)$. In particular, $D-\pi^*\overline{X}_1$ is an effective divisor. This shows that $\Lambda-2\pi^*H$ is contained in $\Eff(\sX)$, which contradicts the description of $\Eff(\sX)$ above. Hence, $\Lambda$ is movable and we have $\Mov(\sX)=\left\langle\Lambda,\pi^*H\right\rangle$.
	
	Recall that there exists a one-dimensional family of lines $l\subseteq X$ on $X$ such that
	\[
	T_X|_l\cong \sO_{\PP^1}(2)\oplus \sO_{\PP^1}(1)\oplus \sO_{\PP^1}(-1).
	\]
	In particular, the nef cone $\Nef(\sX)$ of $\sX$ is contained in the cone $\left\langle \Lambda+\pi^*H,\pi^*H\right\rangle$. Thus, it remains to show that $\Lambda+\pi^*H$ is nef. Note that $X$ is embedded in $\PP^6$ by the complete linear system $|H|$. Therefore, thanks to \cite[Lemma 3.1]{HoeringLiuShao2020}, the vector bundle
	\[
	T_X\otimes \sO_{\PP^6}(3)|_X \otimes \sO_X(K_X)\cong T_X\otimes H
	\]
	is globally generated. Hence, the Cartier divisor class $\Lambda+\pi^*H$ is nef.
\end{proof}

\begin{remark}
	If $X$ is a Fano manifold of Picard number $1$ such that the VMRT $\sC_x\subset \PP(\Omega_{X,x})$ is $0$-dimensional, it is proved in \cite[Theorem 1.1]{HoeringLiu2021} that $T_X$ is big if and only if $X$ is isomorphic to the del Pezzo threefold $V_5$ of degree $5$. In particular, the normalised tangent bundle can not be pseudoeffective by Proposition \ref{Prop:Cones-V5}.
\end{remark}

Now we consider the horospherical $G_2$-variety $\mathbb{X}$. We briefly recall the geometric description of $\mathbb{X}$ and we refer the reader to \cite{Pasquier2009} for more details. Firstly $\mathbb{X}$ is a $7$-dimensional Fano manifold of Picard number $1$ and index $4$. The automorphism group ${\rm Aut}(\mathbb{X})$ acts on $\mathbb{X}$ with two orbits and the unique closed orbit $Z\subset \mathbb{X}$ is a smooth $5$-dimension quadric such that $H|_Z\cong \sO_{\mathbb{Q}^5}(1)$, where $H$ is the ample generator of $\pic(\mathbb{X})$. In particular, $\Lambda$ is movable. On the other hand, by \cite[Proposition 2.3]{PasquierPerrin2010}, it follows that there exists a deformation $\mathfrak{X}\rightarrow \Delta$ such that $\mathfrak{X}_t\cong B_3/P_2$ if $t\not=0$ and $\mathfrak{X}_0\cong \mathbb{X}$. Then the semi-continuity theorem implies that $T_{\mathbb{X}}$ is big. On the other hand, for $X=B_3/P_2$, the total dual VMRT $\check{\sC}'$ is a prime divisor such that
\[
[\check{\sC}'] \equiv 4\Lambda' - 2{\pi'}^*H',
\]
where $\Lambda'$ is the tautological divisor of $\pi':\PP(T_{B_3/P_2})\rightarrow B_3/P_2$ and $H'$ is the ample generator of $\pic(B_3/P_2)$ (see Proposition \ref{Prop:Pseff-Divisor}). Moreover, the VMRT of $\mathbb{X}$ at a general point is the smooth surface $\mathbb{P}(\sO_{\PP^1}(1)\oplus \sO_{\PP^1}(3))$ embedded by $\sO(1)$. This implies that the total dual VMRT $\check{\sC}$ of $\mathbb{X}$ is a prime divisor satisfying
\[
[\check{\sC}] \equiv 4\Lambda - 2\pi^*H.
\]
In particular, the class $[\Lambda]$ is not contained in the interior of $\overline{\Mov}(\PP(T_{\mathbb{X}}))$ since $\bB_{-}(\Lambda)$ contains $\check{\sC}$. This shows that $[\Lambda]$ generates an extremal ray of $\overline{\Mov}(\PP(T_{\mathbb{X} })$.

Next, denote by $\sN$ the normal bundle of $Z$ in $\mathbb{X}$. Then by adjunction formula, we have $\det(\sN)\cong \sO_{\Q^5}(-1)$. Denote by $\mathbb{X}'\rightarrow \mathbb{X}$ the blow-up along $Z$ with exceptional divisor $E$. Then it is known that $E$ is isomorphic to the complete flag manifold of $G_2$-type. In particular, the variety $E\cong \PP(\sN^*)$ is isomorphic to $\PP(\sV)$ over $\Q^5$, where $\sV$ is the Cayley bundle over $\Q^5$ (see \cite{Ottaviani1990}). As $\det(\sV)\cong \sO_{\Q^5}(-1)$, there is an isomorphism $\sN\cong \sV^*(-1)$. 

We claim that we have an isomorphism $\sV^*(-1)\cong \sV$. Indeed, it is clear that $\sV^*(-1)$ is stable as $\sV$ is stable (\cite{Ottaviani1990}). Moreover, an easy computation shows that we have 
\begin{center}
	$c_1(\sV^*(-1))=c_1(\sV)$\quad and \quad $c_2(\sV^*(-1))=c_2(\sV)$.
\end{center}
By \cite[Main Theorem]{Ottaviani1990}, the vector bundle $\sV^*(-1)$ is isomorphic to $\sV$. 

Finally, by \cite[Theorem 3.7]{Ottaviani1990}, the vector bundle $\sV(2)$ and hence $\sN(2)$ are globally generated. As a consequence, it follows from the tangent sequence of $Z$ that the restriction $T_{\mathbb{X}}(2)|_{Z}$ is nef. Moreover, note that $T_{\mathbb{X}}$ is globally generated outside $Z$, thus the vector bundle $T_{\mathbb{X}}(2)$ is nef. On the other hand, by \cite[Theorem 3.5]{Ottaviani1990}, there exist lines $l$ on $Z=\Q^5$ such that
\[
\sV|_l\cong \sO_{\PP^1}(-2) \oplus \sO_{\PP^1}(1).
\]
This implies that $T_{\mathbb{X}}(a)$ cannot be nef for any  $a<2$ and hence the divisor $\Lambda + a\pi^*H$ is nef if and only if $a\geq 2$. In summary, we have the following result.

\begin{proposition}
	Let $X$ be the horospherical $G_2$-variety $\mathbb{X}$, and let $H$ be the ample generator of $\pic(X)$. Denote by $\Lambda$ the tautological divisor class of the projectivized tangent bundle $\pi:\sX=\PP(T_X)\rightarrow X$. Then we have
	\[
	\begin{cases}
		\ \,\Eff(\sX)=\left\langle 4\Lambda-2\pi^*H,\pi^*H \right\rangle  \\ 
		\Mov(\sX) =\left\langle \Lambda,\pi^*H\right\rangle             \\
		\ \Nef(\sX)=\left\langle \Lambda+2\pi^*H,\pi^*H\right\rangle.
	\end{cases}
	\]
\end{proposition}

\begin{remark}
	\label{r.Examples-Big-tangent}
	Recently the second author has found in \cite{Liu2022} infinitely many non-homogeneous Fano manifolds of Picard number $1$ and with big tangent bundle.
\end{remark}

\appendix

\section{Big table for rational homogeneous spaces}
\label{Appendix}

In this appendix we summary the results for rational homogeneous spaces of Picard number one proved in Section \ref{Section:RHS} and provide more details about the invariants and geometric information of them. Let $X=G/P_k$ be a rational homogeneous space of Picard number $1$. Denote by $\textbf{P}(T^*_X)\xrightarrow{\varepsilon} \mathbf{P}(\widetilde{\sO})\rightarrow \textbf{P}(\overline{\sO})$ the Stein factorisation of the projectivized Springer map. Note that the variety $G/P$ is $1$-dimensional if and only if it is one of the following: $A_1/P_1$, $B_1/P_1$, $C_1/P_1$, $D_2/P_1$ and $D_2/P_2$. In particular, the variety $G/P$ is isomorphic to $\PP^1$. As the invariants for $\PP^1$ are trivial, in the table below we shall always assume that $G/P$ has dimension at least $2$.

The first column of the table below gives the type of the Lie group $G$. The second column is the numeration  of the corresponding node in the Dynkin diagram. The third column contains the type of $X$ (see Definition \ref{d.types} and Table \ref{Table:Types-RHS}). The fourth column and fifth column give the values of $a$ and $b$ in Theorem \ref{t.RationalHomSpace}, respectively. The column $6$ gives the type of singularities of $\textbf{P}(\widetilde{\sO})$ in codimension $2$ (cf. Definition \ref{d.types} and Corollary \ref{c.G/Ptypes}) and the notation "-" means that $\textbf{P}(\widetilde{\sO})$ is smooth in codimension $2$. The column $7$ describes the nilpotent orbit $\sO$ and the column $8$ gives the dual defect of the VMRT $\sC_o\subset \PP(\Omega_{X,o})$ of $X$ at a referenced point $o\in X$. The columns $9$ and $10$ contain the index and dimension of $X$, respectively, and the last two columns describe the VMRT $\sC_o$ and its embedding in $\PP(\Omega_{X,o})$, respectively.

The values of $a$ and $b$ are given in Section \ref{Section:RHS} according to the types of $X=G/P_k$. Let us summarise them as follows.
\begin{enumerate}
	\item If $X=G/P_k$ is of type (I) or type (II-s), the method to compute $a$ and $b$ is provided by Proposition \ref{Prop:Pseff-Small}. In particular, we always have  $b=1$ and the value of $a$ are provided in Table \ref{Table:Degree-of-lines}.
	
	\item If $X=G/P_k$ is of type (II-d-d), the method to compute $a$ and $b$ is to use Proposition \ref{Prop:Pseff-Divisor}(1). In particular, we again have $b=1$. The values of $a$ are explicitly determined in Table \ref{Table:a-II-d-d} for $G$ of classical type. The remaining cases for $G$ of exceptional type are $E_7/P_6$, $E_8/P_3$, $E_8/P_4$, $E_8/P_6$ and $F_4/P_4$. In these cases, the induced rational map $\eta:\textbf{P}(\widetilde{\sO})\dashrightarrow X$ is not explicit, so it prevents us to do the computation as that done for classical types in Proposition \ref{p.a(Gamma)classic}. However, the formula provided in Proposition \ref{Prop:Pseff-Divisor}(1) still works in these cases and we leave the calculation of the value $a$ in these five cases for the interested reader.
	
	\item If $X=G/P_k$ is of type (II-d-A1), the method to compute $a$ and $b$ is given by  Proposition \ref{Prop:Pseff-Divisor}(2) and (3). In particular, we have $b=2$ and $a$ is equal to the codegree of the VMRT $\sC_o\subset \PP(\Omega_{X,o})$ of $X$. Moreover, if the VMRT $\sC_o$ is a rational homogeneous space, then the codegree of $\sC_o$ can be found in \cite[p.39 Table 2.1 and p.40 Table 2.2]{Tevelev2005}. The remaining cases are $C_n/P_k$ $(3k\geq 2n+1)$ and $F_4/P_3$ (see Proposition \ref{Prop:Defect-short-root} and Table \ref{Table:Types-RHS}) and we prove the following two lemmas for them.
	
	Before giving the proof, let us briefly recall the basic definition and properties of nef value morphism. Given a polarised projective manifold $(X,H)$, if $K_X$ is not nef, the \emph{nef value} of $(X,H)$ is defined as
	\[
	\tau:={\rm min} \{t\in \R\,|\, K_X+tH\ \text{is nef}\}.
	\]
	The \emph{nef value morphism} of $(X,H)$ is the morphism $\Phi:X\rightarrow Y$ defined by the complete linear system $|m(K_X+\tau H)|$ for $m\gg 0$. If assume in addition that the complete linear system $|H|$ defines an embedding $X\subset \PP^N$, then the dual defect $\defect(X)$ can be determined by the nef value morphism $\Phi$. More precisely, by \cite{BeltramettiFaniaSommese1992} (see also \cite[Theorem 7.48 and Theorem 7.49]{Tevelev2005}), if $\defect(X)>0$, then the general fibre $F$ of $\Phi$ has Picard number $1$ and we have
	\[
	\defect(X)=\defect(F)-\dim(Y),
	\]
	where $\defect(F)$ is the dual defect of $F\subset \PP^d$ embedded by $|H|_F|$.

	\begin{lemma}
		\label{Lemma:C_n-VMRT}
		Let $X=C_n/P_k$ be a rational homogeneous space of type $C$ with $k\geq 2$, and let $\sC_o\subset \PP(\Omega_{X,o})$ be the VMRT at a referenced point $o\in X$. Then $\sC_o$ is isomorphic to the following projective bundle
		\[
		\pi:\PP(\sO_{\PP^{k-1}}(2)\oplus \sO_{\PP^{k-1}}(1)^{\oplus(2n-2k)})\rightarrow \PP^{k-1}.
		\]
		with embedding given by the complete linear system $|\sO(1)|$, where $\sO(1)$ is the tautological line bundle. Moreover, the following statements hold.
		\begin{itemize}
			\item[(i)] The VMRT $\sC_o\subset \PP(\Omega_{X,o})$ is dual defective if and only if $3k\leq 2n$, and if so then we have $\defect(\sC_o)=2n-3k+1$.
			
			\item[(ii)] If $3k\geq 2n+1$, then the dual variety of the VMRT $\sC_o\subset \PP(\Omega_{X,o})$ is a hypersurface of degree $2n-k$.
		\end{itemize}
	\end{lemma}
	
	\begin{proof}
		The description of the VMRT $\sC_o\subset \PP(\Omega_{X,o})$ follows from \cite{LandsbergManivel2003}. For the statement (i), by \cite[Theorem 7.21]{Tevelev2005}, if $2n\geq 3k$, then $\sC_o\subset \PP(\Omega_{X,o})$ is dual defective with $\defect(\sC_o)=2n-3k+1$. For the converse, we assume to the contrary that $3k\geq 2n+1$ and $\defect(\sC_o)>0$. Note that we have
		\[
		\sO_{\sC_o}(K_{\sC_o})\cong \sO(-(2n-2k+1))\otimes \pi^*\sO_{\PP^{k-1}}(2n-3k+1).
		\]
		Thus the nef value $\tau$ of $(\sC_o,\sO(1))$ is equal to $2n-2k+1$. Then the nef value morphism $\Phi$ is defined by the complete linear system $|\pi^*\sO_{\PP^{k-1}}(2n-3k+1)|$. In particular, either $\Phi$ is a map to a point (if $2n-3k+1=0$) or $\Phi$ is just the natural projection $\pi$ (if $2n-3k+1>0$). Let $F$ be a general fibre of $\Phi$. In the former case, we have $F=\sC_o$ and therefore $\rho(F)\geq 2$, which is a contradiction. In the latter case, the variety $F$ is isomorphic to $\PP^{2n-2k}$ and we have $\defect(F)=2n-2k$. In particular, we obtain
		\[
		\defect(F)-\dim(\PP^{k-1})=2n-3k+1\leq 0,
		\]
		which is again a contradiction. Hence, if $3k\geq 2n+1$, the VMRT $\sC_o\subset \PP(\Omega_{X,o})$ is not dual defective.
		
		For the statement (ii), as $3k\geq 2n+1$, the VMRT $\sC_o\subset \PP(\Omega_{X,o})$ is not dual defective. Thus we have $\codeg(\sC_o)=a$ by Proposition \ref{Prop:Pseff-Divisor} and the value of $a$ in this case is computed in Lemma \ref{Lemma:C_n-II-d-A} (see also \cite[Proposition 4.25]{Liu2022}).
	\end{proof}
	
	\begin{lemma}
		\label{Lemma:Codegree-VMRT-F4/P3}
		Let $X$ be the rational homogeneous space $F_4/P_3$, and let $\sC_o\subset \PP(\Omega_{X,o})$ be the VMRT of $X$ at a referenced point $o\in X$. Then $\sC_o$ is isomorphic to a smooth divisor in $|\sO_{\PP(\wedge^2 E)} (2)\otimes \pi^*\sO_{\PP^1}(-3)|$ and the embedding is given by the complete linear system $|\sO_{\PP(\wedge^2 E)}(1)|$, where $E$ is the vector bundle $\sO_{\PP^1}(1)^{\oplus 3}\oplus \sO_{\PP^1}$ and $\pi$ is the natural projection $\PP(\wedge^2 E)\rightarrow \PP^1$.
		
		In particular, the dual variety of $\sC_o\subset \PP(\Omega_{X,o})$ is a hypersurface of degree $8$.
	\end{lemma}
	
	\begin{proof}
		By \cite{HwangMok2004a}, the VMRT $\sC_o$ is isomorphic to the Grassmann bundle of $2$-planes in the dual bundle $E^*$ with embedding given by the complete linear system of Pl\"ucker bundle on $\sC_o$. Thus we have a natural embedding $\sC_o\subset \PP(\wedge^2 E)$ such that the restriction of $\sO_{\PP(\wedge^2 E)}(1)$ to $\sC_o$ is exactly the Pl\"ucker bundle. Moreover, note that the Grassmann variety $\Gr(2,4)\subset \PP^5$ defined by Pl\"ucker embedding is the quadric fourfold. Thus the $\sC_o$ is a smooth divisor in $\PP(\wedge^2 E)$ such that 
		\[
		\sC_o\in |\sO_{\PP(\wedge^2 E)}(2)\otimes \pi^*\sO_{\PP^1}(a)|
		\]
		for some $a\in \Z$. Let $S\subset \sC_o$ be the $\PP^2$-bundle corresponding to the quotient bundle $\wedge^2 E\rightarrow \sO_{\PP^1}^{\oplus 3}$. Then $S\cong \PP^1\times \PP^2$ and denote by $p_2: \PP^1\times \PP^2\rightarrow \PP^2$ the natural projection. Consider a rank $2$ subbundle $V=\sO_{\PP^1}(-1)\oplus \sO_{\PP^1}$ of $E^*$. Then $V$ defines a section $l=\PP(\wedge^2 V^*)\subset S$ of $\sC_o\rightarrow \PP^1$ such that $l$ is a fibre of $p_2$. Note that the normal bundle $N_1$ of $l$ in $\PP(\wedge^2 E)$ is isomorphic to the restriction of the relative tangent bundle of $\pi:\PP(\wedge^2 E)\rightarrow \PP^1$ to $l$. Thus one can easily derive from the relative Euler sequence of $\PP(\wedge^2 E)$ that we have
		\[
		N_1\cong \sO_{\PP^1}(-1)^{\oplus 3} \oplus \sO_{\PP^1}^{\oplus 2}.
		\]
		On the other hand, the normal bundle $N_2$ of $l$ in $\sC_o$ is isomorphic to the restriction of the relative tangent bundle of $\pi|_{\sC_o}:\sC_o\rightarrow \PP^1$ to $l$. Thus we have
		\[
		N_2\cong {\sH}om(V,E^*/V)\cong V^*\otimes (E^*/V) \cong \sO_{\PP^1}(-1)^{\oplus 2}\oplus \sO_{\PP^1}^{\oplus 2}.
		\]
		In particular, it follows that the restriction of the normal bundle of $\sC_o$ in $\PP(\wedge^2 E)$ to $l$ is isomorphic to $N_1/N_2\cong \sO_{\PP^1}(-1)$. This implies
		\[
		\sO_{\PP^1}(-1)\cong \sO_{\PP(\wedge^2 E)}(\sC_o)|_l\cong \sO_{\PP^1}(2+a).
		\]
		Hence, we have $a=-3$. Then one can easily obtain by adjunction formula that
		\[
		\sO_{\sC_o}(K_{\sC_o})\cong (\sO_{\PP(\wedge^2 E)}(-4)\otimes \pi^*\sO_{\PP^1}(4))|_{\sC_o}.
		\]
		In particular, the nef value of $(\sC_o,\sO_{\PP(\wedge^2 E)}(1)|_{\sC_o})$ is $4$ and the nef value morphism $\Phi$ is just the projection $\pi|_{\sC_o}:\sC_o\rightarrow \PP^1$. Let $F$ be a general fibre of $\pi|_{\sC_o}$. Then $F$ is isomorphic to the quadric fourfold $\Q^4$ and $\sO_{\PP(\wedge^2 E)}(1)|_F\cong \sO_{\Q^4}(1)$. In particular, we obtain
		\[
		\defect(F) - \dim(\PP^1) = -1 < 0. 
		\]
		Hence, the VMRT $\sC_o\subset \PP(\Omega_{X,o})$ is not dual defective. Then applying \cite[Theorem 6.2]{Tevelev2005} yields
		\[
		\codeg(\sC_o) = \sum_{i=0}^5 (i+1) c_{5-i}(\Omega_{\sC_o})\cdot \zeta^i,
		\]
		where $\zeta$ is the restriction of the tautological divisor of $\PP(\wedge^2 E)$ to $\sC_o$. Then a straightforward calculation shows that the Chern classes of $\Omega_{\sC_o}$ are as follows:
		\begin{center}
			$c_1=4F-4\zeta$,\quad $c_2=7\zeta^2-13\zeta F$,\quad $c_3=13\zeta^2 F - 6\zeta^3$,\\
			$c_4=3\zeta^4 - 6\zeta^3 F$,\quad $c_5=-6\zeta^4 F$.
		\end{center}
		Finally we conclude by the fact that $\zeta^5=15$ and $\zeta^4 F=2$.
	\end{proof}
	
	\item If $X=G/P_k$ is of type (II-d-A2), then $X$ is isomorphic to $E_7/P_4$ (cf. Table \ref{Table:Types-RHS}). The method to compute the values of $a$ and $b$ are provided in Proposition \ref{Prop:Pseff-Divisor}(2) and (3). In particular, we have $b=1$ and $a$ is equal to the codegree of the VMRT $\sC_o\subset \PP(\Omega_{X,o})$, which is the Segre embedding of $\PP^1\times \PP^2 \times \PP^3$. In particular, by \cite[p.39 Table 2.1]{Tevelev2005}, the codegree of $\sC_o$ is equal to $15$.
	
	\item If the VMRT $\sC_o\subset \PP(\Omega_{X,o})$ is homogeneous, then the dual defect of $\sC_o$ can be calculated by Proposition \ref{p.dualdectiveG/P} (see also \cite[p.39 Table 2.1 and p.40 Table 2.2]{Tevelev2005}). If the VMRT $\sC_o\subset \PP(\Omega_{X,o})$ is not homogeneous, then its dual defect is calculated in Proposition \ref{Prop:Defect-short-root}, Lemma \ref{Lemma:C_n-VMRT} and Lemma \ref{Lemma:Codegree-VMRT-F4/P3}. 
\end{enumerate}

\begin{landscape}
	\fontsize{6}{9.6}\selectfont
	%\tiny
	\renewcommand*{\arraystretch}{1.65}
	\begin{longtable}{|M{0.3cm}|M{2.1cm}|M{0.9cm}|M{1.9cm}|M{0.2cm}|M{0.2cm}|M{2.9cm}|M{2.3cm}|M{1.2cm}|M{1.5cm}|M{2.1cm}|M{1.1cm}|}
		\hline	
		
		$\mathfrak{g}$     &    node $k$           &
		type               &    $a$                &
		$b$                &    $A_i$              &
		Orbit $\mathcal{O}$           & dual defect 
		              &    index              &
		$\dim(X)$          &    VMRT               &
		embedding
		\\
		\hline
		
		\multirow{2}{*}{$A_n$}         &    $k\not=\frac{n+1}{2}$    &
		I                              &    $\min\{k,n-k+1\}$	     &
		$1$                            &          -                  &
		\multirow{2}{*}{$[2^k,1^{n+1-2k}]$     }                       &  
		\multirow{2}{*}{$|n-2k+1|$}                                  &  
		\multirow{2}{*}{$n+1$}                                 &
		\multirow{2}{*}{$k(n+1-k)$}                                    & 
		\multirow{2}{*}{$\PP^{k-1}\times \PP^{n-k}$}                 &
		\multirow{2}{*}{$\sO(1,1)$}
		\\
		\cline{2-6}
		
		                               &    $k=\frac{n+1}{2}$        &
		II-d-A1                        &    $k$             	     &
		$2$                            &    $A_1$                    & 
		                                                             &
		                                                             &    
		                                                             &
		                                                             & 
		                                                             &
		\\
		\hline

		\multirow{5}{*}{$B_n$}         &    $k\leq \frac{2n}{3}$         &
		II-d-A1                        &    $2k$                 	     &
		$2$                            &    $A_1$                        &   $[3^k, 1^{2n-3k+1}]$&
		\multirow{3}{*}{$\max\{0,3k-2n\}$}                                       &
		\multirow{3}{*}{$2n-k$}                                                  &
		\multirow{3}{*}{$\frac{k(4n-3k+1)}{2}$}                                  & 
		\multirow{3}{*}{$\PP^{k-1}\times \Q^{2(n-k)-1}$}                          &
		\multirow{3}{*}{$\sO(1,1)$}
		\\
		\cline{2-7}
		
		                               &$\frac{2n+1}{3}\leq k\leq n-1$ and $k$ odd     &
		II-d-d                         &    $2n-k+1$             	             &
		$1$                            &    $A_1$                                & $[3^{2n+1-2k},2^{3k-2n-1}]$
		                                                                         &
		                                                                         &    
		                                                                         &
		                                                                         & 
		                                                                         &
		\\
		\cline{2-7}

	                                   & $\frac{2n+1}{3}\leq k\leq n-1$ and $k$ even   
	    &	   II-s                    &      $2n-k$            	            
	    &      $1$                     &          -                              
	    & $[3^{2n+1-2k},2^{3k-2n-2},1^2]$
		&
		&    
		&
		& 
		&
		\\
		\cline{2-12}

		                               &    $k=n$ and $k$ odd            &
		II-d-A1                        &    $\frac{n+1}{2}$              &
		$2$                            &    $A_1$                        &  $[3,2^{n-1}]$                  &    $0$                          &
		\multirow{2}{*}{$2n$}                                            &
		\multirow{2}{*}{$\frac{n(n+1)}{2}$}                              & 
		\multirow{2}{*}{$\Gr(2,n+1)$}                                   &
		\multirow{2}{*}{$\sO(1)$}
		\\
		\cline{2-8}
		
		                               &    $k=n$ and $k$ even           &
		II-s                           &    $\frac{n}{2}$       	     &
		$1$                            &    -                            & $[3,2^{n-2},1^2]$              &	$2$                          & 
		&    
		&
		& 
		
		\\
		\hline
		
		\multirow{4}{*}{$C_n$}         &    $k=1$                        &
		II-s                           &    $1$            	             &
		$1$                            &    -                            & 
		$[2^2,1^{2n-4}]$               &	$2n-2$                       &    
		$2n$                           &	$2n-1$                       & 
		$\PP^{2n-2}$                                                     &
		$\sO(1)$
		\\
		\cline{2-12}
		
		                  &  $2\leq k\leq \frac{2n}{3}$ and $k$ odd      &
		II-s              &  $2k-2$                  	                 &
		$1$               &     -                                        &
		$[3^{k-1},2^2,1^{2n-3k-1}]$                                      &
		\multirow{3}{*}{$\max\{0,2n-3k+1\}$}                             &    
		\multirow{3}{*}{$2n-k+1$}                                        &
		\multirow{3}{*}{$\frac{k(4n-3k+1)}{2}$}                          & 
		\multirow{3}{*}{Lemma \ref{Lemma:C_n-VMRT}}                      &
		\multirow{3}{*}{$\sO(1)$}
		\\
		\cline{2-7}

		                      & $2\leq k\leq \frac{2n}{3}$ and $k$ even   &
		   II-d-d             & $2k$                       	              &
		   $1$                & $A_1$                                     &
		   $[3^k,1^{2n-3k}]$  &                                           &
		                      &       	                                  &     
		                      &
		\\
		\cline{2-7}

		                          &    $k\geq \frac{2n+1}{3}$             &
		II-d-A1                   &    $2n-k$                  	          &
		$2$                       &    $A_1$                              &
		$[3^{2n-2k},2^{3k-2n}]$   & 	                                  &    
		                          &                                       & 
	         	                  &
		\\
		\hline

		\multirow{6}{*}{$D_n$}    &     $k\leq \frac{2n-1}{3}$                   &
		II-d-A1                   &     $2k$                                     &
		$2$                       &     $A_1$                                    &
		$[3^k,1^{2n-3k}]$         &		\multirow{3}{*}{$\max\{0,3k-2n+1\}$}     &    
		\multirow{3}{*}{$2n-k-1$}                                                &
		\multirow{3}{*}{$\frac{k(4n-3k-1)}{2}$}                                  & 
		\multirow{3}{*}{$\PP^{k-1}\times \Q^{2(n-k)-2}$}                         &
		\multirow{3}{*}{$\sO(1,1)$}
		\\
		\cline{2-7}

	                               & $\frac{2n}{3}\leq k\leq n-2$ and $k$ odd    &
	    II-s                       &   $2n-k-1$                             	 &
	    $1$                        &          -                                  &
	    $[3^{2n-2k},2^{3k-2n-1},1^2]$  & 	                                     & 
		         	               &                                             & 
	 	                           &
		\\
		\cline{2-7}
		
		                           &  $\frac{2n}{3}\leq k\leq n-2$ and $k$ even  &
		II-d-d                     &  $2n-k$                                     &
		$1$                        &  $A_1$                                      &
		$[3^{2n-2k},2^{3k-2n}]$    &   	                                         &    
		                           &                                             & 
		                           &
		\\
		\cline{2-12}
		
	 	                           & $k\geq n-1$ and $n\geq 3$ odd               &
		I                          & $\frac{n-1}{2}$            	             &
		$1$                        &          -                                  &
		$[2^n]$                    & $2$                                         &    
		\multirow{2}{*}{$2n-2$}                                                  &
		\multirow{2}{*}{$\frac{n(n-1)}{2}$}                                      & 
		\multirow{2}{*}{$\Gr(2,n)$}                                              &
		\multirow{2}{*}{$\sO(1)$}
		\\
		\cline{2-8}

	 	                           &     $k\geq n-1$ and $n\geq 3$ even          &
		II-d-A1                    &     $\frac{n}{2}$                     	     &
		$2$                        &    $A_1$                                    &
		$[2^n]$                    &    $0$                                      &    
		                           &                                             & 
		                           &
		\\
		\hline
		
		\multirow{6}{*}{$E_6$}     &     $1$                                     &
				I                  &     $2$                                     &
		$1$                        &      -                                      &
		$2A_1$                     &	 $4$                                     &
		$12$                       &	$16$                                     & 
		$\mathbb{S}_5$             &	$\sO(1)$
		\\
		\cline{2-12}         
		
		                           &     $2$                                     &
		II-d-A1                    &     $4$                                     &
		$2$                        &     $A_1$                                   &
		$A_2$                      &	 $0$                                     &
		$11$                       &     $21$                                    & $\Gr(3,6)$                 &     $\sO(1)$
		\\
		\cline{2-12}

		                           &     $3$                                     &
		I                          &     $4$                                     &
		$1$                        &      -                                      &
		$A_2 + 2 A_1$              &     $1$                                     &
		$9$                        &     $25$                                    & 
		$\PP^1\times \Gr(2,5)$     &     $\sO(1)$
		\\
		\cline{2-12}
		
		                                   &     $4$                             &
		II-d-A1                            &     $12$                            &
		$2$                                &     $A_1$                           &
		$D_4(a_1)$                         &     $0$                             &
		$7$                                &     $29$                            &
		$\PP^1\times \PP^2\times \PP^2$    &   $\sO(1,1,1)$
		\\
		\cline{2-12}

		                                   &     $5$                             &
		I                                  &     $4$                             &
		$1$                                &      -                              &
		$A_2 + 2 A_1$                      &     $1$                             &
		$9$                                &     $25$                            & 
		$\PP^1\times \Gr(2,5)$             &     $\sO(1,1)$
		\\
		\cline{2-12}
		
		                                   &     $6$                             &
		I                                  &     $2$                             &
		$1$                                &      -                              &
		$2 A_1$                            &     $4$                             &
		$12$                               &     $16$                            &
		$\mathbb{S}_5$                     &     $\sO(1)$
		\\
		\hline

		\multirow{7}{*}{$E_7$}             &     $1$                             &
		II-d-A1                            &     $4$                             &
		$2$                                &     $A_1$                           &
		$A_2$                              &     $0$                             &    
		$17$                               &     $33$                            & 
		$\mathbb{S}_6$                     &     $\sO(1)$
		\\
		\cline{2-12}

		                                   &     $2$                             &
		II-d-A1                            &     $7$                             &
		$2$                                &     $A_1$                           &
		$A_2 + 3 A_1$                      &     $0$                             &
		$14$                               &     $42$                            &
		$\Gr(3,7)$                         &  	 $\sO(1)$
		\\
		\cline{2-12}  
		
		                                   &     $3$                             &
		II-d-A1                            &     $12$                            &
		$2$                                &     $A_1$                           &
		$D_4(a_1)$                         &     $0$                             &
		$11$                               &     $47$                            &
		$\PP^1\times \Gr(2,6)$             &     $\sO(1,1)$
		\\
		\cline{2-12}  
		
		                                   &     $4$                             &
		II-d-A2                            &     $15$                            &
		$1$                                &     $A_2$                           &
		$A_4 + A_2$                        &     $0$                             &
		$8$                                &     $53$                            &
		$\PP^1\times \PP^2\times \PP^3$    &     $\sO(1,1,1)$
		\\
		\cline{2-12}  
		
		                                   &     $5$                             &
		II-d-A1                            &     $12$                            &
		$2$                                &     $A_1$                           &
		$A_3 + A_2 + A_1$                  &     $0$                             &
		$10$                               &     $50$                            &
		$\PP^2\times \Gr(2,5)$             & 	 $\sO(1,1)$
		\\
		\cline{2-12}    
		
		                                   &     $6$                             &
		II-d-d                             &Proposition \ref{Prop:Pseff-Divisor} &
		$1$                                &     $A_1$                           &
		$2 A_2$                            & 	 $3$                             &    
		$13$                               &     $42$                            &
		$\PP^1\times \mathbb{S}_5$         & 	 $\sO(1,1)$
		\\
		\cline{2-12}  
		
		                                   &     $7$                             &
		II-d-A1                            &     $3$                             &
		$2$                                &     $A_1$                           &
		$(3 A_1)^{''}$                     &	 $0$                             &
		$18$                               &     $27$                            &
		$E_6/P_1$                           &     $\sO(1)$
		\\
		\hline

	    \multirow{8}{*}{$E_8$}             &     $1$                             &
		II-d-A1                            &     $8$                             &
		$2$                                &     $A_1$                           &
		$2 A_2$                            &     $0$                             &
		$23$                               &     $78$                            &
		$\mathbb{S}_7$                     &     $\sO(1)$
		\\
		\cline{2-12}
		
			                               &     $2$                             &
		II-d-A1                            &     $16$                            &
		$2$                                &     $A_1$                           &
		$D_4(a_1) + A_2$                   &     $0$                             &
		$17$                               &     $92$                            & 
		$\Gr(3,8)$                         & 	 $\sO(1)$
		\\
		\cline{2-12} 
		
			                               &     $3$                             &
		II-d-d                             &Proposition \ref{Prop:Pseff-Divisor} &
		$1$                                &     $A_1$                           &
		$A_4 + A_2 + A_1$                  &     $1$                             &
		$13$                               &     $98$                            &
		$\PP^1\times \Gr(2,7)$             &     $\sO(1,1)$
		\\
		\cline{2-12}  
		
			                               &     $4$                             &
		II-d-d                             &Proposition \ref{Prop:Pseff-Divisor} &
		$1$                                &     $A_1$                           &
		$A_6 + A_1$                        &     $1$                             &
		$9$                                &     $106$                           &
		$\PP^1\times \PP^2\times \PP^4$    &     $\sO(1,1,1)$
		\\
		\cline{2-12}
		
			                               &     $5$                             &
		II-d-A1                            &     $40$                            &
		$2$                                &     $A_1$                           &
		$E_8(a_7)$                         &     $0$                             &
		$11$                               &     $104$                           &
		$\PP^3\times \Gr(2,5)$             & 	 $\sO(1,1)$
		\\
		\cline{2-12} 
		
			                               &     $6$                             &
		II-d-d                             &Proposition \ref{Prop:Pseff-Divisor} &
		$1$                                &     $A_1$                           &
		$A_4 + A_2$                        &     $2$                             &
		$14$                               &     $97$                            &
		$\PP^2\times \mathbb{S}_5$         &     $\sO(1,1)$
		\\
		\cline{2-12} 
		
			                               &     $7$                             &
		II-d-A1                            &     $12$                            &
		$2$                                &     $A_1$                           &
		$D_4(a_1)$                         &     $0$                             &
		$19$                               &     $83$                            &
		$\PP^1\times E_6/P_1$              &	 $\sO(1,1)$
		\\
		\cline{2-12} 
		
			                               &     $8$                             &
		II-d-A1                            &     $4$                             &
		$2$                                &     $A_1$                           &
		$A_2$                              &     $0$                             &
		$29$                               &     $57$                            &
		$E_7/P_7$                           &	 $\sO(1)$
		\\
		\hline

		\multirow{4}{*}{$F_4$}             &     $1$                             &
		II-d-A1                            &     $4$                             &
		$2$                                &     $A_1$                           &
		$A_2$                              &     $0$                             &
		$8$                                &     $15$                            &
		$\LG(3,6)$                         &     $\sO(1)$
		\\
		\cline{2-12}

		                                   &     $2$                             &
		II-d-A1                            &     $12$                            &
		$2$                                &     $A_1$                           &
		$F_4(a_3)$                         &     $0$                             &
		$5$                                &     $20$                            &
		$\PP^1\times \PP^2$                &	 $\sO(1,2)$
		\\
		\cline{2-12}

		                                       &     $3$                         &
		II-d-A1                                &     $8$                       &
		$2$                                    &     $A_1$                       &
		$F_4(a_3)$                             &     $0$                         &
		$7$                                    &
		$20$                                   &     
		Lemma \ref{Lemma:Codegree-VMRT-F4/P3}  &
		$\sO(1)$
		\\
		\cline{2-12}                                                                                                                                                    
		                                     &     $4$                             &
		II-d-d                               &Proposition \ref{Prop:Pseff-Divisor} &
		$1$                                  &     $A_1$                           &
		$\widetilde{A}_2$                    &     $3$                             & 
		$11$                                 &     $15$                            &
		hyperplane section of $\mathbb{S}_5$ & 	   $\sO(1,1)$
		\\
		\hline   
		
		\multirow{2}{*}{$G_2$}               &     $1$                             &
		II-d-A1                              &     $2$                             &
		$2$                                  &     $A_1$                           &
		$G_2(a_1)$                           &     $0$                             &
		$5$                                  &     $5$                             &
		$\Q^3$                               & 	   $\sO(1)$
		\\
		\cline{2-12}   
		
		                                     &     $2$                             &
		II-d-A1                              &     $4$                             &
		$2$                                  &     $A_1$                           &
		$G_2(a_1)$                           &     $0$                             &
		$3$                                  &     $5$                             &
		$\PP^1$                              &	   $\sO(3)$
		\\
		\hline                  
	\end{longtable}
	
\end{landscape}

\bibliographystyle{alpha}
\bibliography{Math.bib}

\end{document}